\documentclass{amsart}
\usepackage[margin=1in]{geometry}
\usepackage[foot]{amsaddr}
\usepackage[leqno]{amsmath}
\usepackage{amssymb}
\usepackage[dvipsnames]{xcolor}
\usepackage{enumitem}
\usepackage{graphicx}
\usepackage{placeins}
\usepackage[caption=false]{subfig}

\usepackage{pifont}

\usepackage{accents}

\usepackage[hypertexnames=false,colorlinks=true,breaklinks=true,bookmarks=true,urlcolor=blue,citecolor=blue,linkcolor=blue,bookmarksopen=false,draft=false]{hyperref}

\pdfoutput=1 

\allowdisplaybreaks

\date{\small\today}

\vbadness=\maxdimen

\DeclareMathOperator{\Diag}{Diag}
\DeclareMathOperator{\diag}{diag}
\DeclareMathOperator{\ldet}{ldet}

\DeclareMathOperator{\tr}{tr}
\DeclareMathOperator{\rank}{rank}

\theoremstyle{plain}
\newtheorem{theorem}{Theorem} 
\newtheorem{lemma}[theorem]{Lemma}
\newtheorem{proposition}[theorem]{Proposition}
\newtheorem{corollary}[theorem]{Corollary}

\theoremstyle{remark}
\newtheorem{remark}[theorem]{Remark}
\newtheorem{example}[theorem]{Example}

\makeatletter
\renewcommand*{\top}{%
  {\mathpalette\@transpose{}}%
}
\newcommand*{\@transpose}[2]{%
  \raisebox{\depth}{$\m@th#1\scriptscriptstyle\mathsf{T}$}%
}
\makeatother

\newcommand{\zopthing}{z_{\mbox{\normalfont\protect\tiny GNLP-Id}}}
\newcommand{\zop}{\hyperlink{zoptarget}{\zopthing}}

\newcommand{\zopcompthing}{z_{\overline{\mbox{\normalfont\protect\tiny GNLP-Id}}}}
\newcommand{\zopcomp}{\hyperlink{zopcomptarget}{\zopcompthing}}

\newcommand{\zspectralthing}{z_{\mbox{\normalfont\protect\tiny $\mathcal{S}$}}}
\newcommand{\zspectral}{\hyperlink{zspectraltarget}{\zspectralthing}}

\newcommand{\zspectralthingL}{z_{\mbox{\normalfont\protect\tiny $\mathcal{LS}$}}}
\newcommand{\zspectralL}{\hyperlink{zspectraltargetL}{\zspectralthingL}}

\newcommand{\zgammathing}{z_{\mbox{\protect\tiny $\Gamma$}}}
\newcommand{\zgamma}{\hyperlink{zgammatarget}{\zgammathing}}

\newcommand{\zgscalegammathing}{z_{\mbox{\protect\tiny $\Gamma_{\Upsilon}$}}}

\newcommand{\zcmespthing}{z_{\mbox{\normalfont\protect\tiny CMESP}}}

\newcommand{\zcgmespthing}{z_{\mbox{\normalfont\protect\tiny CGMESP}}}
\newcommand{\zcgmesp}{\hyperlink{zcgmesptarget}{\zcgmespthing}}

\newcommand{\GMESP}{\hyperlink{GMESPtarget}{\GMESPsym}}
\newcommand{\GMESPsym}{
\mbox{\rm{GMESP}}}

\newcommand{\MESP}{\hyperlink{MESPtarget}{\MESPsym}}
\newcommand{\MESPsym}{
\mbox{\rm{MESP}}}

\newcommand{\zglinxthing}{z_{\mbox{\normalfont\protect\tiny glinx}}}
\newcommand{\zglinx}{\hyperlink{zglinxtarget}{\zglinxthing}}

\newcommand{\Pnst}{\hyperlink{PnstTarget}{\mathcal{P}(n,s,t)}}
\newcommand{\Pwedgenst}{\hyperlink{PwedgenstTarget}{\mathcal{P}^{\wedge}(n,s,t)}}
\newcommand{\Psubsetnst}{\hyperlink{PsubsetnstTarget}{\mathcal{P}^{\scriptscriptstyle\subset}(n,s,t)}}
\newcommand{\Unst}{\hyperlink{UnstTarget}{\mathcal{U}(n,s,t)}}
\newcommand{\Xnst}{\hyperlink{XnstTarget}{\mathcal{X}(n,s,t)}}

\newcommand{\zscaleglinxfraksoc}{z_{\mbox{\normalfont\protect\tiny glinx$_{\mbox{\tiny $\gamma$}}$}}^{\scriptscriptstyle \subset}
}

\newcommand{\zscalehatglinxfraksoc}{z_{\mbox{\normalfont\protect\tiny glinx$_{\mbox{\tiny $\gamma$}}$}}^{\scriptscriptstyle \subset}
}

\newcommand{\zscaleglinxfrakpsd}{z_{\mbox{\normalfont\protect\tiny glinx$_{\mbox{\tiny $\gamma$}}$}}^\wedge
}
\newcommand{\zglinxfrakpsd}{z_{\mbox{\normalfont\protect\tiny glinx}}^\wedge
}

\newcommand{\zbarglinxthing}{z_{\overline{\mbox{\normalfont\protect\tiny glinx}}}}

\newcommand{\zscaleglinxthing}{z_{\mbox{\normalfont\protect\tiny glinx$_{\mbox{\tiny $\gamma$}}$}}}
\newcommand{\zscaleglinx}{\hyperlink{zscaleglinxtarget}{\zscaleglinxthing}}

\newcommand{\fscaleglinxthing}{f^{\mbox{\normalfont\protect\tiny{sc}}}_{\mbox{\normalfont\protect\tiny glinx}}}
\newcommand{\fscaleglinx}{\hyperlink{fscaleglinxtarget}{\fscaleglinxthing}}

\newcommand{\fgscaleglinxthing}{f^{\mbox{\normalfont\protect\tiny{gsc}}}_{\mbox{\normalfont\protect\tiny glinx}}}
\newcommand{\fgscaleglinx}{\hyperlink{fgscaleglinxtarget}{\fgscaleglinxthing}}

\newcommand{\fgscalegfactthing}{f^{\mbox{\normalfont\protect\tiny{gsc}}}_{\mbox{\normalfont\protect\tiny $\Gamma$}}}
\newcommand{\fgscalegfact}{\hyperlink{fgscalegfacttarget}{\fgscalegfactthing}}

\newcommand{\zgscaleglinxthing}{z_{\mbox{\normalfont\protect\tiny glinx$_{\mbox{\tiny $\Upsilon$}}$}}}
\newcommand{\zgscaleglinx}{\hyperlink{zgscaleglinxtarget}{\zgscaleglinxthing}}

\newcommand{\zlinxthing}{z_{\mbox{\normalfont\protect\tiny linx}}}

\newcommand{\zbarlinxthing}{z_{\overline{\mbox{\normalfont\protect\tiny linx}}}}

\newcommand{\znlpidthing}{z_{\mbox{\normalfont\protect\tiny NLP-Id}}}

\newcommand{\zbarnlpidthing}{z_{\overline{\mbox{\normalfont\protect\tiny NLP-Id}}}}

\newcommand{\zgnlpidfrakSOC}{z_{\mbox{\normalfont\protect\tiny GNLP-Id}}^{\scriptscriptstyle \subset}
}

\newcommand{\zgnlpidcompfrakSOC}{z_{\overline{\mbox{\normalfont\protect\tiny GNLP-Id}}}^{\scriptscriptstyle \subset}
}

\begin{document}

\title[Extended-variable relaxations for CGMESP]{Extended-variable relaxations for the constrained\\ generalized maximum-entropy sampling problem}

\author{Gabriel Ponte}
\address{G. Ponte: University of Michigan, {\rm gabponte@umich.edu}, ORCID: 0000-0002-8878-6647}

\author{Kurt Anstreicher}
\address{K. Anstreicher: University of Iowa, {\rm kurt-anstreicher@uiowa.edu}, ORCID: 0000-0002-6306-9940}

\author{Marcia Fampa}
\address{M. Fampa: Universidade Federal do Rio de Janeiro, {\rm fampa@cos.ufrj.br}, ORCID: 0000-0002-6254-1510}

\author{Jon Lee}
\address{J. Lee: University of Michigan, {\rm jonxlee@umich.edu}, ORCID: 0000-0002-8190-1091}


\begin{abstract}
The \emph{constrained generalized maximum-entropy sampling problem} (CGMESP) is to select an order-$s$ principal submatrix from an order-$n$ covariance matrix, subject to some linear side constraints, so as to maximize the product of its $t$ greatest eigenvalues, $0<t\leq s <n$. GMESP refers to 
the version with no side constraints. 
Introduced more than 25 years ago,
CGMESP is a natural generalization of two fundamental problems in statistical design theory:
(i) constrained maximum-entropy sampling problem (CMESP);
(ii) binary D-optimality (D-Opt). 
In the general case, it can be
motivated by   
a selection problem in the context of principal component analysis (PCA).

We present novel non-convex extended variable formulations 
for CGMESP. Using these formulations as points of departure, we present, first non-convex and then convex,
continuous relaxations for CGMESP.
We demonstrate many relations between different upper bounds for CGMESP, 
including upper bounds from
the literature and our new upper bounds.
We investigate the behavior of our relaxations related to the constraints linking 
the natural variables with the extended variables.
We propose and investigate a generalized scaling technique for bound improvement. 
In the context of branch-and-bound, we determine the better of two natural branching techniques 
for fixing variables to zero. 
Finally, we present numerical experiments illustrating the
value of our methods.


\end{abstract}

\maketitle


\section{Introduction}

Let $C$ be a symmetric positive-semidefinite matrix with rows/columns
indexed from $N:=\{1,2,\ldots,n\}$, with $n >1$ and $\rank(C):=r$.
For integers $t$ and $s$, such that $0<t\leq r$ and $t\leq s<n$, 
let $T:= \{1,\dots,t\}$, $S(x)$ denote the support of $x\in\{0,1\}^n$,  $C_{S(x),S(x)}$ denote the principal submatrix indexed by $S(x)$,  $\mathbf{e}$ denote an all-ones vector, $\lambda_\ell(X)$ denote the $\ell$-th greatest eigenvalue of a symmetric matrix $X$, 
and $\ldet$ the natural logarithm of the determinant. 
For $m\geq 0$, we let $A\in\mathbb{R}^{m\times n}$ and $b\in\mathbb{R}^m$. 
We formally define
the \emph{constrained generalized maximum-entropy sampling problem} (see  \cite{WilliamsPhD,LeeLind2020}).
\begin{equation}\tag{CGMESP}\label{CGMESP}
\begin{array}{ll}
\hypertarget{zcgmesptarget}{\zcgmespthing}(C,s,t,A,b):=&\max \left\{\sum_{\ell \in T} \log( \lambda_\ell(C_{S(x),S(x)})) \,:\, \mathbf{e}^\top x =s,~ Ax\leq b,~ x\in\{0,1\}^n\right\},
\end{array}
\end{equation}
and the closely related \emph{maximum-entropy sampling problem} as
\begin{equation*}\tag{CMESP}\label{CMESP}
\hypertarget{zcmesptarget}{\zcmespthing}(C,s,A,b):= \max \left\{\ldet(C_{S(x),S(x)}) ~:~\mathbf{e}^\top x =s,~ Ax\leq b,~ x\in\{0,1\}^n\right\}.
\end{equation*}
Side constraints $Ax\leq b$ occur naturally in many applications.
But here and throughout, $\hypertarget{GMESPtarget}{\GMESPsym}$ and $\hypertarget{MESPtarget}{\MESPsym}$
refer to the specializations where 
there are no $Ax\leq b$ constraints (i.e., $m=0$). 

We can see GMESP as motivated by 
a selection problem in the context of PCA (principal component analysis); see, for example,  \cite{PCA} and the references therein, for the important topic of PCA.
Specifically, GMESP amounts to 
selecting a subvector of size $s$ from
 a Gaussian $n$-vector, so that geometric mean of the variances associated with the $t$ largest principal components is maximized.
 Linking this back to MESP, we
 can see that problem as selecting a subvector of size $s$ from
 a Gaussian $n$-vector, so that the geometric mean of the variances associated with \emph{all} principal components is maximized. 
 We use the geometric mean of the variances, so as to encourage a selection
where all $t$ of them are large and similar in value; we note that maximizing the geometric mean is equivalent to maximizing the log of the product.  
 
 Expanding on our motivation for GMESP, we assume that 
 we are in a setting where we have $n$ observable Gaussian random variables, with
 a possibly low-rank covariance matrix. We assume that observations are costly, and so we 
 want to select $s \ll n$  for observation. Even the $s$ selected random variables 
 may have a low-rank covariance matrix. 
 Posterior to the selection, we would then carry out
PCA on the associated order-$s$ covariance matrix,
with the aim of identifying the most informative $t < s$ latent/hidden random variables,
where we define \emph{most informative} as corresponding to maximizing the geometric mean of the variances of the 
$t$ dominant principal components. 

\medskip

\noindent {\bf Brief literature review.} 
\ref{CGMESP} was introduced by 
\cite{WilliamsPhD}, with that work eventually published as
\cite{LeeLind2020}. These works introduced the problem
as a common generalization of \ref{CMESP} and the closely related binary D-optimality (0/1 D-Opt) problem; see \cite{MESP2DOPT} for
recent developments on the relationship between these two problems.
\cite{WilliamsPhD,LeeLind2020} 
described
a branch-and-bound (B\&B) approach, and 
presented a ``Lagrangian spectral bound''
for \ref{CGMESP}. The problem had lain dormant for
a long time, while for the special cases 
of \ref{CGMESP} and 0/1 D-Opt there was a
lot of activity (see \cite{FLbook,FL_update,PonteFampaLeeMPB,MESP2DOPT,hugedoptmohit,PFLXadmm,li2025augmented}, and the many references in these works). Recently, 
\cite{SEA_proceedings,GMESP_Alg} introduced 
a ``generalized factorization bound'' for \ref{CGMESP}.
This is the first convex-optimization based relaxation for the problem.
They found that
their approach gives good upper bounds when $s-t$ is very small, and they left handling other cases as a significant challenge. 
\medskip 

\noindent {\bf Organization and contributions.} 
In Section \ref{sec:extform}, we present novel non-convex extended variable formulations of 
for \ref{CGMESP}. 
Using these formulations as points of departure, 
in Section \ref{sec:contrel}, we present, first non-convex and finally convex continuous relaxations for \ref{CGMESP};
convex
continuous relaxations yield efficiently-computable 
upper bounds for \ref{CGMESP}.
In Section \ref{sec:compare-bounds}, we present results demonstrating some relations between different upper bounds for \ref{CGMESP}, including bounds from
the literature and the new bounds proposed in this manuscript. 
In Section \ref{sec:value}, we investigate the behavior of 
our relaxations related to the constraints linking the natural variables ($x)$ with the extended variables.
In Section \ref{sec:gscale}, we propose and investigate a generalized scaling technique for bound improvement. 
In Section \ref{sec:BB}, in the context of B\&B, we determine the better of two natural branching techniques for fixing variables to zero. 
In Section \ref{sec:numexp}, we present numerical experiments illustrating the
value of our methods.
In Section \ref{sec:outlook}, we give some directions for further exploration.
Appendices \ref{app:a}--\ref{app:construct_dgfact} collect some material that would distract from the main flow.

\medskip

\noindent {\bf Notation.}
 Throughout, we denote any all-zero square matrix simply by $0$, while we denote any all-zero (column) vector by $\mathbf{0}$.
We denote any all-one vector
by $\mathbf{e}$, any \hbox{$i$-th} standard unit vector by $\mathbf{e}_i$\,,  
 and the identity matrix of order $n$ by $I_n$\,.
 We let $\mathbb{S}^n$  (resp., $\mathbb{S}^n_+$~, $\mathbb{S}^n_{++}$)
 denote the set of symmetric (resp., positive-semidefinite, positive-definite) matrices of order $n$. 
 We let $\Diag(x)$ denote the $n\times n$ diagonal matrix with diagonal elements given by the components of $x\in \mathbb{R}^n$, and we let $\diag(X)$ denote the $n$-vector with elements given by the diagonal elements of $X\in\mathbb{R}^{n\times n}$. 
  For matrix $X$,  $X_{R,S}$\, is the submatrix  with row (resp., column) indices $R$ ($S$);
 if $R$ (resp., $S$) is the entire row (column) index set, we write $X_{\cdot S}$ ($X_{R\cdot}$); we write $i$ for $\{i\}$. 
We let $\tr(\cdot)$ denote the trace. 
For matrices $X_1$ and $X_2$ having compatible shapes, we have the 
Hadamard (i.e., element-wise) product $X_1\circ X_2$,
and the Frobenius inner product  $X_1\bullet X_2:=\tr(X_1^\top X_2)$.
For $X\in\mathbb{S}^n$ and $i\in N$, we let $\lambda_i(X)$ denote its $i$-th greatest eigenvalue of $X$, and we write  $\lambda(X):=(\lambda_1(X),\lambda_2(X),\ldots,\allowbreak \lambda_n(X))^\top$. Specifically for $C \in \mathbb{S}^n_{+}$ (which plays a special role for us), we define $\lambda_{\max} := \lambda_{1}(C)$,  $\lambda_{\min} := \lambda_{n}(C)$, and  $\mu_{\max}$ ($\mu_{\min}$) as the multiplicity of $\lambda_{\max}$ ($\lambda_{\min}$) as an eigenvalue of $C$. 


\section{Extended-variable formulations of \ref{CGMESP}}\label{sec:extform}

In this section, we derive two alternative formulations for \ref{CGMESP} that will be useful in the development of upper-bounding methods, namely \ref{QCGMESP} and \ref{QcompCGMESP}. To provide intuition for their construction, we begin by introducing the  extended-variable formulation \ref{OCGMESP}, which is based on eigenvector orthogonality.  

Let $\hypertarget{UnstTarget}$
\begin{align*}
&\mathcal{U}(n,s,t) := \left\{ (x,U)\in \{0,1\}^n \times  \mathbb{R}^{n \times t} \!~:~  \mathbf{e}^{\top}x=s,~  
~U^\top U=I_t\,,~ \|U_{i\cdot}\|_2 \leq x_i\,,\,i\in N\right\},\nonumber
\end{align*}
and consider
\begin{equation}\label{OCGMESP}\tag{O-CGMESP}
    \max \left\{\ldet (U^\top C U)  \,:\, Ax\leq b,~ (x,U) \in \Unst\right\}. 
\end{equation}

\begin{theorem}\label{thm:OCGMESP}Let $\hat{x}$ be a feasible solution to {\rm\ref{CGMESP}} and let 
\begin{equation*} 
\hat{U}(\hat{x}):={\rm{argmax}}\left\{\ldet (U^\top C U)  \,:\, A\hat{x}\leq b,~ (\hat{x},U) \in \Unst\right\}.
\end{equation*}
For $\hat{U}\in\hat{U}(\hat{x})$, the objective value of $\hat{x}$ in  {\rm\ref{CGMESP}} is the same as the objective value of $(\hat{x},\hat{U})$ in
 {\rm\ref{OCGMESP}}.
\end{theorem}

\begin{proof}
It suffices to prove the result for any particular element of $\hat{U}(\hat{x})$ because all $\tilde{U}\in\hat{U}(\hat{x})$ lead to the same objective values of $(\hat x, \tilde U)$ in \ref{OCGMESP}. 

Consider a spectral decomposition  $\Diag(\hat x)C \Diag(\hat x):= \hat{V}\hat{\Lambda}\hat{V}^\top$, where $\hat{\Lambda}:=\Diag(\hat\lambda_1,\hat\lambda_2,\ldots,\hat\lambda_n)$, with $\hat\lambda_1\geq \hat\lambda_2\geq \ldots\geq \hat\lambda_n$\,, 
and $\hat{V}\hat{V}^\top=\hat{V}^\top\hat{V}=I_n$\,. 
Next, we prove the result for $\hat{U}:=\hat{V}_{\cdot T}$\,, 
and, subsequently confirm that $\hat{U} \in \hat{U}(\hat{x})$.
 
 Note that for $\hat{U}:=\hat{V}_{\cdot T}$\,, we have that $\hat{U} = \Diag(\hat x)\hat{U}$, because the rows of $\hat{U}$ corresponding  to the zero components of $\hat x$ are all zero. Otherwise, the identity $\Diag(\hat x)C \Diag(\hat x)= \hat{V}\hat{\Lambda}\hat{V}^\top$ would not hold,  as $\Diag(\hat x)C \Diag(\hat x)$ has zero rows and columns corresponding to the zero components of $\hat x$, and the columns of $\hat{U}$ are eigenvectors of $\Diag(\hat x)C \Diag(\hat x)$ that correspond to  positive eigenvalues. Therefore $(\hat{x},\hat{U})$ is a feasible solution for \ref{OCGMESP}. 

Moreover, we have
\begin{equation}\label{exac_rel}
  \hat{U}^\top C \hat{U}= \hat{U}^\top \Diag(\hat x)C \Diag(\hat x)\hat{U} =  \hat{U}^\top\hat{V}\hat{\Lambda}\hat{V}^\top \hat{U} = \hat{\Lambda}_{T,T}\,.
\end{equation}
Note that the positive eigenvalues of $C_{S(\hat x),S(\hat x)}$ and $ \Diag(\hat x)C \Diag(\hat x)$  are the same. From \eqref{exac_rel}, we see that their $t$ largest eigenvalues  are precisely the eigenvalues of   $\hat{U}^\top C \hat{U}$. Therefore, we have that $\ldet(\hat{U}^\top C \hat{U})=\textstyle\sum_{\ell\in T} \log(\lambda_\ell(C_{S(\hat x),S(\hat x)}))$, that is  $(\hat{x},\hat{U})$ has the same objective value as $\hat{x}$ has in {\rm\ref{CGMESP}}.

Finally, consider any matrix $U \in \mathbb{R}^{n\times t}$,  such that $(U,\hat{x})$ is a feasible solution to \ref{OCGMESP}, and note that $\ldet(U^\top C U) = \ldet({U}_{S(\hat x)\cdot}^\top C_{S(\hat x),S(\hat x)} {U}_{S(\hat x)\cdot})$. By the Poincaré separation theorem (see, e.g., \cite[Corollary 4.3.16]{HJBook}), we have
\begin{align*}
&\lambda_\ell(C_{S(\hat x),S(\hat x)}) \geq \lambda_\ell({U}_{S(\hat x)\cdot}^\top C_{S(\hat x),S(\hat x)} {U}_{S(\hat x)\cdot}) = \lambda_\ell({U}^\top C {U}) ,~~ \ell\in T\Rightarrow\\
&\textstyle\sum_{\ell\in T} \log(\lambda_\ell(C_{S(\hat x),S(\hat x)})) \geq \ldet(U^\top C U). 
\end{align*}
We conclude that $\hat{U}\in \hat{U}(\hat{x})$, and  the results follows. 
\end{proof}

 \begin{remark}\label{rem:ocgmesp} 
 We note that if $(\hat{x},\hat{U})$ is an optimal solution to \ref{OCGMESP}, then $\hat{U}\in \hat{U}(\hat{x})$, where $\hat{U}(\hat{x})$ is defined in Theorem \ref{thm:OCGMESP}. Therefore, by  Theorem \ref{thm:OCGMESP}, it follows that  \ref{OCGMESP} constitutes a valid formulation of \ref{CGMESP}, and $\zcgmesp(C,s,t,A,b)=\ldet(\hat{U}^\top C \hat{U})$.
 \end{remark}

Next, we develop the related extended-variable 
formulation \ref{QCGMESP}, which will ultimately lead  to
useful non-convex continuous relaxations of \ref{CGMESP}. 
Analogously to Theorem \ref{thm:OCGMESP}, Theorem \ref{thm:QCGMESP} establishes that \ref{QCGMESP} is a valid formulation of \ref{OCGMESP}, and, consequently, also of \ref{CGMESP}.

Let $\hypertarget{XnstTarget}$
\begin{align*}
&\mathcal{X}(n,s,t) := \left\{ (x,X)\in \{0,1\}^n\times \mathbb{S}^{n} \!~:~  \mathbf{e}^{\top}x=s,
~X\succeq 0,~X^2 = X,~\tr(X) = t,
~ \|X_{i\cdot}\|_2 \leq x_i\,,~i\in N\right\}, 
\end{align*}
and consider
\begin{equation}\label{QCGMESP}\tag{Q-CGMESP}
\textstyle    \max \left\{\sum_{\ell=1}^t \log(\lambda_\ell(C^{1/2}XC^{1/2}))  \,:\, Ax\leq b,~ (x,X) \in \Xnst\right\}.
\end{equation}

\begin{theorem}\label{thm:QCGMESP}\phantom{.}

\begin{enumerate}
\item[\rm($i$)]  Let $(\hat{x},\hat{X})$ be a feasible solution of {\rm\ref{QCGMESP}}. Then there exists $\hat{U}$ such that $(\hat{x},\hat{U})$ is a feasible solution of
 {\rm\ref{OCGMESP}} with the same objective value as $(\hat{x},\hat{X})$ has in {\rm\ref{QCGMESP}}.
\item[\rm($ii$)] Let $(\hat{x},\hat{U})$ be a feasible solution of
 {\rm\ref{OCGMESP}}. Then $(\hat{x},\hat{U}\hat{U}^\top)$ is a feasible solution of  {\rm\ref{QCGMESP}} with the same objective value as $(\hat{x},\hat{U})$ has in {\rm\ref{OCGMESP}}.
\end{enumerate}
\end{theorem}

\begin{proof}\phantom{.}

\noindent ($i$):~ Because $\hat{X}\succeq 0$ and $\hat{X}^2 = \hat{X}$, we have that $\lambda(\hat{X}) \in \{0,1\}^n$. Moreover, because $\tr(\hat{X}) = t$, we have that $\rank(\hat{X}) = t$. Let $\Phi\left(\begin{smallmatrix}
        I_t & 0\\
        0 & 0
    \end{smallmatrix}\right)\Phi^\top$ be a real Schur decomposition of $\hat{X}$ and define $\hat{U}:= \Phi_{\cdot T}$\,. Then, $\hat{U}\hat{U}^\top = \hat{X}$ and $\hat{U}^\top \hat{U} = I_t$\,. Moreover, for $i \in N$, 
        \[
        \|\hat{U}_{i\cdot}\|_2^2 = (\hat{U}\hat{U}^\top)_{ii} = \hat{X}_{ii} = \hat{X}^2_{ii} = \|\hat{X}_{i\cdot}\|_2^2\,.
        \]
        Then, we conclude that 
        \[
        \|\hat{U}_{i\cdot}\|_2 = \|\hat{X}_{i\cdot}\|_2 \leq \hat{x}_i\,.
        \]
    Therefore,  $(\hat{x},\hat{U})\in \Unst$.
    Finally,
    \[
    \textstyle \sum_{\ell=1}^t \log(\lambda_\ell(C^{1/2}\hat{X}C^{1/2})) = \sum_{\ell=1}^t \log(\lambda_\ell(C^{1/2}\hat{U}\hat{U}^\top C^{1/2}))= \sum_{\ell=1}^t \log(\lambda_\ell(\hat{U}^\top C\hat{U}))= \ldet(\hat{U}^\top C\hat{U}).  
    \]

 \noindent ($ii$):~ Let $\hat{X}:=\hat{U}\hat{U}^\top$. Then         \begin{itemize}
        \item $\hat{X}\succeq 0$.
        \item $\hat{X}^2 = \hat{U}\hat{U}^\top \hat{U}\hat{U}^\top =\hat{U}I_t\hat{U}^\top = \hat{X}$.
        \item $\tr(\hat{X}) = \tr(\hat{U}\hat{U}^\top) = \tr(\hat{U}^\top\hat{U}) = \tr(I_t) = t$.
        \item For $i \in N$, we have that 
        $
        \|\hat{X}_{i\cdot}\|_2^2 = (\hat{X}\hat{X}^\top)_{ii} = (\hat{X}^2)_{ii} = \hat{X}_{ii}\,,
        $
        and because $\hat{X}=\hat{U}\hat{U}^\top$, we have that $\hat{X}_{ii}=(\hat{U}\hat{U}^\top)_{ii}=\|\hat{U}_{i\cdot}\|_2^2$\,. Then, we conclude that 
        $
        \|\hat{X}_{i\cdot}\|_2= \|\hat{U}_{i\cdot}\|_2 \leq \hat{x}_i\,.
        $
    \end{itemize}
Therefore,  $(\hat{x},\hat{X})\in \Xnst$.
    Finally,
    \[
   \textstyle \ldet(\hat{U}^\top C\hat{U}) =\sum_{\ell=1}^t \log(\lambda_\ell(\hat{U}^\top C\hat{U})) = \sum_{\ell=1}^t \log(\lambda_\ell(C^{1/2}\hat{U}\hat{U}^\top C^{1/2})) =
    \textstyle \sum_{\ell=1}^t \log(\lambda_\ell(C^{1/2}\hat{X}C^{1/2})).  \eqno\qed
    \]
\renewcommand{\qedsymbol}{}
\end{proof}

Finally, we consider 
a closely-related
formulation which has the same constraints
as \ref{QCGMESP}, but a different objective function. We define  the \emph{companion formulation} of \ref{QCGMESP} as
\begin{equation}\label{QcompCGMESP} 
\tag{$\overline{\mbox{Q-CGMESP}}$}
\begin{array}{ll}
&\textstyle    \max \left\{\sum_{\ell=1}^{n-t} \log(\lambda_\ell (C^{-1/2}(I_n-X)C^{-1/2})) + \ldet(C) \,:\, Ax\leq b,~ (x,X) \in \Xnst\right\}.
\end{array}
\end{equation}
We recall (see e.g., \cite{FLbook}) the \emph{complementary formulation} of \ref{CMESP}, given by 
\begin{equation}\label{CMESPc}
\tag{$\overline{\mbox{CMESP}}$}
\max \left\{\ldet\left(C^{-1}_{S(x),S(x)}\right)+\ldet(C) ~:~\mathbf{e}^\top x =n-s,~ A(\mathbf{e}-x)\leq b,~ x\in\{0,1\}^n\right\}.
\end{equation}
Notice that \ref{CMESPc} is an instance of \ref{CMESP} with different data. Also, $\hat x$ is feasible for \ref{CMESP} if and only if $\mathbf{e}-\hat x$ is feasible for \ref{CMESPc}, and moreover the objective value of 
$\hat x$ in \ref{CMESP} is equal to the  objective value of 
$\mathbf{e}-\hat x$ in \ref{CMESPc}.
Despite that, the same bounding method applied to \ref{CMESP} and \ref{CMESPc} may produce different bounds. 

Next, Proposition \ref{cor:tsmesp} demonstrates  
that the pair of problems \ref{QCGMESP} and \ref{QcompCGMESP} degenerate to the strongly related pair of problems \ref{CMESP} and \ref{CMESPc}\,. Furthermore, in Theorem \ref{thm:QbarCGMESP} we establish that, as is the case for \ref{QCGMESP}, the formulation \ref{QcompCGMESP} is also a valid formulation for \ref{CGMESP}.

\begin{proposition}\label{cor:tsmesp}
    For $t = s$, {\rm\ref{QCGMESP}} is precisely  {\rm\ref{CMESP}}, 
    and {\rm\ref{QcompCGMESP}} is  precisely {\rm\ref{CMESPc}} after substituting $x$ by $\mathbf{e}-x$.
\end{proposition}

\begin{proof}
Let $(\hat x, \hat X)\in \mathcal{X}(n,s,t)$. From  $\hat X\succeq 0$ and $\|\hat X_{i\cdot}\|_2 \leq \hat x_i$\,, we have that $0\leq \hat X_{ii}\leq \hat x_i$\,, for $i\in N$. Moreover, from  $\mathbf{e}^\top \hat x = s$ and $\tr(\hat X) = s$, we have then that  $\hat X_{ii} = \hat x_i$ for $i \in N$. Therefore, $\hat X_{ij} = 0$ for $j \in N$ with $j\neq i$, that is, $\hat X=\Diag(\hat x)$. Both results directly follow.  
\end{proof}

\begin{theorem}\label{thm:QbarCGMESP}\phantom{.} 
Let $(\hat{x},\hat{X})$ be a feasible solution of {\rm\ref{QCGMESP}} and {\rm\ref{QcompCGMESP}}\,. The  
objective values of $(\hat{x},\hat{X})$ in both problems are the same.
\end{theorem}

\begin{proof}
    Because $(\hat x,\hat X)\in\Xnst$, we can easily verify that $I_n-\hat X\succeq 0$, $(I_n-\hat X)^2=I_n-\hat X$, $\tr(I_n-\hat X)=n-t$, and, therefore, $\rank(I_n-\hat X)=n-t$.
   Then, let $(\hat x,\hat U)\in \Unst$ 
   such that $\hat{U}\hat{U}^\top=\hat X$ (note that the existence of these elements is guaranteed by the proof of Theorem \ref{thm:QCGMESP}($i$)), and choose $\hat{V}\in\mathbb{R}^{n\times (n-t)}$ such that  $\hat{V}\hat{V}^\top=I_n-\hat X$ and $\hat{V}^\top \hat{V} = I_{n-t}$\,.

    Note that 
    \begin{align*}
     &\textstyle\sum_{\ell=1}^t \log(\lambda_\ell (C^{1/2}\hat{X}C^{1/2})) = \sum_{\ell=1}^t \log(\lambda_\ell (C^{1/2}\hat{U}\hat{U}^\top C^{1/2}))= \ldet(\hat{U}^\top C \hat{U}), \mbox{ and}\\
    &\textstyle\sum_{\ell=1}^{n-t} \log(\lambda_\ell (C^{-1/2}(I_n-\hat{X})C^{-1/2})) = \sum_{\ell=1}^{n-t} \log(\lambda_\ell (C^{-1/2}\hat{V}\hat{V}^\top C^{-1/2}))= \ldet(\hat{V}^\top C^{-1}\hat{V}).
    \end{align*}

Let $\Phi\Lambda\Phi^\top$ be any
 real Schur decomposition of $C$. Let 
 $
 \textstyle\tilde\Lambda := (1/\lambda_{\max})\Lambda$ and  $H_{n\times n}:= \Phi(I_n-\tilde\Lambda)^{\scriptscriptstyle 1/2}$. Then, we have
    \begin{align*} 
    \ldet({\hat U}^\top C {\hat U}) &=\ldet({\hat U}^\top (1/{\lambda_{\max}})C {\hat U}) +  t\log(\lambda_{\max}) \\
    &= \ldet({\hat U}^\top (I_n-HH^\top) {\hat U}) +  t\log(\lambda_{\max})\\
    &=\ldet(I_t -{\hat U}^\top HH^\top {\hat U}) + t\log(\lambda_{\max})\\
    &=\ldet(I_n -H^\top {\hat U}{\hat U}^\top H) + t\log(\lambda_{\max}) \\
    &=\ldet(H^\top H \!+\! {\tilde\Lambda} \!-\!H^\top \hat{U}\hat{U}^\top H) +t\log(\lambda_{\max})\\
    &= \ldet(H^\top (I_n\!-\!\hat{U}\hat{U}^\top) H \!+\! \tilde\Lambda)  +t\log(\lambda_{\max})\\
    &= \ldet((I_n-\tilde\Lambda)^{1/2}\Phi^\top \hat{V}\hat{V}^\top \Phi(I_n-{\tilde\Lambda})^{1/2} + \tilde\Lambda ) +t\log(\lambda_{\max})\\
    &= \ldet({\tilde\Lambda}^{-1/2}(I_n-\tilde\Lambda)^{1/2}\Phi^\top \hat{V}\hat{V}^\top \Phi(I_n-\tilde\Lambda)^{1/2}{\tilde\Lambda}^{-1/2} + I_n) + \ldet(\tilde\Lambda) +t\log(\lambda_{\max})\\
    &= \ldet(({\tilde\Lambda}^{-1}-I_n)^{1/2}\Phi^\top \hat{V}\hat{V}^\top\Phi({\tilde\Lambda}^{-1}-I_n)^{1/2}+ I_n) + \ldet(C) - n\log(\lambda_{\max}) +t\log(\lambda_{\max})\\
    &= \ldet(\hat{V}^\top \Phi({\tilde\Lambda}^{-1}-I_n) \Phi^\top \hat{V} + I_{n-t}) + \ldet(C) - n\log(\lambda_{\max})+t\log(\lambda_{\max})\\
    &= \ldet(\hat{V}^\top \Phi{\tilde\Lambda}^{-1}\Phi^\top\hat{V} -\hat{V}^\top\hat{V} + I_{n-t}) + \ldet(C) - n\log(\lambda_{\max})+t\log(\lambda_{\max})\\
    &=\ldet(\hat{V}^\top C^{-1} \hat{V}) + \ldet(C),
\end{align*}
where we use the following identities: $HH^\top=I_n - (1/{\lambda_{\max}})C$, $H^\top H=I_n - \tilde\Lambda$,  $\hat{U}^\top \hat{U} = I_t$\,, $\hat{V}^\top \hat{V} = I_{n-t}$\,,  and $\hat{V}\hat{V}^\top = I_n - \hat{U}\hat{U}^\top$\,. 
\end{proof}

Motivated by the successful use of upper-bounding methods for both \ref{CMESP} and \ref{CMESPc}, we will investigate upper-bounding methods for both \ref{QCGMESP} and \ref{QcompCGMESP}, to get upper-bounds for \ref{CGMESP}. Notice that because \ref{CMESPc} is an instance of \ref{CMESP} with different data, any upper-bounding method for one can be applied to the other. For \ref{CGMESP} our work is more involved; we will aim to generalize each bounding method for \ref{CMESP}
in two different ways: to handle each of \ref{QCGMESP} and \ref{QcompCGMESP}.


\section{Continuous relaxations of \ref{QCGMESP}}\label{sec:contrel} 

Letting$\hypertarget{PnstTarget}$
\begin{align*}
&\mathcal{P}(n,s,t) := \left\{ (x,X)\in \mathbb{R}^n\times \mathbb{S}^{n} \!~:~  x \in [0,1]^n,~\mathbf{e}^{\top}x=s, 
~0\preceq X\preceq \Diag(x)\,,~\tr(X) = t,~
~\|X_{i\cdot}\|_2 \leq x_i\,,~i \in N \right\},
\end{align*}
we obtain a convex relaxation of the set $\Xnst$:
\begin{proposition}\label{prop:calP}
    $\Xnst \subset \Pnst$.
\end{proposition}

\begin{proof}
Suppose that $(\hat x,\hat X)\in \Xnst$.
    Because $\hat X\succeq 0$ and $\hat X=\hat X^2$, then $\lambda(\hat{X}) \in \{0,1\}^n$, which implies $\hat{X} \preceq I_n$\,. Let $S:=S(\hat x)$. Then, we have that $\hat{X}_{S,S}\preceq I_S$\,. Finally, because $\|\hat X_{i\cdot}\|_2 =0$, for all $i\in N\setminus S$, $\hat X=\left(\begin{smallmatrix}
        \hat X_{S,S} & 0\\
        0 & 0
    \end{smallmatrix}\right) \preceq \Diag(\hat x)$. The result follows.
\end{proof}

Next, we see that when $t=s$, that is, in  
the special case of \ref{CMESP}, the set $\Pnst$
degenerates to a set that is equivalent to the
feasible region of the continuous relaxation of \ref{CMESP}.

\begin{proposition}\label{prop:P_teqs}
$\mathcal{P}(n,s,s)=\{(x,\Diag(x))\in \mathbb{R}^n\times \mathbb{S}^{n}  ~:~  x \in [0,1]^n,~\mathbf{e}^{\top}x=s\}$.
\end{proposition}

\begin{proof}
       From  $X\succeq 0$ and $\|X_{i\cdot}\|_2 \leq x_i$\,, we have that $0\leq X_{ii}\leq x_i$\,, for $i\in N$. Moreover, from  $\mathbf{e}^\top x = s$ and $\tr(X) = s$, we have then that  $X_{ii} = x_i$ for $i \in N$. Therefore, $X_{ij} = 0$ for $j \in N$ with $j\neq i$. The result follows.    
\end{proof}

We consider now non-convex continuous relaxations of 
the extended-variable formulations \ref{QCGMESP} and \ref{QcompCGMESP}\,, respectively, given by
\begin{equation*} 
\textstyle
    \max \left\{\sum_{\ell=1}^t \log(\lambda_\ell(C^{1/2}XC^{1/2}))  \,:\, Ax\leq b,~ (x,X) \in \Pnst\right\}
\end{equation*}
and
\begin{equation*}
\textstyle
    \max \left\{\sum_{\ell=1}^{n-t} \log(\lambda_\ell(C^{-1/2}(I_n-X)C^{-1/2})) + \ldet(C)  \,:\, Ax\leq b,~ (x,X) \in \Pnst\right\}.
\end{equation*}
We wish to emphasize that the non-convexity of these formulations is only in the non-concavity of the objective functions.  We think of these relaxations as preliminary
steps toward obtaining true convex-programming relaxations.
In the remainder of this section, we investigate how 
 the ``linx'' and ``NLP-Id'' relaxations for \ref{CMESP}, convex-programming relaxations from the literature, as well as variable-fixing strategies based on their Lagrangian duals, can be generalized for \ref{CGMESP}.  
Throughout, if we have a relaxation of a formulation, we call the relaxation
\emph{exact} if the objective value of every feasible solution for the formulation
is the same as its objective value for the relaxation. 


\subsection{Generalizing linx}\label{sec:linx}
Next, we define generalizations of the so-called ``linx bound'' for \ref{CMESP} (see \cite{Kurt_linx}).
First, we define these in their unscaled forms. Later, we introduce
a scaling parameter which is quite important in practice. 
For $C\in\mathbb{S}^n_{+}$\,, we define the \emph{generalized linx bound} as 
\begin{equation}\label{glinx}\tag{glinx}  
\begin{array}{ll}
\hypertarget{zglinxtarget}{\zglinxthing}(C,s,t,A,b) :=\max \left\{\textstyle \frac{1}{2}\ldet(CXC + I_n - X) \,:\, Ax\leq b,~ (x,X) \in \Pnst\right\},
\end{array}
\end{equation}
and, for $C\in\mathbb{S}^n_{++}$\,, we also define the \emph{companion generalized linx bound} as

\begin{align*}\label{gbarlinx} 
\tag{$\overline{\mbox{glinx}}$}  
&\hypertarget{zbarglinxtarget}{\zbarglinxthing}(C,s,t,A,b) :=\max \left\{\textstyle \frac{1}{2}\ldet(C^{-1}(I_n-X)C^{-1} + X)+\ldet(C) ~:~ \right.\\
&\qquad\qquad\qquad\qquad\qquad\qquad \left.Ax\leq b,~ (x,X) \in \Pnst\right\}.
\end{align*}

\begin{theorem} 
\phantom{.}
\begin{enumerate}
 \item[$(i)$] For $C\in\mathbb{S}^n_{+}$\,, {\rm\ref{glinx}} is an exact convex relaxation of {\rm\ref{QCGMESP}}.
 \item[$(ii)$] For $C\in\mathbb{S}^n_{++}$\,, {\rm\ref{gbarlinx}}  is an exact convex relaxation of {\rm\ref{QcompCGMESP}}.
\end{enumerate}
\end{theorem}

\begin{proof} \phantom{.}

\noindent ($i$):~ It is clear that the objective function of \ref{glinx} is concave.
    Moreover, from Proposition \ref{prop:calP}, we have that the feasible region of \ref{QCGMESP} is contained in the feasible region of \ref{glinx}. Now, suppose that $(\hat x,\hat X)$ is a feasible solution to \ref{QCGMESP}.
   Let $(\hat x,\hat U)\in \Unst$  
   such that $\hat{U}\hat{U}^\top=\hat X$ (the existence of these elements is guaranteed by the proof of Theorem \ref{thm:QCGMESP}($i$)), and choose $\hat{V}\in\mathbb{R}^{n\times (n-t)}$ such that   $Q :=\begin{pmatrix}
        \hat{U} \!\!&\!\! \hat{V}
    \end{pmatrix}$ 
    satisfies $Q^\top Q=I_n$\,.
    Noting that $\hat{V}\hat{V}^\top=I_n-\hat{X}$, we have that 
    \begin{align*}
        \ldet(C\hat{X}C + I_n -\hat{X})&= \ldet(Q^\top(C\hat{X}C + I_n -\hat{X})Q)\\
        &=\ldet(Q^\top C\hat{X}C Q + Q^\top (I_n -\hat{X})Q)\\
        &=\ldet(Q^\top C\hat{U}\hat{U}^\top C Q + Q^\top \hat{V}\hat{V}^\top Q).
    \end{align*}
    Let 
    \[M:=
    \begin{pmatrix}
            \hat{U}^\top C \hat{U} & 0\\
            \hat{V}^\top C \hat{U} &  I_{n-t}
        \end{pmatrix}.
    \]
    Then, we have
    \begin{align*}
        Q^\top C\hat{U}\hat{U}^\top C Q + Q^\top \hat{V}\hat{V}^\top Q &=\begin{pmatrix}
            (\hat{U}^\top C \hat{U})(\hat{U}^\top C \hat{U}) & (\hat{U}^\top C \hat{U})(\hat{U}^\top C \hat{V})\\
            (\hat{V}^\top C \hat{U})(\hat{U}^\top C \hat{U})  & (\hat{V}^\top C \hat{U})(\hat{U}^\top C \hat{V}) + I_{n-t}
         \end{pmatrix} = MM^\top.
    \end{align*}
    As $\ldet(M) = \ldet(\hat{U}^\top C\hat{U})$,  we conclude that 
\[
\textstyle \ldet(C\hat{X}C + I_n -\hat{X}) = 2\ldet(\hat{U}^\top C \hat{U})
=2\sum_{\ell=1}^t \log(\lambda_\ell(C^{1/2}\hat{U}\hat{U}^\top C^{1/2}))= 2\textstyle\sum_{\ell=1}^t \log(\lambda_\ell(C^{1/2}\hat{X}C^{1/2})).
\]
The result follows.

\noindent ($ii$): ~Similar to item $(i)$.
\end{proof}

Next, we will demonstrate that the {\rm\ref{glinx}}
and the {\rm\ref{gbarlinx}}
bounds generalize the linx and complementary 
linx bounds for \ref{CMESP} proposed in  \cite{Kurt_linx},  given by  
\begin{align*}\label{linx}\tag{linx}
\textstyle
& \hypertarget{zlinxtarget}{\zlinxthing}(C,s,A,b):=
\max \left\{\textstyle \frac{1}{2}(\ldet\left( C\Diag(x) C + \Diag(\mathbf{e}\!-\!x) \right)) ~:~ Ax\leq b,~ \mathbf{e}^\top x=s,~
x\in[0,1]^n
\right\}, 
\end{align*}

\begin{align*}\label{complinx}
\tag{$\overline{\mbox{linx}}$}
\textstyle
& \hypertarget{zbarlinxtarget}{\zbarlinxthing}(C,s,A,b):=
\max \left\{\textstyle \frac{1}{2}(\ldet\left( C^{-1}\Diag(x) C^{-1} + \Diag(\mathbf{e}\!-\!x) \right)) +\ldet(C) ~:~ \right.\\
& \qquad\qquad\qquad\qquad\qquad\qquad \left. A(\mathbf{e}-x)\leq b,~ \mathbf{e}^\top x=n-s,~
x\in[0,1]^n
\right\}.
\end{align*}

\begin{corollary} 
    For $t = s$, the {\rm\ref{glinx}} bound is precisely the {\rm\ref{linx}} bound for {\rm\ref{CMESP}}
    and the {\rm\ref{gbarlinx}} bound is precisely the {\rm\ref{complinx}} bound for {\rm\ref{CMESP}}.
\end{corollary}

\begin{proof}
     The result follows from Proposition \ref{prop:P_teqs}.
\end{proof}

\cite{Kurt_linx} demonstrated that $\zlinxthing(C,s,A,b)=\zbarlinxthing(C,s,A,b)$, and therefore complementing does not lead to different bounds (see also \cite[Proposition 3.3.16]{FLbook}). Next, we establish a similar result for \ref{glinx} and \ref{gbarlinx}.

\begin{proposition}\label{thm:compglinf}\phantom{.} 
Let $(\hat{x},\hat{X})$ be a feasible solution of {\rm\ref{glinx}} and {\rm\ref{gbarlinx}}. The  
objective values of $(\hat{x},\hat{X})$ in both problems are the same.
\end{proposition}

\begin{proof}
    It suffices to note that for any feasible solution  $(\hat x,\hat X)$ for \ref{glinx},  we have that  
    \begin{align*}
        \textstyle\frac{1}{2}\ldet(C\hat{X}C + I_n -\hat{X})&=\textstyle\frac{1}{2}\ldet(C^{-1}(I_n -\hat{X})C^{-1} + \hat{X}) + \ldet(C).
    \end{align*}
    Therefore the optimal value of \ref{glinx} is the same as the optimal value of \ref{gbarlinx}. 
\end{proof}

In what follows, we formulate the Lagrangian dual of \ref{glinx}.
Then, using standard methodology (see e.g., \cite{yamagishi2026dualpathfixingstrategyapplication}),
we indicate how to  fix variables for \ref{CGMESP} based on  knowledge of a lower bound for \ref{CGMESP} and feasible solutions for \ref{eq:dual_glinx}. 

\begin{theorem}\label{thm:dual-glinx}
    The Lagrangian dual problem of {\rm\ref{glinx}} is
\begin{equation}\label{eq:dual_glinx}\tag{Dglinx}
\begin{array}{rrl}
&\min &-\frac{1}{2}\ldet(2\Theta) + \tr(\Theta) + \nu^\top \mathbf{e} + \pi^\top b + \tau s + \xi t - \frac{n}{2}\\
&\text{\rm s.t.} 
&\upsilon - \nu - A^\top \pi - \tau\mathbf{e} + \eta + \diag(Z)= 0;\\[2.5pt]
&&\frac{1}{2}(W+W^\top) - \Omega +Z - C\Theta C + \Theta + \xi I_n = 0;\\[2.5pt]
&&\|W_{i\cdot}\|_2\leq \eta_i\,,~ \forall~ i\in N;\\[2.5pt]
&&\Theta \succ 0,\upsilon \geq 0, \nu \geq 0, \pi\geq 0, \eta \geq 0, Z \succeq 0, \Omega \succeq 0;\\
&&\tau \in \mathbb{R}, \xi \in \mathbb{R},W \in \mathbb{R}^{n\times n}.
\end{array}
\end{equation}
\end{theorem}

\begin{proof}
    See Appendix \ref{app:proof_dual_glinx}.
\end{proof}

We note that a similar Lagrangian-dual construction can be derived for \ref{gbarlinx}; however, it is unnecessary due to Proposition \ref{thm:compglinf}.

\begin{theorem} 
 Let 
\begin{itemize}
    \item\!${\rm{LB}}$ be the objective-function value of a feasible solution for {\rm\ref{CGMESP}},
    \item $(\hat\Theta,\hat\upsilon,\hat\nu,\hat\pi,\hat\tau,\hat\xi,\hat Z,\hat \Omega,\hat\eta)$ be a feasible solution for {\rm\ref{eq:dual_glinx}} 
    with objective-function value $\hat{\zeta}$.
\end{itemize}
Then, for every optimal solution $x^*$ 
for {\rm\ref{CGMESP}}, we have:
\[
\begin{array}{ll}
x_j^*=0, ~ \forall\,j\in N \mbox{ such that } \hat{\zeta}-{\rm LB}<  \hat{\upsilon}_j\,,\\
x_j^*=1, ~ \forall\,j\in N \mbox{ such that } 
\hat{\zeta}-{\rm LB} < \hat{\nu}_j\,.
\end{array}
\]
\end{theorem}

It is well-known that the \ref{linx} bound can be improved with a scaling parameter
(see \cite{Kurt_linx}), so we develop a scaled version of our \ref{glinx} bound. 
Considering that $\zcgmesp(C,s,t,A,b) = \zcgmesp(\gamma C, s, t,A,b) - t\log(\gamma)$ for $\gamma > 0$, we define
\[
\hypertarget{fscaleglinxtarget}{\fscaleglinxthing}
(X;\gamma):=\textstyle\frac{1}{2}\left(\ldet ( L(X;\gamma) ) - t\log(\gamma)\right),
\]
where $L(X;\gamma) := \gamma CXC + I_n - X$. 

We note that for every $\gamma>0$ and $(\hat x, \hat X) \in \Xnst$,
$\fscaleglinxthing
(\hat X;\gamma)$ is equal to the objective value of $\hat X$ in \ref{glinx}.
Therefore, we have 
the \emph{scaled glinx bound} 
\begin{equation}\label{scaleglinx}\tag{glinx$_\gamma$}  
\begin{array}{ll}
\hypertarget{zscaleglinxtarget}{\zscaleglinxthing}(C,s,t,A,b) :=\max \left\{\fscaleglinx(X;\gamma) \,:\, Ax\leq b,~ (x,X) \in \Pnst\right\}.
\end{array}
\end{equation}
As we have for \ref{linx}, the \ref{glinx} bound is very sensitive to the choice of scaling factor. Therefore a good selection of $\gamma$ is very important in order to obtain a good bound. 
In \cite{Mixing}, it was shown that the objective function of \ref{linx} is convex in the logarithm of the scaling parameter $\gamma$. This result was used to propose a BFGS-based algorithm for optimizing the choice of $\gamma$. In the following, we extend this analysis to \ref{glinx} and discuss how to select its scaling parameter $\gamma$  so as to obtain the strongest possible bound.

\begin{remark}
    We do not present the  scaling for \ref{gbarlinx} because, by multiplying the matrix $C$ by $\gamma$, we can interpret the optimization over $(x,X) \in \Pnst$ as simply solving \ref{glinx} (or \ref{gbarlinx}) with a modified matrix $C \in \mathbb{S}^n_+$\,. Because Theorem \ref{thm:compglinf} demonstrates that \ref{glinx} and \ref{gbarlinx} coincide for every $(x,X) \in \Pnst$, it is straightforward to conclude that any scaling for \ref{gbarlinx} can be achieved via the reciprocal scaling for \ref{glinx}.
\end{remark} 

Let $\psi := \log(\gamma)$, and we formulate the problem 
\begin{equation}\label{min_h_psi_glinx}
    \min_{\psi}\{h(\psi)\},
\end{equation}
where 
$
h(\psi) := \fscaleglinx(X^*;\exp(\psi)),
$
and where $(x^*,X^*) := (x^*(\psi),X^*(\psi))$ is a maximizer of \ref{scaleglinx}\,, with $\gamma(=\exp(\psi))$ fixed.

\begin{theorem}\label{thm:glinx-convex-log-gamma}
The function $h(\psi)$ is convex in $\psi \in \mathbb{R}$.
\end{theorem}

\begin{proof}
$h(\psi)$ is the point-wise maximum of $\fscaleglinx(X;\exp(\psi))$ over feasible $(x,X)$ for \ref{scaleglinx}\,. That is, $h(\psi) = \max\{\fscaleglinx(\exp(\psi);x)\,:\, Ax\leq b, (x,X) \in \Pnst\}$. Then, to conclude that $h$ is convex, it suffices to demonstrate that 
 $\fscaleglinx(\exp(X;\psi))$ is convex in $\psi$, for all fixed $(x,X)$ in the feasible set of \ref{scaleglinx}\,, such that $CXC + I -X\succ 0$ (which holds if and only if $L(X;\exp(\psi)) \succ 0$ for all $\psi \in \mathbb{R}$).
We have that
\begin{align*}
    \textstyle\frac{\partial }{\partial \gamma} \fscaleglinx(X;\gamma)&=  \textstyle \tfrac{1}{2}L(  X;\gamma)^{-1}\bullet ( C X C ) - \tfrac{t}{2\gamma}=  \tfrac{1}{2\gamma} (L(X;\gamma)^{-1}\bullet ( X - I  ) + n - t)
\end{align*}
and
\begin{align*}
    \textstyle\frac{\partial^2 }{\partial \gamma^2} \fscaleglinx(X;\gamma)
&=\textstyle- \frac{1}{2\gamma^2} (L(X;\gamma)^{-1}\bullet ( X - I_n  ) + n - t)\\
&\quad\qquad \textstyle+ \textstyle\frac{1}{2\gamma} (I_n-X)\bullet(L(X;\gamma)^{-1}( CXC)L(X;\gamma)^{-1})\\
&= \textstyle\frac{t}{2\gamma^2}-\tfrac{1}{2}\tr(L(X;\gamma)^{-1} CXCL(X;\gamma)^{-1}CXC).
\end{align*}
Therefore
\[
\textstyle\frac{\partial }{\partial \psi} \fscaleglinx(X;\gamma)=  \textstyle\gamma\frac{\partial }{\partial \gamma} \fscaleglinx(X;\gamma) = \tfrac{1}{2}\left(L(X;\gamma)^{-1}\bullet ( X - I_n ) + n - t\right),
\]
and 
\begin{align*}
    \textstyle\frac{\partial^2 }{\partial \psi^2} \fscaleglinx(&X;\gamma)=  \textstyle\gamma\frac{\partial }{\partial \gamma} \fscaleglinx(X;\gamma)  + \gamma^2\frac{\partial^2 }{\partial \gamma^2} \fscaleglinx(X;\gamma)\nonumber\\
&= \tfrac{\gamma}{2}\tr(L(X;\gamma)^{-1}CXC) - \tfrac{\gamma^2}{2}\tr(L(X;\gamma)^{-1}CXCL(X;\gamma)^{-1}CXC). 
\end{align*}

Now, it remains to demonstrate that
$\frac{\partial^2 f}{\partial\psi^2}(\exp(\psi), X^*) \geq 0, \; \forall \psi.$ Let 
\[
M(\psi,X^*) := \exp(\psi) L(\exp(\psi),  X)^{\scriptscriptstyle -1/2}CXCL(\exp(\psi),  X)^{\scriptscriptstyle -1/2},
\]
and note that 
\begin{align*}
    \textstyle\frac{\partial^2 }{\partial \psi^2} \fscaleglinx(X;\exp(\psi)) &= \tfrac{1}{2}\tr(M(\psi,X^*)(I-M(\psi,X^*)))\\ 
    &= \tfrac{1}{2}\tr(M(\psi,X^*)^{\scriptscriptstyle 1/2}(I-M(\psi,X^*))M(\psi,X^*)^{\scriptscriptstyle 1/2}).
\end{align*}
From the definition of $\fscaleglinx(\exp(X;\psi))$, it is easy to see that $0 \preceq M(\psi,X^*) \preceq I_n$\,. Then, we conclude that $\frac{\partial^2 }{\partial \psi^2} \fscaleglinx(X;\exp(\psi)) \geq 0$.
\end{proof}

Considering Theorem \ref{thm:glinx-convex-log-gamma}, we can apply, for example, a quasi-Newton method to solve \eqref{min_h_psi_glinx}, see \cite[Section 4]{Mixing} and \cite[Section 3.3.5]{FLbook} for more details. 


\subsection{Generalizing NLP-Id} 
Next, we define generalizations of the so-called ``NLP-Id bound'' for \ref{CMESP} (see \cite{AFLW_Using}).
 For $C\in \mathbb{S}^n_+$\,, consider a factorization $I_n - (1/{\lambda_{\max}})C = HH^\top$,
with $H\in \mathbb{R}^{n\times p}$, for some $p$ satisfying $n-\mu_{\max}\le p \le n$. For $C\in \mathbb{S}^n_{++}$\,, consider a factorization $(1/{\lambda_{\min}})C -I_n= GG^\top$,
with $G\in \mathbb{R}^{n\times q}$, for some $q$ satisfying $n-\mu_{\min}\le q \le n$.
For $C\in\mathbb{S}^n_{+}$\,, we define the \emph{generalized NLP-Id bound} as 
\begin{equation}\label{GNLP-Id}\tag{GNLP-Id}  
\begin{array}{ll}
&\hypertarget{zoptarget}{\zopthing}(C,s,t,A,b;H) := \max \left\{\textstyle \ldet(I_p -H^\top X H) + t\log(\lambda_{\max}) ~:~ \right.\\
&\qquad\qquad\qquad\qquad\qquad\qquad  \left. Ax\leq b,\,(x,X) \in \Pnst\right\},
\end{array}
\end{equation}
and, for $C\in\mathbb{S}^n_{++}$\,,  we define the \emph{companion generalized NLP-Id bound} as
\begin{equation}\label{GNLP-Id-comp}
\tag{$\overline{\mbox{GNLP-Id}}$}
\begin{array}{ll}
&\hypertarget{zopcomptarget}{\zopcompthing}(C,s,t,A,b;G) :=\max \left\{\textstyle \ldet(I_q +G^\top X G) + t\log(\lambda_{\min})  ~:~ \right.\\
&\qquad\qquad\qquad\qquad\qquad\qquad  \left. Ax\leq b,\, (x,X) \in \Pnst\right\}.
\end{array}
\end{equation}

\begin{theorem} 
\phantom{.}
     \begin{enumerate}
         \item[\rm($i$)] For $C\in \mathbb{S}^n_+$\,, {\rm\ref{GNLP-Id}} is an exact convex relaxation of {\rm\ref{QCGMESP}}. 
         \item[\rm($ii$)]  For $C \in \mathbb{S}^n_{++}$\,, {\rm\ref{GNLP-Id-comp}} is an exact convex relaxation of {\rm\ref{QcompCGMESP}}. 
     \end{enumerate}
\end{theorem}

\begin{proof}
It is clear that the objective functions of \ref{GNLP-Id} and \ref{GNLP-Id-comp} are concave.  Moreover, from Proposition \ref{prop:calP}, we have that the feasible region of \ref{QCGMESP} (and of \ref{QcompCGMESP}) is contained in the feasible region of \ref{GNLP-Id} (and of \ref{GNLP-Id-comp}).

\medskip
\noindent ($i$):~
    Now, suppose that $(\hat x,\hat X)$ is a feasible solution to \ref{QCGMESP}.
     Let $(\hat x,\hat U)\in \Unst$
   such that $\hat{U}\hat{U}^\top=\hat X$ (the existence of these elements is guaranteed by the proof of Theorem \ref{thm:QCGMESP}($i$)).
   Then we have 
     \begin{align*} 
     \ldet(I_p -H^\top \hat{X} H) + t\log(\lambda_{\max}) &= \ldet(I_p -H^\top \hat{U}\hat{U}^\top H) + t\log(\lambda_{\max}) \\
    &= \ldet(I_t -\hat{U}^\top HH^\top \hat{U}) + t\log(\lambda_{\max})\\
    &=\ldet(\hat{U}^\top (I_n-HH^\top) \hat{U}) +  t\log(\lambda_{\max})\\
    &=\textstyle \ldet(\hat{U}^\top C \hat{U})=\sum_{\ell=1}^t \log(\lambda_\ell(C^{1/2}\hat{U}\hat{U}^\top C^{1/2}))=\textstyle\sum_{\ell=1}^t \log(\lambda_\ell(C^{1/2}\hat{X}C^{1/2})).
\end{align*}
The result in item $(i)$ follows.
\smallskip

\noindent ($ii$):~
Similarly to ($i$),  suppose that $(\hat x,\hat X)$ is a feasible solution to 
\ref{QcompCGMESP} with finite objective value. Let $(\hat x,\hat U)\in \Unst$  such that $\hat{U}\hat{U}^\top=\hat X$.
   Then we have 
     \begin{align*} 
     \ldet(I_q + G^\top \hat{X} G) + t\log(\lambda_{\min}) &= \ldet(\hat{U}^\top (I_n+GG^\top) \hat{U}) +  t\log(\lambda_{\min})\\
     &=\ldet(\hat{U}^\top C \hat{U})\\
&=\textstyle\sum_{\ell=1}^t \log(\lambda_\ell(C^{1/2}\hat{X}C^{1/2}))\\ &=\textstyle\sum_{\ell=1}^{n-t} \log(\lambda_\ell(C^{-1/2}(I-\hat{X})C^{-1/2}))  + \ldet(C),
\end{align*}
where the last equation follows from Theorem \ref{thm:QbarCGMESP}. 
\end{proof}

\begin{proposition}\label{prop:GNLP-Id-invariant}
\phantom{.}
\begin{enumerate}
    \item[\rm($i$)] $\zop(C,s,t,A,b;H)$ does not depend on the chosen  $H$.
    \item[\rm($ii$)] $\zopcomp(C,s,t,A,b;G)$ does not depend on the chosen  $G$.
\end{enumerate}
\end{proposition}

\begin{proof}
\phantom{.}

\noindent ($i$):~
    The objective value of $(\hat{x},\hat{X})$ in \ref{GNLP-Id} is given by 
    \begin{align*}
        \ldet(I_p - H^\top \hat{X} H) + t\log(\lambda_{\max})&= \ldet(I_n - \hat{X}^{\scriptscriptstyle 1/2} H  H^\top \hat{X}^{\scriptscriptstyle 1/2} ) + t\log(\lambda_{\max})\nonumber\\
        &= \ldet(I_n + \hat{X}^{\scriptscriptstyle 1/2}((1/\lambda_{\max}) C-I_n) \hat{X}^{\scriptscriptstyle 1/2}) + t\log(\lambda_{\max})\,, 
    \end{align*}
   which only depends on $C$.
   \smallskip
   
\noindent ($ii$):~
    The objective value of $(\hat{x},\hat{X})$ in \ref{GNLP-Id-comp} is given by 
    \[
        \ldet(I_q + G^\top \hat{X} G) + t\log(\lambda_{\min})=  \ldet(I_n - \hat{X}^{\scriptscriptstyle 1/2}(I_n-(1/\lambda_{\min}) C) \hat{X}^{\scriptscriptstyle 1/2}) + t\log(\lambda_{\min})\,, 
    \]
    which only depends on $C$.
\end{proof}

\begin{remark}
   In view of Proposition~\ref{prop:GNLP-Id-invariant}, and for ease of notation, we omit the arguments $H$ and $G$ in $\zop$ and $\zopcomp$\,, respectively, throughout the remainder of the paper whenever their specification is not required. 
\end{remark}

Next, we demonstrate that the {\rm\ref{GNLP-Id}}
and {\rm\ref{GNLP-Id-comp}}
bounds generalize the NLP-Id and complementary 
NLP-Id bounds for \ref{CMESP} proposed in  \cite{AFLW_Using},  given by 
\begin{equation}\label{NLP-Id}\tag{NLP-Id}  
\begin{array}{ll}
&\hypertarget{znlpidtarget}{\znlpidthing}(C,s,A,b) := \max \left\{\textstyle \ldet(I_n + \Diag(x)^{\scriptscriptstyle 1/2}((1/\lambda_{\max}) C-I_n)\Diag(x)^{\scriptscriptstyle 1/2}) + s\log(\lambda_{\max})  \,:\right.\\
& \qquad\qquad\qquad\qquad\qquad\qquad~~\left.Ax\leq b,\, \mathbf{e}^\top x=s,\, x\in [0,1]^n\right\},
\end{array}
\end{equation}
\begin{equation}\label{NLP-Id-comp}
\tag{$\overline{\mbox{NLP-Id}}$} 
\begin{array}{ll}
\textstyle
& \hypertarget{zbarnlpidtarget}{\zbarnlpidthing}(C,s,A,b):=
\max \left\{\textstyle \ldet(I_n + \Diag(x)^{\scriptscriptstyle 1/2}(\lambda_{\min} C^{-1}-I_n)\Diag(x)^{\scriptscriptstyle 1/2})+\ldet(C)\right.\\
& \qquad\qquad\qquad\qquad\qquad\left.  + (s-n)\log(\lambda_{\min})~:~  A(\mathbf{e}-x)\leq b,~ \mathbf{e}^\top x=n-s,~
x\in[0,1]^n
\right\}.
\end{array}
\end{equation}

\begin{corollary} 
    For $t = s$, the {\rm\ref{GNLP-Id}} bound is precisely the {\rm\ref{NLP-Id}} bound for {\rm\ref{CMESP}}
    and the {\rm\ref{GNLP-Id-comp}} bound is precisely the {\rm\ref{NLP-Id-comp}} bound for {\rm\ref{CMESP}}.
\end{corollary}

\begin{proof}
     It is easy to see that for the \ref{NLP-Id} case the result follows from Proposition \ref{prop:P_teqs}.       In the \ref{NLP-Id-comp} case, we first rewrite its formulation in the following equivalent form for clarity.
     \begin{align}\label{NLP-Id-comp-b} 
\textstyle
& 
\max \left\{\textstyle \ldet(I_n + \Diag(\mathbf{e}-x)^{\scriptscriptstyle 1/2}(\lambda_{\min} C^{-1}-I_n)\Diag(\mathbf{e}-x)^{\scriptscriptstyle 1/2})+\ldet(C) + (s-n)\log(\lambda_{\min})~:~ \right.\\
& \qquad\qquad\qquad\qquad\left. Ax\leq b,~ \mathbf{e}^\top x=s,~
x\in[0,1]^n
\right\}.\nonumber
\end{align}
     Now let $\hat{x}$ be a feasible solution for \eqref{NLP-Id-comp-b} and let $\Phi\Lambda\Phi^\top$ be any
 real Schur decomposition of $C$.  Then, we have
    \begin{align*}
        &\ldet(I_n + \Diag(\mathbf{e}-\hat{x})^{\scriptscriptstyle 1/2}(\lambda_{\min} C^{-1}-I_n)\Diag(\mathbf{e}-\hat{x})^{\scriptscriptstyle 1/2})+\ldet(C)+(s-n)\log(\lambda_{\min})\\
        &~= \ldet(I_n - (I_n - \lambda_{\min}\Lambda^{-1})^{\scriptscriptstyle 1/2}\Phi^\top\Diag(\mathbf{e}-\hat{x})\Phi(I_n - \lambda_{\min}\Lambda^{-1})^{\scriptscriptstyle 1/2})+\ldet((1/\lambda_{\min})\Lambda) + s\log(\lambda_{\min})\\
        &~= \ldet((1/\lambda_{\min})\Lambda - ((1/\lambda_{\min})\Lambda - I_n)^{\scriptscriptstyle1/2}\Phi^\top\Diag(\mathbf{e}-\hat{x})\Phi((1/\lambda_{\min})\Lambda - I_n)^{\scriptscriptstyle 1/2})+ s\log(\lambda_{\min})\\
        &~= \ldet((1/\lambda_{\min})\Lambda - ((1/\lambda_{\min})\Lambda - I_n) + ((1/\lambda_{\min})\Lambda - I_n)^{\scriptscriptstyle1/2}\Phi^\top\Diag(\hat{x})\Phi((1/\lambda_{\min})\Lambda - I_n)^{\scriptscriptstyle 1/2})+ s\log(\lambda_{\min})\\
        &~= \ldet(I_n + ((1/\lambda_{\min})\Lambda - I_n)^{\scriptscriptstyle1/2}\Phi^\top\Diag(\hat{x})\Phi((1/\lambda_{\min})\Lambda - I_n)^{\scriptscriptstyle 1/2})+ s\log(\lambda_{\min})\\
        &~= \ldet(I_n+\Diag(\hat{x})^{\scriptscriptstyle 1/2}((1/\lambda_{\min})C - I_n)\Diag(\hat{x})^{\scriptscriptstyle 1/2}) + s\log(\lambda_{\min})\\
        &~= \ldet(I_n+\Diag(\hat{x})^{\scriptscriptstyle 1/2}GG^\top\Diag(\hat{x})^{\scriptscriptstyle 1/2}) + s\log(\lambda_{\min})\\
        &~= \ldet(I_q+G^\top\Diag(\hat{x})G) + s\log(\lambda_{\min})\,.
    \end{align*}
     The result follows from Proposition \ref{prop:P_teqs}.
\end{proof}

\begin{remark}
We note that the construction of the convex relaxations \ref{GNLP-Id} and \ref{GNLP-Id-comp} of  \ref{CGMESP} relies on  specific scalings of the covariance matrix $C$. Consequently, in  contrast to \ref{glinx}, optimizing a scaling factor to strengthen the bound is not applicable to these relaxations. 
\end{remark}

Next, we formulate the Lagrangian dual of \ref{GNLP-Id} and \ref{GNLP-Id-comp},
and using standard methodology (see e.g., \cite{yamagishi2026dualpathfixingstrategyapplication}),
we indicate how to  fix variables for \ref{CGMESP} based on  knowledge of a lower bound for \ref{CGMESP} and feasible solutions for \ref{GNLP-Id} and \ref{GNLP-Id-comp}.

\begin{theorem}\label{thm:dual-GNLP-Id}
    The Lagrangian dual problem of {\rm\ref{GNLP-Id}} is
    \begin{equation}\label{eq:dual_op}\tag{DGNLP-Id}
\begin{array}{rrl}
&\min &-\ldet(\Theta) + \tr(\Theta)  + \nu^\top \mathbf{e} + \pi^\top b + \tau s + \xi t - p\\
&\text{\rm s.t.} 
&\upsilon - \nu - A^\top \pi - \tau\mathbf{e} + \eta + \diag(Z)= 0;\\[2.5pt]
&&\frac{1}{2}(W+W^\top) - \Omega +Z +H\Theta H^\top + \xi I_n = 0;\\[2.5pt]
&&\|W_{i\cdot}\|_2\leq \eta_i\,,~ \forall~ i\in N;\\[2.5pt]
&&\Theta \succ 0,\upsilon \geq 0, \nu \geq 0, \pi\geq 0, \eta \geq 0, Z \succeq 0, \Omega \succeq 0;\\
&&\tau \in \mathbb{R}, \xi \in \mathbb{R},W \in \mathbb{R}^{n\times n}.
\end{array}
\end{equation}
\end{theorem}

\begin{proof}
    See Appendix \ref{app:proof_dual_GNLP_Id}.
\end{proof}

\begin{theorem}\label{thm:dual-GNLP-Id-comp}
    The Lagrangian dual problem of {\rm\ref{GNLP-Id-comp}} is
    \begin{equation}\label{eq:dual_op_comp}
\tag{$\overline{\mbox{DGNLP-Id}}$}
\begin{array}{rrl}
&\min &-\ldet(\Theta) + \tr(\Theta)  + \nu^\top \mathbf{e} + \pi^\top b + \tau s + \xi t - q\\
&\text{\rm s.t.} 
&\upsilon - \nu - A^\top \pi - \tau\mathbf{e} + \eta + \diag(Z)= 0;\\[2.5pt]
&&\frac{1}{2}(W+W^\top) - \Omega +Z - G\Theta G^\top + \xi I_n = 0;\\[2.5pt]
&&\|W_{i\cdot}\|_2\leq \eta_i\,,~ \forall~ i\in N;\\[2.5pt]
&&\Theta \succ 0,\upsilon \geq 0, \nu \geq 0, \pi\geq 0, \eta \geq 0, Z \succeq 0, \Omega \succeq 0;\\
&&\tau \in \mathbb{R}, \xi \in \mathbb{R},W \in \mathbb{R}^{n\times n}.
\end{array}
\end{equation}
\end{theorem}

\begin{proof}
    See Appendix \ref{app:proof_dual_GNLP_Id_comp}.
\end{proof}

\begin{theorem} 
 Let 
\begin{itemize}
    \item\!${\rm{LB}}$ be the objective-function value of a feasible solution for {\rm\ref{CGMESP}},
    \item $(\hat\Theta,\hat\upsilon,\hat\nu,\hat\pi,\hat\tau,\hat\xi,\hat Z,\hat \Omega,\hat\eta)$ be a feasible solution for {\rm\ref{eq:dual_op}} or {\rm\ref{eq:dual_op_comp}} with objective-function value $\hat{\zeta}$.
\end{itemize}
Then, for every optimal solution $x^*$ 
for {\rm\ref{CGMESP}}, we have:
\[
\begin{array}{ll}
x_j^*=0, ~ \forall\,j\in N \mbox{ such that } \hat{\zeta}-\mbox{LB} < \hat{\upsilon}_j\thinspace,\\
x_j^*=1, ~ \forall\,j\in N \mbox{ such that } \hat{\zeta}-{\rm LB} < \hat{\nu}_j\thinspace.\\
\end{array}
\]
\end{theorem}


\section{Bound comparisons for GMESP}\label{sec:compare-bounds}

In this section, we  present the  existing upper-bounding methods for \ref{CGMESP} from the literature. We then provide a theoretical comparison, in the $\GMESP$ case, between these existing bounds and the new bounds proposed.


\subsection{\textbf{Spectral bound}} \cite{LeeLind2020} presents the
\emph{spectral bound}  for $\GMESP$, given by
\begin{equation*}
    \hypertarget{zspectraltarget}{\zspectralthing}(C,t) := \textstyle\sum_{\ell=1}^t \log(\lambda_\ell(C)).
\end{equation*} 
We note that the spectral bound is oblivious to $s$,
and so it should be particularly bad when $s$ is small. 
More generally,  
\cite{LeeLind2020} defines the \emph{(Lagrangian) spectral bound} for \ref{CGMESP} as 
\begin{equation}\label{eq:specbound_constrained}
    \hypertarget{zspectraltargetL}{\zspectralthingL}(C,s,t,A,b):= \min_{\pi\in \mathbb{R}^m_+} v(\pi) := 
\sum_{\ell=1}^t\log(\lambda_\ell\left( D^\pi C D^\pi \right)) + \pi^\top b ~-\! 
\min_{
\substack{ K\subset N,\\ |K|=s-t}
}~
\sum_{j\in K}\sum_{i=1}^m \pi_i A_{ij}\,,
\end{equation}
where $D^\pi\in \mathbb{S}^n_{++}$ is  the diagonal matrix defined by
\begin{equation}\label{def:Dpi-spec-bound}
D^\pi_{\ell,j}:= 
\left\{
  \begin{array}{ll}
    \exp\{-\frac{1}{2}\sum_{i=1}^m \pi_i A_{ij}\},& \hbox{ for $\ell=j$;} \\
    0,&\hbox{ for $\ell\not=j$.}
  \end{array}
\right.    
\end{equation}

\begin{remark}
We note that: (i) for \ref{CGMESP}, the spectral bound clearly cannot be
improved by scaling, (ii) for the special case of \ref{CMESP},
the spectral bound cannot be
improved by complementation (see [Prop. 3.1.2]\cite{FLbook}), and (iii) for general \ref{CGMESP},
the complementation principle does not apply at all. 
\end{remark}


\subsection{\textbf{Generalized factorization bound}}
Next, we define the ``generalized factorization bound''.
It is based on
 the following fundamental lemma. 

\begin{lemma}[\protect{\cite[Lemma 13]{Nikolov}}]\label{Ni13}
 Let $\omega\in\mathbb{R}_+^k$ satisfy $\omega_1\geq \omega_2\geq \cdots\geq \omega_k$\,, define $\omega_0:=+\infty$, and let $t$ be an integer satisfying
 $0<t\leq k$. Then there exists a unique integer $i$, with $0\leq i< t$, such that
 \begin{equation*} 
 \omega_{i }>\textstyle\frac{1}{t-i }\textstyle\sum_{\ell=i+1}^k \omega_{\ell}\geq \omega_{i+1}\,.
 \end{equation*}
\end{lemma}

Suppose that  $\omega\in\mathbb{R}^k_+$ with  
$\omega_1\geq\omega_2\geq\cdots\geq\omega_k$\,.
Let $\hat\imath$ be the unique integer defined by Lemma \ref{Ni13}. We define
\begin{equation*} 
\phi_t(\omega):=\textstyle\sum_{\ell=1}^{\hat\imath} \log(\omega_\ell) + (t - \hat\imath)\log\left(\frac{1}{t-{\hat\imath}} \sum_{\ell=\hat\imath+1}^{k}
\omega_\ell\right).
\end{equation*}
For $X\in\mathbb{S}_{+}^k$~, we define the \emph{$\Gamma$-function} as
\begin{equation*} 
\Gamma_t(X):= \phi_t(\lambda(X)),
\end{equation*}
For  $C\in\mathbb{S}^n_{+}$\,, we factorize $C=FF^\top$,
with $F\in \mathbb{R}^{n\times k}$, for some $k$ satisfying $r\le k \le n$. 
The \emph{generalized factorization bound} for \ref{CGMESP} from 
\cite{GMESP_Alg}  
is
\begin{align*}\label{ddgfact}\tag{\mbox{DDGFact}}
\hypertarget{zgammatarget}{\zgammathing} (C,s,t,A,b;F)\!:=
&\max \!\left\{ \Gamma_t(F^\top \Diag(x)F) \, : \, \mathbf{e}^\top x=s,\,Ax\leq b,\,
x \in [0,1]^n
\right\}, 
\end{align*}
and ${\zgammathing} (C,s,t,A,b):= {\zgammathing} (C,s,t,A,b; F)$,  where we omit the parameter $F$ in the simplified notation for $\zgammathing(C,s,t,A,b;F)$ because, as has been proven in \cite{GMESP_Alg} (also see \cite{FactPaper}), the factorization bound does not depend on which factorization of $C$ is used.

\begin{remark}\label{rem:gamma_invariant}
We note that: (i) for \ref{CGMESP}, the factorization bound cannot be
improved by scaling (see \cite{GMESP_Alg}), (ii) for the special case of \ref{CMESP},
the factorization bound can be
improved by complementation (see \cite{FactPaper}), 
and (iii) for general \ref{CGMESP},
the complementation principle does not apply at all. 
Expanding on (iii), we do not know 
how to give a factorization-type bound based on the companion formulation  \hbox{{\rm\ref{QcompCGMESP}}}.
\end{remark}


\subsection{\bf{Bound comparisons}}
Theoretical comparisons of upper bounds for experimental-design problems have been the subject of extensive study and are useful for understanding the behavior of different upper-bounding methods  (see, e.g., \cite{FactPaper,PonteFampaLeeMPB,MESP2DOPT,GMESP_Alg,li2025augmented}). Recently, \cite{GMESP_Alg} established a relation between the generalized factorization bound  and the spectral bound for $\GMESP$ which we present in the following theorem. 
\begin{theorem}[\cite{GMESP_Alg}]\label{thm:ddgfact-almost-dominates}
    Let $C\in\mathbb{S}^n_{+}$\,, with $r:=\rank(C)$, $0<t\leq r$, and $t\leq s < n$. Then,
we have
$
 \textstyle
\zgamma(C,s,t,\cdot,\cdot)
- \zspectral(C,t)\leq
t\log({s}/{t}).
$
\end{theorem}

 In the remainder of this section, we present new results relating upper bounds for $\GMESP$. Next, we summarize the relations that will be established. 

\begin{enumerate}
\item \emph{glinx bound dominance:} For $\gamma := 1/\lambda_{t}(C)^2$ (and hence also for the optimal $\gamma$), we have $\zscaleglinx(C,s,t,\cdot,\cdot)\allowbreak \leq \zspectral(C,t)$\,. 
 \item \emph{Spectral bound dominance:}
        If any of the following successively weaker conditions holds,   then $\zspectral(C,t) \leq  \zgamma(C,s,t,\cdot,\cdot)$\,.
    \begin{enumerate}
       \item
       $\mu_{\max} \geq nt/s$,
        \item
        $\tr(C) \geq t\lambda_{\max}$ and $\sum_{\ell = 1}^t \log\big(\frac{nt}{s}\frac{\lambda_{\ell}(C)}{\tr(C)}\big) \leq 0$,
        \item
        $\sum_{i=1}^t\log(\lambda_i(C)) - \Gamma_t(C) + t \log(n/s)\leq 0$.
    \end{enumerate}
    \item \emph{{\rm GNLP-Id} bound dominance:}
    \begin{enumerate}
        \item $\zop(C,s,t,\cdot,\cdot)-\zspectral(C,t) \leq   \sum_{i=1}^t \log(\lambda_{\max}/\lambda_i(C))$,
        \item if $\mu_{\max} \geq t$\,, then   $\zop(C,s,t,\cdot,\cdot) \leq \zspectral(C,t)$\,.\\[-5pt]
    \end{enumerate}

    \item \emph{{\rm GNLP-Id-comp} bound dominance:} For $C \in \mathbb{S}^n_{++}$\,, we have
    \begin{enumerate}
        \item $\zopcomp(C,s,t,\cdot,\cdot)-\zspectral(C,t) \leq   \sum_{i=t+1}^n \log(\lambda_i(C)/\lambda_{\min})$,
        \item if $\mu_{\min} \geq n-t$\,, then   $\zopcomp(C,s,t,\cdot,\cdot) \leq \zspectral(C,t)$\,.\\[-5pt]
    \end{enumerate}
    \item \emph{Relations among three or more bounds:} For \ref{scaleglinx}\,, consider the scaling parameter $\gamma := 1/\lambda_t(C)^2$, then we have
    \begin{enumerate}
        \item if $\mu_{\max} \geq nt/s$,  
    then $\zscaleglinx(C,s,t,\cdot,\cdot)  \leq \zspectral(C,t) \leq  \zgamma(C,s,t,\cdot,\cdot)$  (and hence also for the optimal $\gamma$), and  $\zop(C,s,t,\cdot,\cdot)  \leq \zspectral(C,t) \leq  \zgamma(C,s,t,\cdot,\cdot)$\,,        
        \item if $\mu_{\max} \geq n-s + t$, 
        then $\zscaleglinx(C,s,t,\cdot,\cdot) = \zop(C,s,t,\cdot,\cdot) = \zspectral(C,t) = \zcgmesp(C,s,t,\cdot,\cdot)$ (and so $\gamma := 1/\lambda_t(C)^2$ is an optimal scaling parameter in this case). 
    \end{enumerate}
\end{enumerate}

\begin{theorem}\label{thm:glinxrelaxed-dominates-spec}
   Let $C\in\mathbb{S}^n_{+}$\,, with $r:=\rank(C)$, $0<t\leq r$,  $t\leq s < n$, and $\gamma := 1/\lambda_t(C)^2$.  Let
\begin{equation}\label{relaxedglinx}  
z(C,s,t) :=\max \left\{ \fscaleglinx(X;\gamma) \,:\,
0\preceq X \preceq I_n,\, \tr(X) = t\right\}.
\end{equation}
   Then, we have $z(C,s,t) \leq \zspectral(C,t)$.
\end{theorem}

\begin{proof}
    Let $({x}^*,{X}^*)$ be an optimal solution to \eqref{relaxedglinx}\,. Let $\Phi \Lambda \Phi^\top$ be any real Schur decomposition of $C$. Then, we have that 
    \begin{align*}
    2\cdot z(C,s,t) &= \ldet(\gamma CX^*C + I_n -X^*) - t\log(\gamma)\\
        &=  \ldet( \gamma\Lambda \Phi^\top X^* \Phi \Lambda  + I_n -\Phi^\top X^* \Phi) - t\log(\gamma)\\
        &\leq \textstyle\sum_{i \in N} \log( (\gamma \lambda_i(C)^2-1) (\Phi^\top X^* \Phi)_{ii} +1) - t\log(\gamma),
    \end{align*}
    where the inequality comes from Hadamard's inequality (see, for example,  \cite[Theorem 7.8.6]{HJBook} for the more general
Oppenheim’s inequality). Note that $0\preceq X^* \preceq I_n \Rightarrow 0 \preceq \Phi^\top X^* \Phi \preceq I_n$ and $\tr(\Phi^\top X^* \Phi) =\tr(X^*\Phi\Phi^\top) =  \tr(X^*) = t$. Then, for $d:=\diag(\Phi^\top X^* \Phi)$, we have that $d \in [0,1]^n$ and $\mathbf{e}^\top d = t$. 

Let $\zeta_i := \gamma \lambda_i(C)^2 - 1$ for $i \in N$, and note that we have $\zeta_1 \geq \zeta_2 \geq \cdots \geq \zeta_n$\,, where $\zeta_i \geq 0$ for $i \in T$, and $\zeta_i \leq 0$ for $i \in N\setminus T$.
      Consider the following convex optimization problem\footnote{This problem is a variation of the ``RM (rate maximization) loading problem'', see \cite{papandreou2007bit}.}
\begin{equation}\label{eq:prob_mask}
        \max \left\{\textstyle\sum_{i\in N}\log(\zeta_i d_i  + 1)     \, : \, \mathbf{e}^\top d=t,~d\in[0,1]^n
\right\}.
\end{equation}
Then, from the optimality conditions\footnote{This same proof idea can be seen in \cite[Theorem 51]{MESP2DOPT}, by replacing $d$ by $x$ and $t$ by $s$}, we can verify that $d_i^* := 1$ for $i \in T$ and $d_i^*:= 0$ for $i\in N\setminus T$ is an optimal solution for \eqref{eq:prob_mask}. Then, we have that $z(C,s,t) \leq \frac{1}{2}\big(\sum_{i=1}^t \log(\gamma \lambda_i(C)^2) - t\log(\gamma)\big) = \sum_{i=1}^t \log(\lambda_i(C)) = \zspectral(C,t)$.
\end{proof}

\begin{corollary}\label{cor:glinx-dominates-spec}
     Let $C\in\mathbb{S}^n_{+}$\,, with $r:=\rank(C)$, $0<t\leq r$, and $t\leq s < n$. Consider the  \ref{scaleglinx} bound with scaling parameter $\gamma := 1/\lambda_t(C)^2$.  
   Then, we have $\zscaleglinx(C,s,t,\cdot,\cdot) \leq \zspectral(C,t)$.
\end{corollary}

\begin{proof}
    The result comes directly from Theorem \ref{thm:glinxrelaxed-dominates-spec} and the fact that $\zscaleglinx(C,s,t,\cdot,\cdot)\leq z(C,s,t)$, with $z(C,s,t)$ defined in \eqref{relaxedglinx}.
\end{proof}

\begin{remark}
   In  
   Corollary \ref{cor:glinx-dominates-spec}, we establish that  the \ref{scaleglinx} bound with scaling parameter $\gamma := 1/\lambda_t(C)^2$ dominates the spectral bound for $\GMESP$. This result is new even for $\MESP$.
\end{remark}

Generally, the \ref{GNLP-Id} bound and the  \ref{GNLP-Id-comp} bound do not dominate the spectral bound. But, motivated by Theorem \ref{thm:ddgfact-almost-dominates}, we derive analogous inequalities that relate the \ref{GNLP-Id} bound and the \ref{GNLP-Id-comp} bound to the spectral bound through the eigenvalues of  $C$.
\begin{theorem} 
Let $C\in\mathbb{S}^n_{+}$\,, with $r:=\rank(C)$, $0<t\leq r$, and $t\leq s < n$. Then, we have
\begin{enumerate}
    \item[\rm($i$)] $\zop(C,s,t,\cdot,\cdot)-\zspectral(C,t) \leq   \sum_{i=1}^t \log(\lambda_{\max}/\lambda_i(C)),$\\[-5pt]
    \item[\rm($ii$)] $\zopcomp(C,s,t,\cdot,\cdot)-\zspectral(C,t) \leq   \sum_{i=t+1}^n \log(\lambda_i(C)/\lambda_{\min}).$
\end{enumerate}
\end{theorem}

\begin{proof}
    Consider the factorizations $I_n - (1/{\lambda_{\max}})C = HH^\top$ and $(1/{\lambda_{\min}})C -I_n= GG^\top$. As $0 \preceq X \preceq I_n$\,,  we have $0 \preceq I_n - H^\top X H \preceq I_n$ and $0 \preceq G^\top X G \preceq G^\top G$.
    Then,
    \begin{align*}
        \zop(C,s,t,\cdot,\cdot;H) &=  \ldet\left(I_n - H^\top X H\right) + t\log(\lambda_{\max})
        \leq t\log(\lambda_{\max})\,,
    \end{align*}  
    and 
    \begin{align*}
        \zopcomp(C,s,t,\cdot,\cdot;G) &= \ldet\left(I_n + G^\top X G\right) + t\log(\lambda_{\min})\\
        &\leq \ldet\left(I_n + G^\top G\right) + t\log(\lambda_{\min})= \ldet(C) + (t-n)\log(\lambda_{\min})\,.
    \end{align*}
Additionally, we have 
    \[
    \zspectral(C,t) + \textstyle\sum_{i=1}^t \log(\lambda_{\max}/\lambda_i(C)) = t \log(\lambda_{\max})\,,
    \]
and
\[
    \zspectral(C,t) +\textstyle\sum_{i=t+1}^n \log(\lambda_i(C))= \textstyle\sum_{i= 1}^n \log(\lambda_i(C)) = \ldet(C).
    \] 
    The result follows.
\end{proof}

In the following corollary, we point out that the \ref{GNLP-Id} bound and the \ref{GNLP-Id-comp} bound \emph{dominate} the spectral bound under conditions on the eigenvalues of $C$.

\begin{corollary}\label{cor:dom_op_spec}
Let $C\in\mathbb{S}^n_{+}$\,, with $r:=\rank(C)$, $0<t\leq r$, and $t\leq s < n$.  
\begin{enumerate}
    \item[\rm($i$)] If $\mu_{\max} \geq t$, then we have  $\zop(C,s,t,\cdot,\cdot) \leq \zspectral(C,t)$.
    \item[\rm($ii$)] If $\mu_{\min} \geq n-t$, then we have  $\zopcomp(C,s,t,\cdot,\cdot) \leq \zspectral(C,t)$.
\end{enumerate}
\end{corollary}

Next, we demonstrate that when $\mu_{\max}$ is greater than a threshold of $n-s+t$, then the \ref{scaleglinx} bound with scaling parameter $\gamma := 1/\lambda_t(C)^2$, the  \ref{GNLP-Id} bound, and the spectral bound equal the optimal objective value of GMESP.

\begin{theorem} 
    Let $C\in\mathbb{S}^n_{+}$\,, with $r:=\rank(C)$, $0<t\leq r$, and $t\leq s < n$. Consider the {\rm\ref{scaleglinx}} bound with scaling parameter $\gamma := 1/\lambda_t(C)^2$.  If $\mu_{\max} \geq n-s + t$, then for any set $S \subset N$ with $|S| = s$, we have 
    \[
    \textstyle\sum_{\ell \in T} \log(\lambda_\ell(C_{S,S})) = t\log(\lambda_{\max}),
    \]
    and therefore
        \[
        \zscaleglinx(C,s,t,\cdot,\cdot)  = \zop(C,s,t,\cdot,\cdot) = \zspectral(C,t) = \zcgmesp(C,s,t,\cdot,\cdot),
        \]
    and so $\gamma := 1/\lambda_t(C)^2$ is an optimal scaling parameter in this case.
\end{theorem}

\begin{proof}
Consider any set $S \subset N$, with $|S| = s$. Let $Q \in \mathbb{R}^{n\times s}$ be defined as $Q_{S\cdot} := I_s$ and $Q_{N\setminus S \cdot} := 0$. By the Poincaré separation theorem (see, e.g., \cite[Corollary 4.3.16]{HJBook}), we have that
$
\lambda_i(C) \geq \lambda_i(Q^\top C Q) \geq \lambda_{n-s+i}(C)$ for $i=1,\dots,s$.
In particular, we have for $i\in T$ that $\lambda_i(C) = \lambda_{\max}$ and that 
\[
 \lambda_i( C_{S,S} ) = \lambda_i(Q^\top C Q) \geq  \lambda_{n-s+i}(C) = \lambda_{\max}\,.
\]
Therefore $\sum_{\ell \in T} \log(\lambda_\ell(C_{S,S})) = t\log(\lambda_{\max})$ for any set $S \subset N$ with $|S|=s$. From 
Corollary \ref{cor:glinx-dominates-spec}, we have that $\zscaleglinx(C,s,t,\cdot,\cdot) \leq \zspectral(C,t) = t\log(\lambda_{\max}$).  Moreover, because $\mu_{\max} \geq n-s+t \geq t$, we have from Corollary \ref{cor:dom_op_spec} that $\zop(C,s,t,\cdot,\cdot) \leq \zspectral(C,t) = t\log(\lambda_{\max})$. The result follows.
\end{proof}

\cite{GMESP_Alg} showed numerically that the \ref{ddgfact} bound typically presents larger gaps than the spectral bound when $s-t$ is large. Here, we present a theoretical comparison that establishes sufficient conditions under which the spectral bound  dominates the  \ref{ddgfact} bound.

\begin{theorem}\label{thm:dom-spec-over-ddgfact}
    Let $C\in\mathbb{S}^n_{+}$\,, with $r:=\rank(C)$, $0<t\leq r$, and $t\leq s < n$.  If any of the following successively weaker conditions holds,   then we have $\zspectral(C,t) \leq  \zgamma(C,s,t,\cdot,\cdot)$. 
    \begin{enumerate}
       \item[\rm($i$)] $\mu_{\max} \geq nt/s$,
        \item[\rm($ii$)] $\tr(C) \geq t\lambda_{\max}$ and $\sum_{\ell = 1}^t \log\big(\frac{nt}{s}\frac{\lambda_{\ell}(C)}{\tr(C)}\big) \leq 0$,
        \item[\rm($iii$)]  
        $\sum_{i=1}^t\log(\lambda_i(C)) - \Gamma_t(C) + t \log(n/s)\leq 0$.
    \end{enumerate}
\end{theorem}

\begin{proof}
    Let $\hat{x}:=(s/n)\mathbf{e}$ and note that $\hat{x}$ is a feasible solution to \ref{ddgfact} when there are no side constraints $Ax\leq b$. Let $F:=C^{1/2}$, then
    \[
    \zgamma(C,s,t,\cdot,\cdot) \geq \Gamma_t(F^\top \Diag(\hat{x}) F) = \Gamma_t((s/n)C) = \Gamma_t(C) - t\log(n/s).
    \]
   Then, we have that
    \[
    \zspectral(C,t) - \zgamma(C,s,t,\cdot,\cdot) \leq \textstyle\sum_{i = 1}^t \log(\lambda_i(C)) - \Gamma_t(C) + t\log(n/s).
    \]
    Therefore, item ($iii$) follows. 

    Now consider that $\tr(C) \geq t\lambda_{\max}$\,. By Lemma \ref{Ni13}, this implies that the unique integer $\hat\imath$ defining $\Gamma_t(C)$ satisfies $\hat\imath = 0$. Therefore 
    \[
    \Gamma_t(C) = t\log((1/t) \tr(C)).
    \]
    Then, item ($ii$) follows from item ($iii$). 
    Finally, assume that $\mu_{\max} \geq nt/s$, then $\tr(C) \geq (nt/s)\lambda_{\max} \geq t\lambda_{\max}$\,. Then, item ($i$) follows from item ($ii$).
\end{proof}

\begin{remark}
    From 
    Corollary \ref{cor:glinx-dominates-spec}, we also conclude that for any condition presented in Theorem \ref{thm:dom-spec-over-ddgfact}, the {\rm\ref{scaleglinx}} bound with scaling parameter $\gamma := 1/\lambda_t(C)^2$ dominates the {\rm\ref{ddgfact}} bound. 
\end{remark}

Combining Corollary \ref{cor:dom_op_spec} (item 1) and 
Theorem \ref{thm:dom-spec-over-ddgfact} (item 1), we immediately get the 
following result.

\begin{corollary}
     Let $C\in\mathbb{S}^n_{+}$\,, with $r:=\rank(C)$, $0<t\leq r$, and $t\leq s < n$. If $\mu_{\max} \geq nt/s$,  
    then we have $\zop(C,s,t,\cdot,\cdot) \leq  \zgamma(C,s,t,\cdot,\cdot)$.
\end{corollary}


\section{\texorpdfstring{On the value of the constraints linking $x$ and $X$}{On the value of the constraints linking x and X}}\label{sec:value}

In this section, we examine the
value of the various constraints in $\Pnst$ that link 
$x$ and $X$. These constraints can
be computationally expensive for
solvers to work with, so we are interested
in how important they are for
getting good bounds.


\subsection{\texorpdfstring{The case of the row-norm constraints $\|X_{i\cdot}\|_2 \leq x_i$\,, $i\in N$\,}{The case of the row-norm constraints}}\label{sec:SOCP}~

\medskip

We start with a small example demonstrating that the row-norm (SOCP) constraints $\|X_{i\cdot}\|_2 \leq x_i$ are not redundant in $\Pnst$.

\begin{example}
We demonstrate that even for $n=3$, $s = 2$ and $t=1$, 
the constraints $x \in [0,1]^n$, $\mathbf{e}^\top x = s$, $0 \preceq X\preceq \Diag(x)$,  and $\tr(X) = t$ do not,  in  general,  imply that $\|X_{i\cdot}\|_2 \leq x_i$ for $i \in N$. 

    Let $\hat x := (1/4,~ 1,~ 3/4)^\top$ and $\hat X := 
    \left(\begin{smallmatrix}
        1/5 \;&\; 9/50 \;&\; 0\\
        9/50 \;&\; 3/10 \;&\; 0\\
        0 \;&\; 0 \;&\; 1/2
    \end{smallmatrix}\right). 
    $ Note that $\mathbf{e}^\top \hat{x} = 2$ and $\tr(\hat{X}) = 1$. Then, we have $\lambda_n(\hat{X}) \approx 0.063$ and $\lambda_n(\Diag(\hat{x})-\hat{X})  \approx 0.003$, so $0\preceq \hat{X}\preceq \Diag(\hat{x})$. But
    \[\textstyle
    \|\hat{X}_{1\cdot}\|_2 = \frac{\sqrt{181}}{50} ~ > ~  
    \frac{1}{4} = \hat{x}_1 \,. \eqno \clubsuit
    \]
\end{example}     

We have observed, experimentally, that when
we remove the constraints $\|X_{i\cdot}\|_2 \leq x_i$\,,
upper bounds do not suffer much, 
which is useful to know, because computational times improve significantly. 
Toward mathematically analyzing the effect of
dropping the (non-redundant) constraints 
 $\|X_{i\cdot}\|_2 \leq x_i$\,,
we define$\hypertarget{PwedgenstTarget}$ 
    \begin{align*}
    &\mathcal{P}^\wedge
    (n,s,t) := \left\{ (x,X)\in \mathbb{R}^n\times \mathbb{S}^{n} \!~:~  x \in [0,1]^n,~\mathbf{e}^{\top}x=s,
    ~0\preceq X\preceq \Diag(x)\,,~\tr(X) = t \right\}.
    \end{align*}
Relative to $\Pnst$, $\Pwedgenst$ simply  omits the constraints $\|X_{i\cdot}\|_2 \leq x_i\,,~i \in N$.
 In what follows, $\zscaleglinxfrakpsd$
 denotes what we would get for
 $\zscaleglinx$\,, if we replaced
 $\Pnst$ by $\Pwedgenst$ in \ref{scaleglinx}\,.
 Additionally, $\zglinxfrakpsd$ refers to   $\zscaleglinxfrakpsd$ with $\gamma:=1$.

Next, we upper bound the deterioration in
 \ref{scaleglinx} 
from  removing the $\|X_{i\cdot}\|_2 \leq x_i$
constraints, when $\gamma := 1/\lambda_t(C)^2$. 

\begin{theorem}\label{thm:upper_bound_diff_glinx_SOC}
Let $C\in\mathbb{S}^n_{+}$\,, with $r:=\rank(C)$, $0<t\leq r$, and $t\leq s < n$. Let $\gamma := 1/\lambda_t(C)^2$ and $\kappa_i := \lambda_i(C)/\lambda_t(C)$ for $i \in N$.  Let $g(u):=\sqrt{1 + \tfrac{t}{n}(u^2 - 1)}$, for $u\in\mathbb{R}$.
Then, we have that
    \[
\zscaleglinxfrakpsd(C,s,t,\cdot,\cdot)-\zscaleglinx(C,s,t,\cdot,\cdot) \leq \left({1+{\frac{4(s-t)}{n}}}\right)^{-1}\left( \sum_{i=1}^{t-1}\log\left(\frac{\kappa_{i}}{g(\kappa_i)}\right)+ \sum_{i=t+1}^n \log\left(\frac{1}{g(\kappa_i)}\right)\right).
    \]
\end{theorem}

\begin{proof}
    Let $\bar{x} := (s/n)\mathbf{e}$ and $\bar{X} :=(t/n) I_n$\,, and note that $(\bar{x},\bar{X})\in \Pnst$.  
    Consider  $(\tilde{x},\tilde{X})$ to be an optimal solution of \ref{scaleglinx} with $\gamma:=1/\lambda_t(C)^2$, when 
 $\Pnst$ is replaced by $\Pwedgenst$.
    Because $0 \preceq \tilde{X} \preceq \Diag(\tilde{x}) \preceq I_n$\,, we have for each $i\in N$ that
    \[
        \|\tilde{X}_{i\cdot}\|_2^2 = (\tilde{X}^2)_{ii} \leq \tilde{X}_{ii} \leq \tilde{x}_i ~\Rightarrow~ \|\tilde{X}_{i\cdot}\|_2 \leq \sqrt{\tilde{x}_i}\,.
    \]
    Moreover, because $\sqrt{\tilde{x}_i}-\tilde{x}_i = \frac{1}{4} - (\sqrt{\tilde{x}_i} - 1/2)^2$ we obtain $\sqrt{\tilde{x}_i}\leq \tilde{x}_i + \frac{1}{4}$, and so
    \begin{equation}\label{eq:tildeXsoc_ineq}
    \textstyle
          \|\tilde{X}_{i\cdot}\|_2 \leq \tilde{x}_i + \frac{1}{4}.
    \end{equation}
    
    Let $\theta:= \frac{4(s-t)}{n + 4(s-t)}$ and note that $\theta \in [0,1)$ and 
    \begin{equation}\label{eq:theta4}
    \textstyle
        \frac{\theta}{4} = (1-\theta)\frac{s-t}{n}.
    \end{equation} 
    Let
    \[
    \hat{x}:=\theta\tilde{x} + (1-\theta)\bar{x},\qquad \hat{X}:=\theta\tilde{X} + (1-\theta)\bar{X}.
    \]
    Then, for $i \in N$, we have
    \begin{align*}
        \|\hat{X}_{i\cdot}\|_2 &= \|\theta \tilde{X}_{i\cdot} + (1-\theta)\bar{X}_{i\cdot}\|_2\\
         &\leq \theta\| \tilde{X}_{i\cdot}\|_2 + (1-\theta)\|\bar{X}_{i\cdot}\|_2\\
         &\leq \theta(\tilde{x}_i + 1/4) + (1-\theta)t/n\\
         &= \theta\tilde{x}_i +(1-\theta)(({s-t})/{n}) + (1-\theta)(t/n)\\
         &= \theta\tilde{x}_i +(1-\theta)(s
         /{n}) \\
         &= \theta\tilde{x}_i +(1-\theta)\bar{x}_i = \hat{x}_i\,,
    \end{align*}
    where the first inequality follows from the triangle inequality, the second inequality from \eqref{eq:tildeXsoc_ineq}, and the second equation from \eqref{eq:theta4}. Because the remaining constraints defining $\Pwedgenst$ are convex, they are preserved under the convex combination, and we conclude that $(\hat{x},\hat{X}) \in \Pnst$. 

    By concavity of $\fscaleglinxthing$ in $X$, we have  
    \[
     \fscaleglinx(\hat{X}) \geq \theta \fscaleglinx(\tilde{X}) + (1-\theta)\fscaleglinx(\bar{X}) = \theta \zscaleglinxfrakpsd(C,s,t,\cdot,\cdot) + (1-\theta)\fscaleglinx(\bar{X}).
    \]
       Because $(\hat{x},\hat{X}) \in \Pnst$, it follows that $\zscaleglinx(C,s,t,\cdot,\cdot) \geq \fscaleglinx(\hat{X})$.  
    Rearranging, we obtain
    \begin{equation}\label{eq:ord}
    \zscaleglinxfrakpsd(C,s,t,\cdot,\cdot)-\zscaleglinx(C,s,t,\cdot,\cdot) \leq  (1-\theta)(\zscaleglinxfrakpsd(C,s,t,\cdot,\cdot)-\fscaleglinx(\bar{X})).
    \end{equation}
  Note that $\Pwedgenst\subset\{X\in\mathbb{S}^n~:~0\preceq X \preceq I_n,\, \tr(X) = t\}$.
    Therefore, by Theorem \ref{thm:glinxrelaxed-dominates-spec}, it follows that $\zscaleglinxfrakpsd(C,s,t,\cdot,\cdot) \leq \sum_{\ell = 1}^t \log(\lambda_\ell(C))$, which combined with \eqref{eq:ord}, yields  
     \[
    \zscaleglinxfrakpsd(C,s,t,\cdot,\cdot)-\zscaleglinx(C,s,t,\cdot,\cdot) \leq  (1-\theta)(\textstyle\sum_{\ell = 1}^t \log(\lambda_\ell(C))-\fscaleglinx(\bar{X})).
    \]
    Setting $\kappa_\ell := \lambda_\ell(C)/\lambda_t(C)$ for $\ell \in N$, we can verify that
    \[
    \fscaleglinx(\bar{X}) = \textstyle\sum_{\ell\in N} \log\left((1 + (t/n)(\kappa_{\ell}^2 -1))^{\scriptscriptstyle 1/2}\right) + t\log(\lambda_t(C)) = \textstyle\sum_{\ell\in N} \log(g(\kappa_{\ell})) + t\log(\lambda_t(C)).
    \]
    Noting that $1-\theta = n/(n+4(s-t)) = (1 + 4(s-t)/n)^{-1}$, we obtain
    \begin{align*}
        \zscaleglinxfrakpsd(C,s,t,\cdot,\cdot)-\zscaleglinx(C,s,t,\cdot,\cdot) &\leq   (1 + 4(s-t)/n)^{-1}\left(\textstyle\sum_{\ell = 1}^t \log(\kappa_\ell)-\textstyle\sum_{\ell\in N} \log(g(\kappa_{\ell}))\right)\\
        &=  (1 + 4(s-t)/n)^{-1}\left(\textstyle\sum_{\ell = 1}^{t-1} \log(\kappa_\ell/g(\kappa_{\ell}))+\textstyle\sum_{\ell=t+1}^n \log(1/g(\kappa_{\ell}))\right).
    \end{align*}
    Note that in the last equality, we consider that  $\kappa_t = 1$, and hence,  $g(\kappa_t) = 1$. 
\end{proof}

\begin{remark}
To build intuition for the upper bound in Theorem \ref{thm:upper_bound_diff_glinx_SOC}, we refer to Figure \ref{fig:bounds-diff-SOC}. The figure contains two plots illustrating the contribution of a single  eigenvalue $\lambda_i(C)$ to the bound as a function of its magnitude relative to $\lambda_t(C)$, measured by the ratio $\kappa_i:=\lambda_i(C)/\lambda_t(C)$. The left plot corresponds to the case where  $\lambda_i(C)\geq \lambda_t(C)$ (i.e., $\kappa_i \geq 1$), while the right plot corresponds to the case where  $\lambda_i(C)\leq \lambda_t(C)$ (i.e., $\kappa_i\leq 1$). In both plots we see that if $\lambda_i(C)=\lambda_t(C)$ (i.e., $\kappa_i=1$), it does not contribute to the bound. 

As $\kappa_i$ moves away from $1$ in either direction, the contribution of $\lambda_i(C)$ increases rapidly at first and then gradually levels off, approaching a horizontal  dashed line. The dashed lines represent upper bounds on the contribution: $\tfrac{1}{2}\log(n/t)$ asymptotically approached as $\kappa_i \to \infty$ in the left plot, and $\tfrac{1}{2}\log(n/(n-t))$, attained at  $\kappa_i = 0$ in the right plot.  

The five curves in each plot correspond to different values of the ratio $t/n$ and exhibit a complementary pattern across the two plots. When $t/n$ is small, the eigenvalue $\lambda_i$ contributes more significantly if it exceeds $\lambda_t$, while its impact is  negligible if it is less than $\lambda_t$\,. Conversely, when $t/n$ is large, $\lambda_i$ contributes more significantly if it is smaller than $\lambda_t$, and has  negligible impact when it is larger. This asymmetry is counteracted by a balancing effect: the side with larger per-eigenvalue contribution contains fewer eigenvalues. Specifically, when $t/n$ is small,  there is only a small number $t$ of eigenvalues that exceed $\lambda_t$ (with $\kappa_i\geq 1$), whereas when $t/n$ is large, only $n-t$  eigenvalues fall below $\lambda_t$ (with $\kappa_i\leq 1$). Thus, although individual contributions may be large, their aggregate remains controlled.

Finally, the multiplicative factor $(1+4(s-t)/n)^{-1}$ reduces the entire bound and decreases as $s$ increases. This effect remains modest, however, as the factor is bounded below by $1/5$.

    \begin{figure}[!ht]
    \centering
    \includegraphics[width=0.49\textwidth]{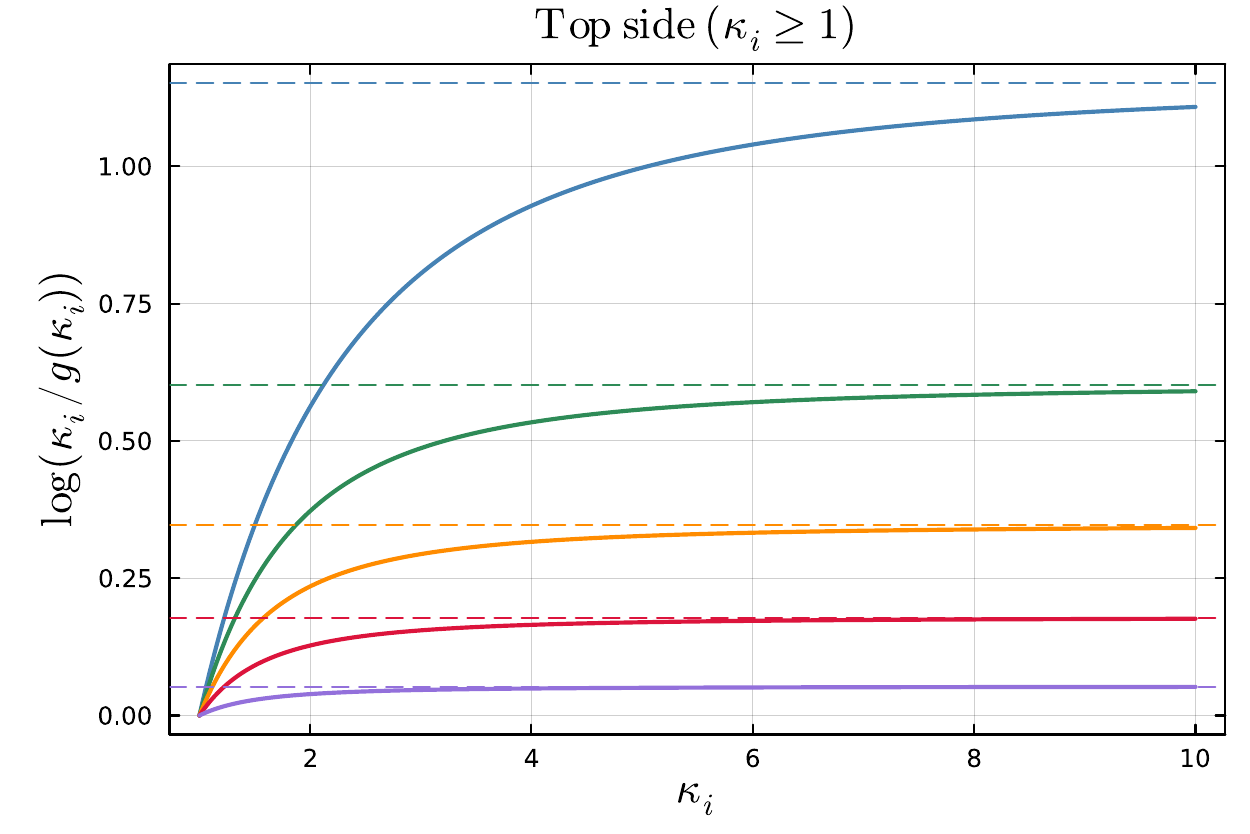}
    \includegraphics[width=0.49\textwidth]{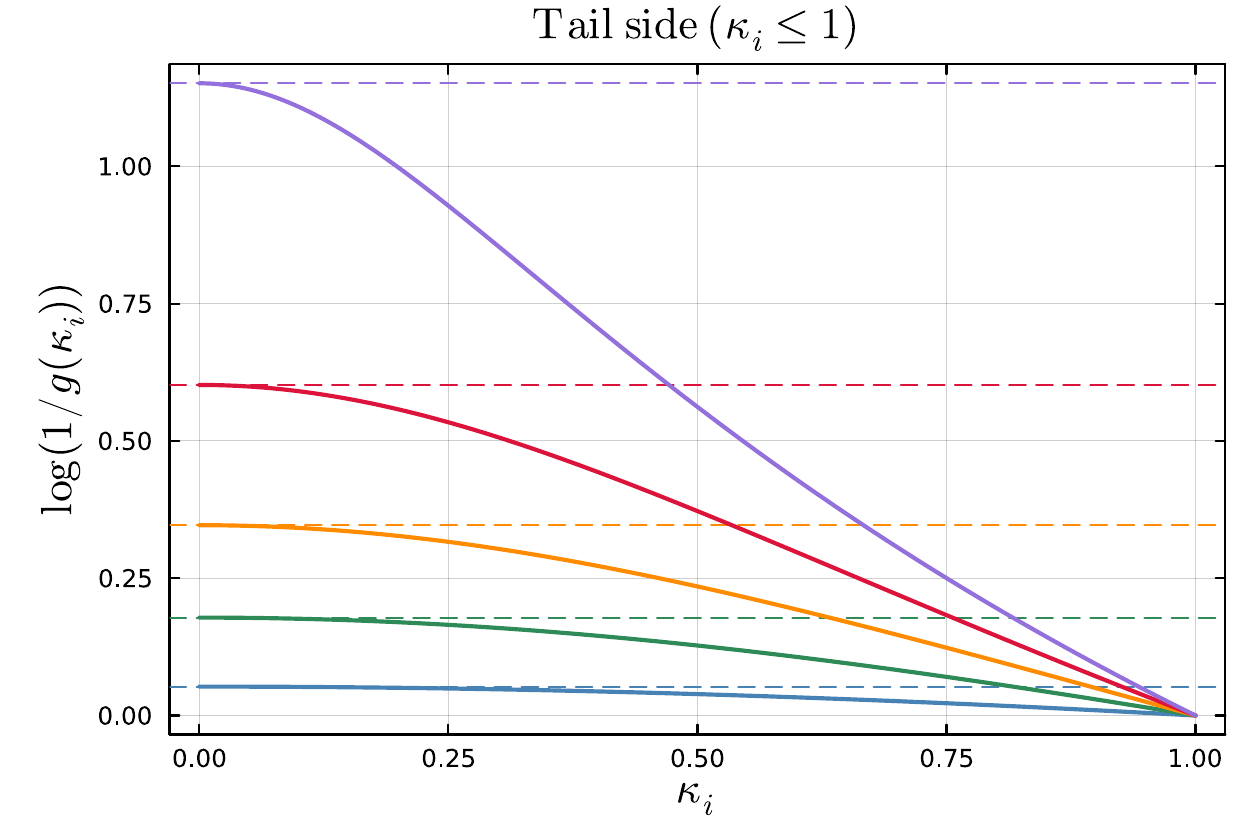}\\
    \includegraphics[width=0.6\textwidth]{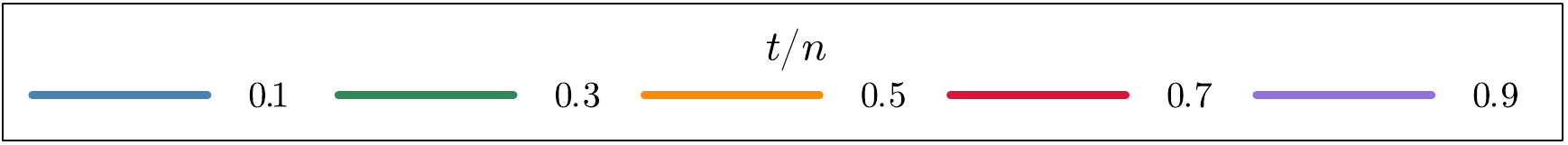}
    \caption{
    Contribution of the eigenvalue $\lambda_i(C)$ to the upper bound in Theorem \ref{thm:upper_bound_diff_glinx_SOC} as a function of $\kappa_i:=\lambda_i(C)/\lambda_t(C)$, for  different ratios $t/n$}
   \label{fig:bounds-diff-SOC}
\end{figure}
\end{remark}

 Although we have derived an upper bound on the deterioration in \ref{scaleglinx} resulting from removing the conic constraints $\|X_{i\cdot}\|_2 \leq x_i$ when $\gamma := 1/\lambda_t(C)^2$, and have observed only minimal deterioration in our numerical experiments, we emphasize that these conic constraints are not redundant. In fact, we next present a family of examples (with $\gamma := 1$)
showing that the loss
in the \ref{glinx} bound
can grow linearly in the size of the problem.
We start with a useful lemma, which lets us
take a constant decrease example, and 
build from it a family of examples with useful properties. The construction is more general than we need (it can handle side constraints), but we present it in this greater generality as
it may be used in future work. 

In the next lemma and its proof, 
for $X \in \mathbb{R}^{m\times n}$ and $Y \in \mathbb{R}^{p\times q}$, we denote their Kronecker product by $X \otimes Y \in \mathbb{R}^{mp\times nq}$.
Further, $\mathbf{1}_k$ denotes the vector of ones having $k$
components. Note that we continue to use $\mathbf{e}$ for
a vector of ones when the number of components is clear from context.

\begin{lemma}\label{lem:familyblocks}
    Let $k \geq 1$ be an integer.
    For any vector $y\in \mathbb{R}^n$, let  $y^{(k)} := \mathbf{1}_k \otimes y$, and for any matrix $M \in \mathbb{R}^{m\times n}$, let $M^{\oplus k} := I_k \otimes M$.
    \begin{enumerate}
        \item[\rm($i$)] Let $(\hat{x},\hat{X})$ be a feasible solution for the {\rm\ref{glinx}} associated with 
        {\rm CGMESP}$( C, s, t, A, b)$, and let its objective value be $\hat\zeta$. Then, $(\hat{x}^{(k)},\hat{X}^{\oplus k})$ is feasible for the {\rm\ref{glinx}} associated with ${\rm CGMESP}({C}^{\oplus k},k{s},k{t},{A}^{\oplus k},\allowbreak b^{(k)})$, and its objective value is $k\hat{\zeta}$.       
     \item[\rm($ii$)]  Let $(\hat\Theta,\hat\upsilon,\hat\nu,\hat\pi,\hat\tau,\hat\xi, \hat Z, \hat\Omega,\hat W,\hat\eta)$ be a feasible solution for the {\rm\ref{eq:dual_glinx}} associated with {\rm CGMESP}$( C, s, t,\allowbreak A, b)$, and let its objective value be $\hat\zeta$. Then, $(\hat\Theta^{\oplus k},\hat\upsilon^{(k)},\hat\nu^{(k)},\hat\pi^{(k)},
 \hat\tau,\hat\xi,\hat Z^{\oplus k},\hat\Omega^{\oplus k},
 \hat W^{\oplus k},\hat\eta^{(k)})$ is a feasible solution for the {\rm\ref{eq:dual_glinx}} associated with {\rm CGMESP}$({C}^{\oplus k},k{s},k{t},{A}^{\oplus k},b^{(k)})$, and its objective value is $k\hat\zeta$.
    \end{enumerate}
\end{lemma}

\begin{proof}
    We repeatedly use the standard Kronecker identities $
(A\otimes B)(C\otimes D)=(AC)\otimes(BD)$ (see, e.g., \cite[Lemma 4.2.10]{HJBook2}), 
$\Diag(y^{(k)}) = I_k\otimes \Diag(y)$,
together with the facts that whenever $M\succeq 0$, we have
$I_k\otimes M \succeq 0$ (see, e.g., \cite[Eq. 4.2.8]{HJBook2}), and $\ldet(I_k\otimes M)=k\ldet(M)$ (see, e.g., \cite[Page 249, Prob.~1]{HJBook2}).

For part ($i$), because $(\hat{x},\hat{X})$ is feasible for \ref{glinx}, we have:
\begin{itemize}
    \item $\mathbf{e}^\top \hat{x}^{(k)} = k \mathbf{e}^\top \hat{x} = ks$; 
    \item $(I_k\otimes A)\hat x^{(k)} = \mathbf{1}_k  \otimes (A\hat{x}) \leq \mathbf{1}_k \otimes b = b^{(k)}$;  
    \item $\tr(\hat{X}^{\oplus k}) = k \tr(\hat{X}) = kt$;
    \item $\hat{X}^{\oplus k} = I_k \otimes \hat{X} \succeq 0$;  
    \item $\Diag(\hat{x}^{(k)})-I_k\otimes \hat{X} =  I_k\otimes (\Diag(\hat{x}) - \hat{X}) \succeq 0$;
    \item  The row-norm constraints are preserved blockwise, because each row of
$\hat{X}^{\oplus k}$ has exactly one block equal to a row of $\hat{X}$ and zeros elsewhere.
\end{itemize}
We conclude that $(\hat x^{(k)},\hat{X}^{\oplus k})$ is feasible for \ref{glinx}, and with objective value
\begin{align*}
        \ldet({C}^{\oplus k}\hat{X}^{\oplus k}{C}^{\oplus k} + I_{nk} - \hat{X}^{\oplus k}) &= \ldet((I_k \otimes C)(I_k\otimes \hat{X})(I_k \otimes C) + I_{nk} - (I_k \otimes \hat{X}))\\
        &= \ldet(I_k \otimes (C\hat{X}C) + I_k \otimes I_{n} - I_k \otimes \hat{X})\\ 
        &= \ldet(I_k \otimes (C\hat{X}C +  I_{n} - \hat{X}))\\ 
        &= k\ldet(C\hat{X}C +  I_{n} - \hat{X}) = k \hat{\zeta}.
    \end{align*}

    For part ($ii$), to lighten the notation, we define  
    \[
    (\tilde\Theta,\tilde\upsilon,\tilde\nu,\tilde\pi,\tilde\tau,\tilde\xi, \tilde Z, \tilde\Omega,\tilde W,\tilde\eta) := 
    (\hat\Theta^{\oplus k},\hat\upsilon^{(k)},\hat\nu^{(k)},\hat\pi^{(k)},
 \hat\tau,\hat\xi,\hat Z^{\oplus k},\hat\Omega^{\oplus k},
 \hat W^{\oplus k},\hat\eta^{(k)}).
\]
Because $(\hat\Theta,\hat\upsilon,\hat\nu,\hat\pi,\hat\tau,\hat\xi, \hat Z, \hat\Omega,\hat W,\hat\eta)$ is feasible for \ref{eq:dual_glinx}, we have:
    \begin{itemize}
        \item $ \tilde\upsilon - \tilde\nu - (I_k\otimes A)^\top \tilde\pi - \tilde\tau\mathbf{e} + \tilde\eta + \diag(\tilde Z) = \mathbf{1}_k\otimes (\hat\upsilon - \hat\nu - A^\top \hat\pi  - \hat\tau\mathbf{e} + \hat\eta+  \diag(\hat Z)) = 0$;
        \item $\frac{1}{2}(\tilde{W}+\tilde{W}^\top)-\tilde\Omega+\tilde{Z}
-C^{\oplus k}\tilde\Theta C^{\oplus k}+\tilde\Theta+\tilde\xi I_{kn} = I_k\otimes(\frac{1}{2}(\hat{W}+\hat{W}^\top)-\hat\Omega+\hat{Z}
-C\hat\Theta C+\hat\Theta+\hat\xi I_{n}) = 0$;
        \item the row-norm constraints again hold blockwise, and all other cone constraints are preserved by Kronecker product with $I_k$\,.
    \end{itemize}
    For the objective value, we have
\begin{align*}
   &-\textstyle\frac{1}{2}\ldet(2\tilde\Theta) + \tr(\tilde\Theta) + \tilde\nu^\top \mathbf{e} + \tilde\pi^\top b^{(k)} + \tilde\tau (ks) + \tilde\xi (kt) - \textstyle\frac{nk}{2}\\
   &=k\left(-\textstyle\frac{1}{2}\ldet(2\hat\Theta) + \tr(\hat\Theta) + {\hat\nu}^\top \mathbf{e} + {\hat\pi}^\top b + \hat\tau s + \hat\xi t - \textstyle\frac{n}{2}\right) = k\hat\zeta.
\end{align*}
The result follows.
\end{proof}

\begin{theorem}\label{thm:family_glinx}
    There exists an infinite set of positive integers $\mathcal{I}$ and a family of positive-semidefinite matrices $\{C_{\tilde n}\}_{\tilde n \in \mathcal{I}}$\,, 
    where each 
    $C_{\tilde n}$ is of order ${\tilde n}$, 
    such that 
    \[
        \zglinxfrakpsd(C_{\tilde n}\,,s_{\tilde n}\,,t_{\tilde n}\,,\cdot,\cdot) - \zglinx(C_{\tilde n}\,,s_{\tilde n}\,,t_{\tilde n}\,,\cdot,\cdot) \geq \delta {\tilde n}
    \]
    for some positive scalar $\delta \geq 5\cdot 10^{-6}$.
\end{theorem}

\begin{proof}(sketch)
See Appendix \ref{app:sols_family_glinx} for details.
First, we construct a particular small GMESP instance
with $n=6$ and associated rational matrix $C$, $s=4$, and $t=3$. 
The data and key quantities are chosen to be rational so that relevant calculations can be carried out exactly, thereby enabling a fully rigorous proof.
We give a rational point $(\hat{x},\hat{X})\in   \Pwedgenst \setminus \Pnst$ with  objective value $\frac{1}{2}\ldet(C\hat{X}C + I - \hat{X}) \approx 11.804396$. So we have 
$\zglinxfrakpsd(C_n,s_n,t_n,\cdot,\cdot) \geq 11.80439$.
We construct a rational solution $(\hat\Theta,\hat\upsilon,\hat\nu,\hat\pi,\hat\tau,\hat\xi,\hat Z,\hat \Omega,\hat\eta)$
that is feasible for \ref{eq:dual_glinx} with objective value $-\frac{1}{2}\ldet(2\hat\Theta) + \tr(\hat\Theta) + {\hat\nu}^\top \mathbf{e} + \hat\pi^\top b + \hat\tau s + \hat\xi t - \frac{n}{2} \approx  11.804352$, and therefore $\zglinx(C,s,t,\cdot,\cdot) \leq 11.80436$. Then, for this particular instance, we have
\[
\zglinxfrakpsd(C,s,t,\cdot,\cdot) - \zglinx(C,s,t,\cdot,\cdot) \geq  3\cdot 10^{-5}.
\]
Now, considering Lemma \ref{lem:familyblocks},
with $k\geq 1$,
we have 
\[
\zglinxfrakpsd({C}^{\oplus k},k{s},k{t},\cdot, \cdot) - \zglinx({C}^{\oplus k},k{s},k{t},\cdot, \cdot) \geq  3\cdot 10^{-5}k.
\]
For $\tilde n \in \mathcal{I}:=\{6,12,18,\ldots\}$,
letting $C_{\tilde n}:={C}^{\oplus k}$,
$s_{\tilde n}:=ks$, and $t_{\tilde n}:=kt$,
we have $3\cdot 10^{-5}k=
5\cdot 10^{-6}{\tilde n}$, and the result follows.
\end{proof}

\begin{remark}
We note that with respect to
the construction in the
proof of Theorem \ref{thm:family_glinx},
we have $\zcgmespthing(C,s,\allowbreak t,\cdot,\cdot)$ $\approx$ $11.67922$, and so 
$\zcgmespthing(C_{\tilde n}\,,s_{\tilde n}\,,t_{\tilde n}\,,\cdot,\cdot) \approx 11.67922{\tilde n}/6$.
Considering the linear growth of 
$\zglinx(C_{\tilde n}\,,\allowbreak s_{\tilde n}\,,\allowbreak t_{\tilde n}\,,\cdot,\cdot)$, we can
conclude that 
$\zcgmespthing(C_{\tilde n}\,,s_{\tilde n}\,,t_{\tilde n}\,,\cdot,\cdot)$ grows linearly in $\tilde n$, which implies, by  Theorem
\ref{thm:family_glinx}, that the row-norm constraints 
can yield a reduction in the optimality gap that is not asymptotically vanishing.
See \cite[Sections 2 and 4]{ChenFampaLeeLinxGaps} for similar results
regarding ``masking'' the linx relaxation. 
\end{remark}

We have one final point to make concerning the 
$\|X_{i\cdot}\|_2 \leq x_i$ constraints with regard to convex relaxations of the set $\Xnst$.
It is natural to wonder whether 
the constraints 
$\|X_{i\cdot}\|_p \leq x_i$ are valid,
 for some $p \in [1,2)$, 
 which could help with relaxations.
 But we will next see that, in general, for all $p \in [1,2)$, such constraints are not valid.

\begin{example}
    Choose any $p \in [1,2)$, and let $k$ be any integer satisfying $k > (2^p-1)^{2/(2-p)} \geq 1$. Let $q$ be a positive integer, let $n := k+1+q$, and let $\eta:=\mathbf{0} \in \mathbb{R}^q$. Let $\hat{u} \in \mathbb{R}^{k+1}$ be the unit vector defined by
    $\hat{u} := \big(2^{-1/2}, (2k)^{-1/2}, (2k)^{-1/2}, \ldots, (2k)^{-1/2}\big)^\top$, and let
    $\hat{X} := \left(\begin{smallmatrix} \hat{u} \\ \eta \end{smallmatrix}\right)\left(\begin{smallmatrix} \hat{u} \\ \eta \end{smallmatrix}\right)^\top$.
    Note that $\hat{X} \succeq 0$, $\hat{X}^2 = \hat{X}$, $\tr(\hat{X}) = \hat{u}^\top \hat{u} = 1$,  $\|\hat{X}_{i\cdot}\|_2 = |\hat{u}_i| \leq 1/\sqrt{2}$  for $i = 1, \ldots, k+1$, and $\|\hat{X}_{i\cdot}\|_2 = 0$ for $i = k+2, \ldots, n$. 
    Then, $\hat{X}_{1\cdot} = (2^{-1}, (2\sqrt{k})^{-1}, \ldots, (2\sqrt{k})^{-1}, 0, \ldots, 0)$ and therefore
    \[
    \|\hat{X}_{1\cdot}\|_p = 2^{-1}(1 + k^{(2-p)/2})^{1/p} ~>~ 2^{-1}(1 + 2^p - 1)^{1/p} = 1.
    \]
    Now, let $t := 1$, choose integer $s$ such that $k+1 \leq s < n\, (=k+1+q)$, and define $\hat{x} \in \{0,1\}^n$ by $\hat{x}_i := 1$ for $i = 1, \ldots, s$ and $\hat{x}_i := 0$ for $i = s+1, \ldots, n$. We employ no side constraints $Ax \leq b$. Then, the solution $(\hat{x}, \hat{X})$ is feasible for {\rm \ref{QCGMESP}} (i.e., $(\hat{x}, \hat{X}) \in \Xnst$) but does not satisfy  $\|X_{1\cdot}\|_p \leq x_1$\,. \hfill $\clubsuit$
\end{example}

    
\subsection{\texorpdfstring{The case of $X\preceq \Diag(x)$\,}{The case of the semidefiniteness constraint}}
\medskip

We have observed, experimentally, the importance of the $X\preceq \Diag(x)$ constraint.  
Toward mathematically analyzing the effect of
dropping this constraint,
we define$\hypertarget{PsubsetnstTarget}$
    \begin{align*}
    &\mathcal{P}^{\scriptscriptstyle\subset}
    (n,s,t) := \left\{ (x,X)\in \mathbb{R}^n\times \mathbb{S}^{n} \!~:~  x \in [0,1]^n,~\mathbf{e}^{\top}x=s,
    ~0\preceq X\preceq I_n\,,~\tr(X) = t\,,~
~\|X_{i\cdot}\|_2 \leq x_i\,,~i \in N \right\}
    \end{align*}
Relative to $\Pnst$, $\mathcal{P}^
{\scriptscriptstyle\subset}
(n,s,t)$ replaces 
$X \preceq \Diag(x)$ with the weaker inequality $X \preceq I_n$\,.
In what follows, for each relaxation  $R\in \{$\ref{scaleglinx}\,,~\ref{GNLP-Id},~\ref{GNLP-Id-comp}$\}$,
    $z_R^{\scriptscriptstyle\subset}(C,s,t,A,b)$ 
     denotes what we would get for
 $z_R(C,s,t,A,b)$\,, if we replaced
 $\Pnst$ by $\Psubsetnst$ in the relaxation $R$.

We will demonstrate that under some technical conditions, the spectral bound
$\zspectral(C,t)$, which is known to be quite poor
compared to other bounds for MESP,
 dominates 
$\zscaleglinxfraksoc(C,s,t,\cdot,\cdot)$ 
for all $\gamma>0$, as well as
$\zgnlpidfrakSOC(C,s,t,\cdot,\cdot)$ and $\zgnlpidcompfrakSOC(C,s,t,\cdot,\cdot)$\,.

   \begin{theorem}\label{thm:spec-dom-21norm}
    Let $C\in\mathbb{S}^n_{+}$\,, with $r:=\rank(C)$, $0<t\leq r$, and $t\leq s < n$. Let $\Phi \Lambda \Phi^\top$ be a real Schur decomposition of $C$.
    If there exists an orthonormal $\Phi$ such that $s \geq \|\Phi_{\cdot T}\|_{2,1}$\,, then 
    \begin{itemize}
        \item[\rm($i$)]   $\zspectral(C,t) \leq \zscaleglinxfraksoc(C,s,t,\cdot,\cdot)$
        for all scaling parameters $\gamma>0$; moreover,
         $\gamma := 1/\lambda_t(C)^2$ 
         minimizes $\zscaleglinxfraksoc(C,s,t,\cdot,\cdot)$
         and yields  $\zspectral(C,t) = \zscalehatglinxfraksoc(C,s,t,\cdot,\cdot)$\,.
        \item[\rm($ii$)] $\zspectral(C,t) \leq \zgnlpidfrakSOC(C,s,t,\cdot,\cdot)$ and  $\zspectral(C,t) \leq \zgnlpidcompfrakSOC(C,s,t,\cdot,\cdot)$\,.
    \end{itemize}
\end{theorem}
 
\begin{proof}
We will construct a feasible solution $(\hat{x},\hat{X}) \in 
\Psubsetnst$
and demonstrate that the objective functions of \ref{scaleglinx}\,, \ref{GNLP-Id} and \ref{GNLP-Id-comp} evaluated at $(\hat{x},\hat{X})$ all equal $\zspectral(C,t)$.

Let $\Phi \Lambda \Phi^\top$ be a real Schur decomposition of $C$ satisfying $s \geq \|\Phi_{\cdot T}\|_{2,1}$\,, and define  $\hat{X} := \Phi_{\cdot T} \Phi_{\cdot T}^\top$\,. Because the columns of    $\Phi_{\cdot T}$ are orthonormal, we have $\Phi_{\cdot T}^\top \Phi_{\cdot T} = I_t$\,, which gives $\tr(\hat{X}) = t$ and $\hat{X}^2 = \hat{X}$, and therefore $0\preceq \hat{X} \preceq I_n$\,. Moreover, for each $i \in N$, we have 
\[
\|\hat{X}_{i\cdot}\|_2^2 = \|\Phi_{i,T}\Phi_{\cdot T}^\top\|_2^2 = \Phi_{i,T}\Phi_{\cdot T}^\top\Phi_{\cdot T}\Phi_{i,T}^\top = \|\Phi_{i,T}\|_2^2\,.
\]
Therefore $\|\hat{X}_{i\cdot}\|_2 =  \|\Phi_{i,T}\|_2$\,, and in particular $0 \leq \|\hat{X}_{i\cdot}\|_2 \leq 1$.
Let
\begin{equation}\label{eq:thm_dom_spec_op_theta}
    \theta := (s-\|\hat{X}\|_{2,1})/(n- \|\hat{X}\|_{2,1}).
\end{equation}
Because $\|\hat{X}\|_{2,1} = \|\Phi_{\cdot T}\|_{2,1} \leq s \leq n$, it follows that $\theta \in [0,1]$. Now let  
\[
\hat{x}_i := \|\hat{X}_{i\cdot}\|_2 + (1-\|\hat{X}_{i\cdot}\|_2)\theta \,,\qquad i \in N.
\]
By construction, $\hat{x} \in [0,1]^n$. Moreover, because $\theta \geq 0$, we have
$\hat{x}_i \geq \|\hat{X}_{i\cdot}\|_2$ for $i \in N$. Finally, 
\[
\mathbf{e}^\top \hat{x} = \|\hat{X}\|_{2,1} + (n-\|\hat{X}\|_{2,1})\theta = s,
\]
where the last equation follows from \eqref{eq:thm_dom_spec_op_theta}. Then, we conclude that $(\hat{x},\hat{X}) \in \mathcal{P}^{\scriptscriptstyle \subset}(n,s,t)$. 

We now evaluate the objective function
of each relaxation at $(\hat{x},\hat{X})$.

\medskip

\noindent ($i$):~  Fix any $\gamma > 0$ and define $L(X;\gamma) := \gamma C{X}C + I - {X}$. Recall that $\hat{X}=\Phi\Diag(\mathbb{I}_t)\Phi^\top$, where $\mathbb{I}_t\in\mathbb{R}^n$ has its first $t$ entries equal to one and its remaining entries equal to zero.  The objective function of \ref{scaleglinx} is $\fscaleglinx(X;\gamma):= \textstyle\frac{1}{2}(\ldet(L(X;\gamma)) - t\log(\gamma))$. Then, 
    \[
      L(\hat{X};\gamma)=  \Phi \Diag((\gamma\lambda^2_1(C),\dots,\gamma\lambda_t(C)^2, 1,\dots,1))  \Phi^\top,
    \]
    and therefore 
    \[
    \fscaleglinx(\hat{X};\gamma) = \textstyle\frac{1}{2}(\sum_{\ell=1}^t \log(\gamma \lambda_i(C)^2) - t\log(\gamma)) = \zspectral(C,s).
    \]
    We note that the value of $\fscaleglinx(\hat X;\gamma)$ is independent of $\gamma$, and hence  $\zscaleglinxfraksoc \geq \zspectral(C,t)$ for all $\gamma > 0$. For the reverse inequality, note that  $\Psubsetnst \subset \mathcal{P}(n,s,t)$. Therefore, by  Theorem \ref{thm:glinxrelaxed-dominates-spec}, we obtain $\zscalehatglinxfraksoc \leq \zspectral(C,t)$  for the specific choice of $\gamma := 1/\lambda_t(C)^2$.

\medskip

\noindent ($ii$):~  Let  $H_{n\times n}:=\Phi(I_n-\frac{1}{\lambda_{\max}}\Lambda)^{\scriptscriptstyle 1/2}$ and $G_{n \times n}:= \Phi(\frac{1}{\lambda_{\min}}\Lambda - I_n)^{\scriptscriptstyle 1/2}$. Then, both $H^\top \hat{X} H$ and $G^\top \hat{X} G$ are zero outside their principal $t \times t$ submatrices indexed by $T$. Moreover, we have
\[
(H^\top \hat{X} H)_{T,T} = I_t-\textstyle\frac{1}{\lambda_{\max}}\Lambda_{T,T}\,,\qquad  (G^\top \hat{X} G)_{T,T} = \textstyle\frac{1}{\lambda_{\min}}\Lambda_{T,T} - I_t\,.
\]
 Then, 
\[
    \textstyle\ldet(I_n - H^\top \hat{X} H) + t\log(\lambda_{\max})= \ldet(\frac{1}{\lambda_{\max}}\Lambda_{T,T})+ t\log(\lambda_{\max}) = \textstyle\sum_{\ell=1}^t \log(\lambda_\ell(C)), 
\]
and similarly,
\[
    \textstyle\ldet(I_n + G^\top \hat{X} G) + t\log(\lambda_{\min})= \ldet(\frac{1}{\lambda_{\min}}\Lambda_{T,T})+ t\log(\lambda_{\min}) = \textstyle\sum_{\ell=1}^t \log(\lambda_\ell(C)). 
\]
The result follows.
\end{proof}

The following corollary provides a simpler and easily-verifiable stronger condition for the same conclusion as Theorem \ref{thm:spec-dom-21norm}. 

\begin{corollary} 
    Let $C\in\mathbb{S}^n_{+}$\,, with $r:=\rank(C)$, $0<t\leq r$, and $t\leq s < n$.
    If $s \geq \sqrt{nt}$, then 
    \begin{itemize}
        \item[\rm($i$)]   $\zspectral(C,t) \leq \zscaleglinxfraksoc(C,s,t,\cdot,\cdot)$
        for all scaling parameters $\gamma>0$; moreover,
         $\gamma := 1/\lambda_t(C)^2$  
         minimizes $\zscaleglinxfraksoc(C,s,t,\cdot,\cdot)$
         and yields  $\zspectral(C,t) = \zscalehatglinxfraksoc(C,s,t,\cdot,\cdot)$\,.
        \item[\rm($ii$)] $\zspectral(C,t) \leq \zgnlpidfrakSOC(C,s,t,\cdot,\cdot)$ and  $\zspectral(C,t) \leq \zgnlpidcompfrakSOC(C,s,t,\cdot,\cdot)$\,.
    \end{itemize}
\end{corollary}

\begin{proof}
    Let $\omega := (\|\Phi_{1 T}\|_2\,,\|\Phi_{2 T}\|_2\,,\dots\,,\|\Phi_{n T}\|_2)$ for any orthonormal matrix $\Phi \in \mathbb{R}^{n \times n}$.  Note that 
\[
\|\Phi_{\cdot T}\|_{2,1}^2 = (\mathbf{e}^\top \omega)^2 \leq \|\mathbf{e}\|_2^2 \|\omega\|_2^2 = n \|\Phi_{\cdot T}\|_F^2 = n \tr(\Phi_{\cdot T}\Phi_{\cdot T}^\top) =  n \tr(\Phi_{\cdot T}^\top\Phi_{\cdot T}) = n \tr(I_t) = nt.
\]
 Then, $\|\Phi_{\cdot T}\|_{2,1} \leq \sqrt{nt}$  and the result follows from Theorem \ref{thm:spec-dom-21norm}.
\end{proof}


\section{Generalized scaling for CGMESP}\label{sec:gscale}

\cite{gscale} introduced the notion of ``generalized scaling'' for \ref{CMESP}. Their key idea was to extend the classical ordinary scaling parameter $\gamma \in \mathbb{R}_{++}$\,, introduced in \cite{AFLW_Using}, with a vector parameter, possibly resulting in tighter bounds. 
When the vector parameter takes the form $\gamma \mathbf{e}$, their approach reduces to ordinary scaling. 
In their work, the authors introduced generalized-scaling schemes for the linx,  DDFact and  BQP bounds. They showed that, for linx and BQP, the resulting objective functions are convex in the component-wise logarithm of the vector scaling parameter, a property that was leveraged to design BFGS-based algorithms for optimizing the parameter. Although an analogous convexity result was not established for DDFact in \cite{gscale}, the authors applied a BFGS-based method to locally optimize its vector scaling parameter, which proved effective in improving the unscaled 
bound. Very recently \cite[Prop. 10]{Fatma2026} demonstrated that, similarly to linx and BQP, the g-scaled factorization bound proposed in \cite{gscale} is also convex in the component-wise logarithm of the vector scaling parameter.   

Motivated by these results, we propose generalized scaling procedures for both \ref{glinx} and \ref{ddgfact}. It is straightforward to extend the convexity result for the g-scaled factorization bound for \ref{CMESP} from \cite{Fatma2026} to \ref{ddgfact}. 
In contrast, 
the g-scaled \ref{glinx} bound that we introduce here is not convex in the component-wise logarithm of the vector scaling parameter. Nevertheless, our numerical experiments indicate that applying a BFGS-based algorithm to locally optimize this parameter still yields significant improvements in the resulting bounds. 

In the interest of clarity and following \cite{gscale}, we write \emph{g-scaling} for 
generalized scaling and \emph{o-scaling} for ordinary scaling.
Throughout, we let  $\Upsilon:=\left(\gamma_1,\gamma_2,\ldots,\gamma_n\right)^\top \in \mathbb{R}_{++}^n$ denote the ``scaling vector''. 

Next, we present the g-scaled upper bounds for \ref{CGMESP}.


\subsection{g-scaling for glinx}
 We now apply g-scaling to the generalized linx bound. For $\Upsilon \in \mathbb{R}^n_{++}$\,, we define
\[
\hypertarget{fgscaleglinxtarget}{\fgscaleglinxthing}
(X;\Upsilon):=\textstyle\frac{1}{2}\ldet(\Diag(\Upsilon) C \Diag(\Upsilon)X\Diag(\Upsilon) C \Diag(\Upsilon) + I_n-X )- 2\textstyle\sum_{i \in N} \log(\gamma_i) X_{ii}\,,
\]
and the associated \emph{g-scaled generalized linx bound}
\begin{equation}\label{gscaleglinx}\tag{glinx$_\Upsilon$}  
\begin{array}{ll}
\hypertarget{zgscaleglinxtarget}{\zgscaleglinxthing}(C,s,t,A,b) :=\max \left\{\fgscaleglinx(X;\Upsilon)  \,:\, Ax\leq b,~ (x,X) \in \Pnst\right\}.
\end{array}
\end{equation}

Note that we can interpret \ref{gscaleglinx} as applying the unscaled \ref{glinx} bound to the symmetrically-scaled matrix $\Diag(\Upsilon)C\Diag(\Upsilon)$, and then correcting by $-2\sum_{i\in N}\log(\gamma_i)X_{ii}$. It is easy to check that for $\Upsilon:=\sqrt[\leftroot{-1}\uproot{2}4]{\gamma}\,\mathbf{e}$, we have 
$\fgscaleglinxthing
(X;\Upsilon)=
\fscaleglinxthing
(X;\gamma)
$ when $\tr(X)=t$, 
and therefore \ref{gscaleglinx} truly generalizes \ref{scaleglinx}.

The following theorem establishes that the \ref{gscaleglinx} bound is a valid upper bound for \ref{CGMESP}, enabling us to incorporate g-scaling into \ref{glinx}.

\begin{theorem}
    For any fixed $\Upsilon \in \mathbb{R}^n_{++}$\,, the 
    \ref{gscaleglinx} bound is an upper bound for \ref{CGMESP}. 
\end{theorem}
\begin{proof}
    Let $(\hat{x},\hat{U})$ be an optimal solution for \ref{OCGMESP}, and let $S \subset N$ with $|S|=s$ denote the support of $\hat{x}$. We will consider that \ref{OCGMESP} is a valid extended-variable  formulation of \ref{CGMESP} (see Remark \ref{rem:ocgmesp}).
    
    Define $E:= \Diag(\Upsilon)^{\scriptscriptstyle -1} \hat{U}$ and $V := E(E^\top E)^{-1/2}$. Let $\hat{X}:=V V^\top$.
    
   As $(\hat{x},\hat{U}) \in \Unst$, we have $\mathbf{e}^{\top} \hat{x}=s$, $\hat{x} \in \{0,1\}^n$, $\hat{U}^\top \hat{U}=I_t$, and $\|\hat{U}_{i\cdot}\|_2 \leq \hat{x}_i$ for $i\in N$. Note that $V  \in \mathbb{R}^{n\times t}$ satisfies $V^\top V  = I_t$\,. Moreover, for $i\in N\setminus S$, $E_{i\cdot} = (1/\gamma_i) \hat{U}_{i\cdot} = \mathbf{0}^\top$, so $V_{i\cdot} =  E_{i\cdot}(E^\top E)^{\scriptscriptstyle -1/2} = \mathbf{0}^\top$. 

   Then, we have
\begin{itemize}
    \item $\hat{X}\succeq 0$, since $\hat{X}=V V^\top$.
    \item $\hat{X}^2 = V (V^\top V )V^\top =V I_tV^\top = \hat{X}$.
    \item $\tr(\hat{X}) = \tr(V V^\top) = \tr(V^\top V ) = \tr(I_t) = t$.
    \item For $i \in N$, 
    $
    \|\hat{X}_{i\cdot}\|_2^2 = (\hat{X}\hat{X}^\top)_{ii} = (\hat{X}^2)_{ii} = \hat{X}_{ii}=(V V^\top)_{ii}=\|V_{i\cdot}\|_2^2\,,
    $
    so $\|\hat{X}_{i\cdot}\|_2 = \|V_{i\cdot}\|_2$\,. If $i \in N\setminus S$, then $V_{i\cdot}=\mathbf{0}^\top$ gives $\|\hat{X}_{i\cdot}\|_2 = 0 = \hat{x}_i$\,. If $i \in S$, then because $V^\top V =I_t$\,, we have $V V^\top \preceq I_n$\,, which yields $\|V_{i\cdot}\|_2^2 = (V V^\top)_{ii} \leq 1 = \hat{x}_i$\,. In either case, $\|\hat{X}_{i\cdot}\|_2 \leq \hat{x}_i$\,.
\end{itemize}
Therefore, we have that $(\hat{x},\hat{X}) \in \Xnst$, and by Proposition \ref{prop:calP}, we conclude that $(\hat{x},\hat{X}) \in \Pnst$.
    
     Choose  $ Q\in\mathbb{R}^{n\times (n-t)}$ such that   $R :=\begin{pmatrix}
        V  \!\!&\!\!  Q
    \end{pmatrix}$ 
    satisfies $R^\top R=I_n$\,.
    Note that $R R^\top=I_n$ and, therefore,  $ Q Q^\top=I_n-\hat{X}$. 
   Define     
    \[M:=
    \begin{pmatrix}
            V^\top  \Diag(\Upsilon)C \Diag(\Upsilon) V  & 0\\
             Q^\top  \Diag(\Upsilon)C \Diag(\Upsilon) V  &  I_{n-t}
        \end{pmatrix}.
    \]
    We can verify that 
    \begin{align*}
        &\ldet( \Diag(\Upsilon)C \Diag(\Upsilon)\hat{X} \Diag(\Upsilon)C \Diag(\Upsilon) + I_n -\hat{X})\\
        &\quad = \ldet(R^\top( \Diag(\Upsilon)C \Diag(\Upsilon)\hat{X} \Diag(\Upsilon)C \Diag(\Upsilon) + I_n -\hat{X})R)\\
        &\quad =\ldet(R^\top  \Diag(\Upsilon)C \Diag(\Upsilon)V V^\top  \Diag(\Upsilon)C \Diag(\Upsilon) R + R^\top  Q Q^\top R)\\
        &\quad = \ldet(MM^\top) =2\ldet(V^\top  \Diag(\Upsilon)C \Diag(\Upsilon) V).
    \end{align*}
Then, we have 
    
\begin{equation}\label{eq:gscaleglinx_upper_bound_eqa}
        \tfrac{1}{2}\ldet(\Diag(\Upsilon)C\Diag(\Upsilon) \hat{X} \Diag(\Upsilon)C\Diag(\Upsilon) + I_n - \hat{X}) =  \ldet(V^\top \Diag(\Upsilon)C\Diag(\Upsilon)V ).
    \end{equation}
    Observe that $\Diag(\Upsilon)V  = \Diag(\Upsilon)E(E^\top E)^{-1/2} = \hat{U}(E^\top E)^{-1/2}$, which yields
    \[
    \ldet(V^\top\Diag(\Upsilon)C\Diag(\Upsilon)V ) = \ldet(\hat{U}^\top C \hat{U}) + \ldet((E^\top E)^{\scriptscriptstyle -1}) = \ldet(\hat{U}^\top C \hat{U}) + \ldet(V^\top \Diag(\Upsilon)^2 V ),
    \]
    where the final equality follows from  $\hat{U}^\top\hat{U} = I_t$.
    
    We note that $\zcgmesp(C,s,t,A,b)=\ldet(\hat{U}^\top C \hat{U})$ (see Remark \ref{rem:ocgmesp}). Combining this with \eqref{eq:gscaleglinx_upper_bound_eqa} gives
    \begin{equation}\label{eq:gscaleglinx_equality_reformulation}
    \zcgmesp(C,s,t,A,b) =  \tfrac{1}{2}\ldet(\Diag(\Upsilon)C\Diag(\Upsilon) \hat{X} \Diag(\Upsilon)C\Diag(\Upsilon) + I_n - \hat{X}) - \ldet(V^\top \Diag(\Upsilon)^2 V ).
    \end{equation}

    Let $\Phi\Diag(\sigma)\Phi^\top$ be a real Schur decomposition of $V^\top \Diag(\Upsilon)^2 V $, where $\Phi \in \mathbb{R}^{t \times t}$ satisfies $\Phi\Phi^\top = \Phi^\top \Phi = I_t$ and $\sigma \in \mathbb{R}^t_{++}$\,. 
    Define $W := V \Phi$, and observe that $W^\top W = I_t$, $WW^\top = V V^\top = \hat{X}$, and
 \[
    W^\top \Diag(\Upsilon)^2 W = \Phi^\top V^\top\Diag(\Upsilon)^2V \Phi=
    \Phi^\top \Phi \Diag(\sigma) \Phi^\top \Phi
    = \Diag(\sigma).
    \]
    Hence, for all $j \in T$, we have $
    \sigma_j = W_{\cdot j}^\top \Diag(\Upsilon)^2 W_{\cdot j} = \sum_{i \in N} \gamma_i^2 W_{ij}^2$\,. Note that 
     the identity $W^\top W = I_t$ gives $\sum_{i \in N}W_{ij}^2 = 1$, for all $j\in T$. Applying the weighted AM-GM inequality (see, e.g., \cite[Theorem 9]{hardy1952inequalities}), we obtain for each $j \in T$
     \[
     \sigma_j = \textstyle\sum_{i \in N} \gamma_i^2 W_{ij}^2 \geq \textstyle\prod_{i \in N} (\gamma_i^2)^{W_{ij}^2} ~\Rightarrow~ \log(\sigma_j)  \geq \textstyle\sum_{i\in N} W_{ij}^2 \log(\gamma_i^2).
     \]
     Therefore, summing over $j\in T$ and using the fact that $WW^\top = \hat{X}$, we have
     \[
     \textstyle\sum_{j\in T}\log(\sigma_j)  \geq  \textstyle\sum_{i\in N} \log(\gamma_i^2)\sum_{j\in T}W_{ij}^2 = 2 \textstyle\sum_{i\in N}\log(\gamma_i)\hat{X}_{ii}\,.
     \]
     Recalling that $\sigma$
     denotes the eigenvalues of $V^\top\Diag(\Upsilon)^2 V$, we conclude from \eqref{eq:gscaleglinx_equality_reformulation} that
     \[
     \zcgmesp(C,s,t,A,b) \leq \tfrac{1}{2}\ldet(\Diag(\Upsilon)C\Diag(\Upsilon) \hat{X} \Diag(\Upsilon)C\Diag(\Upsilon) + I_n - \hat{X})  -  2 \textstyle\sum_{i\in N}\log(\gamma_i)\hat{X}_{ii}\,.
     \]
    Because $(\hat{x},\hat{X}) \in \Pnst$ and $A\hat{x}\leq b$, the pair $(\hat{x},\hat{X})$ is feasible for \ref{gscaleglinx}\,. As the right-hand side above equals $\fgscaleglinx(\hat{X};\Upsilon)$, 
    the result follows.
\end{proof}

With  
Corollary \ref{cor:glinx-dominates-spec}, we have that the \ref{scaleglinx} bound  with o-scaling parameter $\gamma := 1/\lambda_t(C)^2$
dominates the spectral bound for $\GMESP$. The following theorem strengthens this:  the \ref{gscaleglinx} bound with a given  g-scaling parameter 
$\Upsilon$  dominates the Lagrangian spectral bound for \ref{CGMESP}.

\begin{theorem} 
   Let $C\in\mathbb{S}^n_{+}$\,, with $r:=\rank(C)$, $0<t\leq r$, and $t\leq s < n$.  Given the Lagrangian spectral bound for {\rm \ref{CGMESP}} defined in \eqref{eq:specbound_constrained} with a diagonal matrix $D^{\pi} \in \mathbb{S}^n_{++}$ defined in \eqref{def:Dpi-spec-bound},  consider the  {\rm\ref{gscaleglinx}} bound with g-scaling parameter $\Upsilon := \lambda_t(D^{\pi}CD^{\pi})^{\scriptscriptstyle -1/2} \diag(D^{\pi})$.  
   Then, we have
    \[
    \zgscaleglinx(C,s,t,A,b) \leq \zspectralL(C,s,t,A,b).
    \]
\end{theorem}

\begin{proof}
    Let $({x}^*,{X}^*)$ be an optimal solution to \ref{gscaleglinx} with g-scaling parameter $\Upsilon:=\lambda_t(D^{\pi}CD^{\pi})^{\scriptscriptstyle -1/2} \diag(D^{\pi})$. 
    Let $E := D^{\pi} C D^{\pi}$ and let $\beta_j:= \sum_{i=1}^m \pi_i A_{ij}$ for $j \in N$, so that $\gamma_j = \lambda_t(E)^{\scriptscriptstyle -1/2}\exp(-\frac{1}{2}\beta_j)$.  Expanding the logarithm gives
    \begin{align*}
        -2\textstyle\sum_{j \in N} \log(\gamma_j)X^*_{jj} =  \textstyle\sum_{j\in  N} \log(\lambda_t(E))X^*_{jj} + \textstyle\sum_{j\in  N}  \beta_j X^*_{jj} = t \log(\lambda_t(E)) + \textstyle\sum_{j\in  N}  \beta_j X^*_{jj}\,.
    \end{align*}
    Now let $w := x^* -\diag(X^*)$. Because $0 \preceq X^* \preceq \Diag(x^*)$ and $x^*\in [0,1]^n$, we have $w \in [0,1]^n$, and because  $\mathbf{e}^\top x^* = s$ and $\tr(X^*) = t$, we get $\mathbf{e}^\top w = s-t$. Substituting $X^*_{jj} = x^*_j - w_j$\,,
    \[
    \textstyle\sum_{j\in  N}  \beta_j X^*_{jj} = \textstyle\sum_{j\in  N}  \beta_j x^*_j - \textstyle\sum_{j\in  N}  \beta_j w_j\,.
    \]
    For the first sum in the right-hand side of the equation above, $Ax^* \leq b$ and $\pi \geq 0$ give $\textstyle\sum_{j\in  N}  \beta_j x^*_j = \pi^\top A x^* \leq \pi^\top b$. For the second, because $w \in [0,1]^n$ with $\mathbf{e}^\top w = s-t$, the sum $\textstyle\sum_{j\in  N}  \beta_j w_j$ is minimized by setting $w_j :=1$ on the $s-t$ indices with smallest $\beta_j$\,, so 
    \[
    \sum_{j\in  N}  \beta_j w_j \geq \min_{\stackrel{K\subset N,}{|K|=s-t}}~
\sum_{j\in K}\beta_j = \min_{\stackrel{K\subset N,}{|K|=s-t}}~\sum_{j\in K}\sum_{i=1}^m \pi_i A_{ij}\,.
    \]  
    Putting these together, we conclude that
    \begin{align*}
        -2\textstyle\sum_{j \in N} \log(\gamma_j)X^*_{jj} &= t \log(\lambda_t(E)) + \textstyle\sum_{j\in  N}  \beta_j X^*_{jj} \\
        &\leq t \log(\lambda_t(E)) + \pi^\top b - \min_{\stackrel{K\subset N,}{|K|=s-t}}~\sum_{j\in K}\sum_{i=1}^m \pi_i A_{ij}\,.
    \end{align*}

    It remains to bound the log-determinant term. Let $\Phi \Lambda \Phi^\top$ be any real Schur decomposition of $E$ and let $h(X;\Upsilon) := \frac{1}{2}\ldet(\Diag(\Upsilon) C \Diag(\Upsilon)X\Diag(\Upsilon) C \Diag(\Upsilon) + I_n-X)$. Then, 
    \begin{align*}
   h(X^*;\Upsilon)
    &= {\textstyle\frac{1}{2}}\ldet(\lambda_t(E)^{-2}EX^*E + I_n-X^*)\\
    &=  {\textstyle\frac{1}{2}}\ldet( \lambda_t(E)^{-2}\Lambda \Phi^\top X^* \Phi \Lambda  + I_n -\Phi^\top X^* \Phi) \\
    &\leq {\textstyle\frac{1}{2}}\textstyle\sum_{\ell \in N} \log( (\lambda_t(E)^{-2} \lambda_{\ell}(E)^2-1) (\Phi^\top X^* \Phi)_{\ell \ell} +1),\\
    &\leq -t\log( \lambda_t(E)) +  \textstyle\sum_{\ell=1}^t \log(\lambda_{\ell}(E)),
    \end{align*}
    where the first inequality comes from Hadamard's inequality (see, for example,  \cite[Theorem 7.8.6]{HJBook} for the more general
Oppenheim’s inequality) and the second inequality follows from the same reasoning as in Theorem \ref{thm:glinxrelaxed-dominates-spec}.

Adding the two bounds together and  noting that the $\pm t\log(\lambda_t(E))$ terms cancel, we obtain
\begin{align*}
    \zgscaleglinx(C,s,t,A,b) &= h(X^*;{\Upsilon}) -2\textstyle\sum_{j \in N} \log(\gamma_j)X^*_{jj}\\
    &\leq \sum_{\ell =1}^t \log(\lambda_\ell(E))  + \pi^\top b - \min_{\stackrel{K\subset N,}{|K|=s-t}}~\sum_{j\in K}\sum_{i=1}^m \pi_i A_{ij}\\
    &= \zspectralL(C,s,t,A,b).
\end{align*}
\end{proof}


\subsection{g-scaling for DDGFact}

Extending g-scaling to the generalized factorization bound is more subtle than for \ref{glinx}, as the \ref{ddgfact} formulation involves only the variable $x$ and lacks an $\ldet$
 term that could absorb the generalized scaling directly. To address this, we establish the following lemma, which provides the eigenvalue inequality needed to incorporate the scaling vector $\Upsilon$ into the \ref{ddgfact} bound.

\begin{lemma}\label{lem:Lee-scale}
 Let $C\in\mathbb{S}^n_{+}$\,, with $r:=\rank(C)$, $0<t\leq r$. Let $\Upsilon \in \mathbb{R}^n_{++}$\,. For every set $S \subset N$ with $|S|=s$ and $0 < t \leq s < n$ and $\rank(C_{S,S})\geq t$, we have 
\[
 \textstyle\sum_{\ell =1}^t \log( \lambda_\ell(C_{S,S})) \leq \textstyle\sum_{\ell=1}^t \log(\lambda_\ell((\Diag(\Upsilon)^{\scriptscriptstyle 1/2}C \Diag(\Upsilon)^{\scriptscriptstyle 1/2})_{S,S}))
   - \textstyle\sum_{\ell=s-t+1}^s \log(\lambda_\ell(\Diag(\Upsilon)_{S,S})).
 \]
\end{lemma}

\begin{proof}
 From \cite[Proposition 2.1]{LeeLind2020}, we have that for any $\Upsilon \in \mathbb{R}^n_{++}$\,,
\begin{align*}
\textstyle\sum_{\ell=1}^t \log(\lambda_\ell(C_{S,S})) 
&= \textstyle\sum_{\ell=1}^t \log\bigl(\lambda_\ell\bigl(
   \Diag(\Upsilon)^{\scriptscriptstyle -1/2}_{S,S}
   \Diag(\Upsilon)^{\scriptscriptstyle 1/2}_{S,S} C_{S,S} \Diag(\Upsilon)^{\scriptscriptstyle 1/2}_{S,S}
   \Diag(\Upsilon)^{\scriptscriptstyle -1/2}_{S,S}\bigr)\bigr) \\
& \leq\textstyle\sum_{\ell=1}^t \log\bigl(
   \lambda_\ell\bigl(\Diag(\Upsilon)^{\scriptscriptstyle -1/2}_{S,S}\bigr)
   \lambda_\ell\bigl(\Diag(\Upsilon)^{\scriptscriptstyle 1/2}_{S,S} C_{S,S} \Diag(\Upsilon)^{\scriptscriptstyle 1/2}_{S,S}\bigr)
   \lambda_\ell\bigl(\Diag(\Upsilon)^{\scriptscriptstyle -1/2}_{S,S}\bigr)\bigr) \\
&= \textstyle\sum_{\ell=1}^t \log(\lambda_\ell(\Diag(\Upsilon)^{\scriptscriptstyle 1/2}_{S,S} C_{S,S} \Diag(\Upsilon)^{\scriptscriptstyle 1/2}_{S,S}))
   + \textstyle\sum_{\ell=1}^t \log(\lambda_\ell(\Diag(\Upsilon)^{\scriptscriptstyle -1}_{S,S})) \\
&= \textstyle\sum_{\ell=1}^t \log(\lambda_\ell(\Diag(\Upsilon)^{\scriptscriptstyle 1/2}_{S,S} C_{S,S} \Diag(\Upsilon)^{\scriptscriptstyle 1/2}_{S,S}))
   - \textstyle\sum_{\ell=s-t+1}^s \log(\lambda_\ell(\Diag(\Upsilon)_{S,S})),
\end{align*}
where the inequality comes from \cite[Theorem 3.3.4]{HJBook2}. Because $\Diag(\Upsilon)^{\scriptscriptstyle 1/2}$ is a diagonal matrix with positive diagonal entries, we have $\Diag(\Upsilon)^{\scriptscriptstyle 1/2}_{S,S} C_{S,S} \Diag(\Upsilon)^{\scriptscriptstyle 1/2}_{S,S} = (\Diag(\Upsilon)^{\scriptscriptstyle 1/2} C \Diag(\Upsilon)^{\scriptscriptstyle 1/2})_{S,S}$. The result follows.
\end{proof}

For  $C\in\mathbb{S}^n_{+}$\,, let $C=FF^\top$,
where $F\in \mathbb{R}^{n\times k}$ for some $k$ satisfying $r\le k \le n$. Building on the previous lemma, for $\Upsilon \in \mathbb{R}^n_{++}$ we define
\[
\hypertarget{fgscalegfacttarget}{\fgscalegfactthing}
(x,y;\Upsilon)
:=\Gamma_t(F^\top \Diag(\Upsilon)^{\scriptscriptstyle 1/2}\Diag(x)\Diag(\Upsilon)^{\scriptscriptstyle 1/2}F )- \textstyle\sum_{i \in N} \log(\gamma_i) y_{i}\,,
\]
and the associated \emph{g-scaled generalized factorization bound}
\begin{align*}\label{ddgfactscale}\tag{\mbox{DDGFact$_{\Upsilon}$}}
\hypertarget{zscalegammatarget}{\zgscalegammathing} (C,s,t,A,b;F)\!:=
&\max \!\left\{\fgscalegfact(x,y;\Upsilon)  \, : \, \mathbf{e}^\top x=s,\, \mathbf{e}^\top y=t,\, \mathbf{0} \leq y\leq x, \,Ax\leq b,\,
x\!\in\![0,1]^n
\right\}.
\end{align*}
Note that, in contrast to \ref{ddgfact}, the formulation \ref{ddgfactscale} requires an auxiliary variable $y$, which arises from the second term in Lemma \ref{lem:Lee-scale}.
We can interpret \ref{ddgfactscale} as applying the unscaled \ref{ddgfact} bound to the symmetrically-scaled matrix $\Diag(\Upsilon)^{\scriptscriptstyle 1/2}C \Diag(\Upsilon)^{\scriptscriptstyle 1/2}$, and then correcting by $- \textstyle\sum_{i \in N} \log(\gamma_i) y_{i}$.
The following theorem establishes that \ref{ddgfactscale} is a valid upper bound for \ref{CGMESP}, enabling us to incorporate g-scaling into \ref{ddgfact}.

\begin{theorem}\label{thm:Gamma-upper-bound}
For fixed $\Upsilon \in \mathbb{R}^n_{++}$\,, the \ref{ddgfactscale} bound is an upper bound for \ref{CGMESP}.
\end{theorem}
\begin{proof}
Let $\hat x \in \{0,1\}^n$ be any feasible solution to \ref{CGMESP}, with support $S \subset N$, $|S|=s$. By Lemma \ref{lem:Lee-scale},
\[
\textstyle\sum_{\ell=1}^t \log(\lambda_\ell(C_{S,S})) \leq \textstyle\sum_{\ell=1}^t \log(\lambda_\ell((\Diag(\Upsilon)^{\scriptscriptstyle 1/2} C \Diag(\Upsilon)^{\scriptscriptstyle 1/2})_{S,S})) - \textstyle\sum_{\ell=s-t+1}^s \log(\lambda_\ell(\Diag(\Upsilon)_{S,S})).
\]
Because $\Diag(\Upsilon)^{\scriptscriptstyle 1/2} C \Diag(\Upsilon)^{\scriptscriptstyle 1/2} = (\Diag(\Upsilon)^{\scriptscriptstyle 1/2} F)(\Diag(\Upsilon)^{\scriptscriptstyle 1/2} F)^\top$, evaluating the objective function of \ref{ddgfact} 
at $\hat x$ yields
\[
\textstyle\sum_{\ell=1}^t \log(\lambda_\ell((\Diag(\Upsilon)^{\scriptscriptstyle 1/2} C \Diag(\Upsilon)^{\scriptscriptstyle 1/2})_{S,S})) \leq \Gamma_t(F^\top \Diag(\Upsilon)^{\scriptscriptstyle 1/2} \Diag(\hat x) \Diag(\Upsilon)^{\scriptscriptstyle 1/2} F).
\]
Combining the two inequalities above, we have
\begin{equation}\label{eq:combined-bound}
\textstyle\sum_{\ell=1}^t \log(\lambda_\ell(C_{S,S})) \leq \Gamma_t(F^\top \Diag(\Upsilon)^{\scriptscriptstyle 1/2} \Diag(\hat x) \Diag(\Upsilon)^{\scriptscriptstyle 1/2} F) - \textstyle\sum_{\ell=s-t+1}^s \log(\lambda_\ell(\Diag(\Upsilon)_{S,S})).
\end{equation}
As $\Diag(\Upsilon)$ is a diagonal matrix with positive entries, $\sum_{\ell=s-t+1}^s \log(\lambda_\ell(\Diag(\Upsilon)_{S,S}))$ is equal to the sum of the $t$ least elements in $\{\log(\gamma_i): i \in S\}$. This quantity coincides with the optimal value of the linear program
\begin{equation}\label{eq:subprob_gamma_gscale}
\min\!\left\{\textstyle\sum_{i \in N} \log(\gamma_i)\, y_i \,:\, 0 \leq y \leq \hat x,\ \mathbf{e}^\top y = t\right\}.
\end{equation}
Let $\hat y$ be an optimal solution of \eqref{eq:subprob_gamma_gscale}. Substituting into \eqref{eq:combined-bound} yields
\[
\textstyle\sum_{\ell=1}^t \log(\lambda_\ell(C_{S,S})) \leq \Gamma_t(F^\top \Diag(\Upsilon)^{\scriptscriptstyle 1/2} \Diag(\hat x) \Diag(\Upsilon)^{\scriptscriptstyle 1/2} F) - \textstyle\sum_{i \in N} \log(\gamma_i)\, \hat y_i = \fgscalegfact(\hat x, \hat y; \Upsilon).
\]
As $(\hat x, \hat y)$ is feasible for \ref{ddgfactscale}\,, the result follows.
\end{proof}

\begin{remark}
    A g-scaled \ref{glinx} bound could also be derived by applying the argument of Theorem \ref{thm:Gamma-upper-bound} together with Lemma \ref{lem:Lee-scale}. However, this approach would yield a weaker bound, as it would require introducing an auxiliary variable $y$ into the optimization problem; from the definition of $\Pnst$, the choice $y = \diag(X)$ would be feasible, so the resulting bound would be dominated by \ref{gscaleglinx}\,.
\end{remark}

\cite{gscale} demonstrated that the g-scaling for DDFact admits a stationary point at $\Upsilon = \mathbf{e}$ in the absence of side constraints $Ax\leq b$. More recently, by establishing that the g-scaled factorization bound is convex in $\log(\Upsilon)$, \cite{Fatma2026} established that $\Upsilon = \mathbf{e}$ is in fact  optimal  in this setting.  Adapting the approach used for \ref{CMESP} in \cite[Theorem 6, part (iv)]{gscale}, we demonstrate,  by leveraging Lemma \ref{lem:optimal_xhat_diagf} below, that the analogous result for \ref{ddgfactscale} follows from a considerably simpler argument. In particular,  Theorem \ref{thm:UpsOptOnes} establishes that $\Upsilon=\mathbf{e}$ is optimal for \ref{ddgfactscale} when no side constraints $Ax\leq b$ are present. 
We begin by deriving the gradient of the g-scaled objective function, with the proof proceeding similarly to that of \cite[Theorem 6, part (iv)]{gscale}.

\begin{lemma}\label{lem:gradUps_gscaleddgfact}
    Let $(\hat{x},\hat{y})$ be a feasible solution to {\rm\ref{ddgfactscale}} for a fixed g-scaling parameter $
    \hat\Upsilon \in \mathbb{R}^n_{++}$\,. Define
     \begin{equation*}
         M_{\Upsilon}(\Upsilon;{x}) := F^\top \Diag(\Upsilon)^{\scriptscriptstyle 1/2}\Diag( x)\Diag(\Upsilon)^{\scriptscriptstyle 1/2} F. 
     \end{equation*}
     Let $\hat{r}:=\rank(M(\hat\Upsilon;\hat{x}))$ and let $\textstyle\sum_{\ell=1}^{k} \hat\lambda_{\ell} \hat u_{\ell} \hat u_{\ell}^{\top}$ be a spectral decomposition of $M(\hat\Upsilon;\hat{x})$ with $k \geq \hat{r}$. Let
       $\hat{\Theta}:=\sum_{\ell=1}^{k} {\hat \beta}_\ell \hat{u}_\ell\hat{u}_\ell^\top$\,, where
\begin{equation}\label{defbetaa}
\hat{\beta}_\ell:=\left\{
\begin{array}{ll}
        \textstyle 1/\hat{\lambda}_\ell\,,
       &~1\leq \ell\leq \hat{\iota};\\
     1/\hat{\delta},&~\hat{\iota}<\ell\leq \hat{r};\\
     (1+\epsilon)/\hat{\delta},&~\hat{r}<\ell\leq k,
\end{array}\right.
\end{equation}
 for any $\epsilon>0$, where $\hat{\iota}$ is the unique integer defined  in Lemma \ref{Ni13} for $\lambda_\ell=\hat{\lambda}_\ell$\,, and
$
\hat \delta:=\frac{1}{t-\hat \iota}\sum_{\ell=\hat \iota+1}^{k}\hat \lambda_\ell
$\,. Then, the function $f_{\mbox{\normalfont\protect\tiny $\Gamma$}}(\hat{x},\hat{y};\Upsilon)$ is differentiable in $\Upsilon$ with gradient
    \[
    g_{\mbox{\normalfont\protect\tiny $\Upsilon$}}(\hat{x},\hat{y};\hat\Upsilon) := \hat{x}\circ \diag(F\hat\Theta F^\top) - \Diag(\hat\Upsilon)^{-1}\hat{y}.
    \]    
\end{lemma}

\begin{lemma}[\protect{\cite[Lemma 20]{GMESP_Alg}}]\label{lem:optimal_xhat_diagf}
   Let $\hat{x}$ be an optimal  solution of {\rm \ref{ddgfact}}, 
   where there are no side constraints $Ax\leq b$. Let  
     $F(\hat x)= F^\top \Diag(\hat x) F =: \sum_{\ell=1}^{k} \hat\lambda_{\ell} \hat u_{\ell} \hat u_{\ell}^{\top}$ be a spectral decomposition of $F(\hat x)$. Let
       $\hat{\Theta}:=\sum_{\ell=1}^{k} {\hat \beta}_\ell \hat{u}_\ell \hat{u}_\ell^\top$\,, where $\hat\beta$ is defined in \eqref{defbetaa}. 
Then, for every $i,j\in N$, 
we have
\begin{enumerate}
        \item[\rm($i$)] 
        $\diag(F\hat\Theta F^\top)_{i} \geq \diag(F\hat\Theta F^\top)_{j}\,$, if $\hat{x}_i > \hat{x}_j$\,,
        \item[\rm($ii$)] 
        $\diag(F\hat\Theta F^\top)_{i} = \diag(F\hat\Theta F^\top)_{j}\,$, if $\hat{x}_i\,, \hat{x}_j\in(0,1)$.
    \end{enumerate}
\end{lemma}

\begin{theorem}\label{thm:UpsOptOnes}
    Let $\hat{\Upsilon}:=\mathbf{e}$, and consider {\rm\ref{ddgfactscale}} and {\rm\ref{ddgfact}} with no side constraints $Ax\leq b$. Let $\hat{x}$ be an optimal solution of {\rm\ref{ddgfact}}. Let  
     $F(\hat x)= F^\top \Diag(\hat x) F =: \sum_{\ell=1}^{k} \hat\lambda_{\ell} \hat u_{\ell} \hat u_{\ell}^{\top}$ be a spectral decomposition of $F(\hat x)$ and define $\hat{\Theta}:=\sum_{\ell=1}^{k} {\hat \beta}_\ell \hat{u}_\ell \hat{u}_\ell^\top$\,, where $\hat\beta$ is defined in \eqref{defbetaa}.  Let $\hat{y}:=\hat{x}\circ \diag(F\hat{\Theta}F^\top)$. Then,  $(\hat{x},\hat{y})$ is an optimal solution to {\rm\ref{ddgfactscale}} for $\hat\Upsilon=\mathbf{e}$ and
    \[
    g_{\Upsilon}(\hat{x},\hat{y};\hat{\Upsilon}) = \mathbf{0}.
    \]
\end{theorem}

\begin{proof}
    We first verify that, given $\hat{x}$, the constructed $\hat{y}$ ensures that ($\hat{x},\hat{y}$) is feasible for $\ref{ddgfactscale}$\,. As $\hat x$ is optimal for \ref{ddgfact}, we have $\mathbf{e}^\top \hat{x} = s$ and $\hat{x} \in [0,1]^n$. Furthermore,  $\hat y$ satisfies
    \begin{itemize}
        \item $\mathbf{e}^\top \hat{y} = t$:  from \cite[Eq. 16]{GMESP_Alg}, we have that $\hat{x}^\top  \diag(F\hat{\Theta}F^\top) = t$ and therefore it follows that
        \[
        \mathbf{e}^\top \hat{y}=\textstyle\sum_{i\in N}\hat{x}_i \diag(F\hat{\Theta}F^\top)_i = \hat{x}^\top \diag(F\hat{\Theta}F^\top) = t.
        \]
        
        \item $\hat{y} \geq \mathbf{0}$: for $i \in N$, we have $\hat{y}_i = \hat{x}_i (F\hat\Theta F^\top)_{ii} = \hat{x}_i (F_{i\cdot} \hat\Theta F_{i\cdot}^\top) \geq 0$
         because $\hat\Theta \succeq 0$ and $\hat{x}_i \geq 0$.
         
        \item $\hat{y} \leq \hat{x}$: first, we will prove that $\hat{y} \leq \mathbf{e}$. Let $E:= \Diag(\hat{x})^{\scriptscriptstyle 1/2}F$, and consider the (compact) singular value decomposition of $E=: U\Sigma V^\top$ with $U^\top U = V^\top V = I_{\hat{r}}$\,, where   $\hat{r}:= \rank(E)$. 
            Because $E^\top E = F^\top \Diag(\hat x)F = V\Sigma^2V^\top$, then from the construction of $\hat\Theta$ we have $V^\top \hat\Theta V = \Diag(\hat\beta_1,\dots,\hat\beta_{\hat r})$, and therefore
\[
M:=\Sigma V^\top \hat\Theta V\Sigma
  =\Diag(\hat\beta_1\hat\lambda_1,\dots,\hat\beta_{\hat r}\hat\lambda_{\hat r}).
\]
From Lemma \ref{Ni13} and the definition of $\hat\beta$, it follows that $0\leq \hat\beta_{i}\hat\lambda_i\leq 1$ for $i=1,\dots,{\hat r}$, and hence $0\preceq M \preceq I_{\hat{r}}$\,. Then, we have that
            \[
            UMU^\top = U\Sigma V^\top \hat\Theta V\Sigma U^\top =E\hat\Theta E^\top= \Diag(\hat{x})^{\scriptscriptstyle 1/2}F\hat\Theta F^\top \Diag(\hat{x})^{\scriptscriptstyle 1/2},
            \]
           and therefore $\diag(UMU^\top) = \hat{y}$.
        Because $0\preceq M\preceq I_{\hat{r}}$\,, from \cite[Obs. 7.7.2]{HJBook} we have that  $UMU^\top \preceq UU^\top \preceq I_n$\,. This implies that $\hat{y} \leq \mathbf{e}$.
                
        Next, we consider the partition of $N$ into three subsets:   $\mathcal{I}_0:=\{i\in N\,:\,\hat{x}_i = 0\}$,  $\mathcal{I}_1:=\{i\in N\,:\,\hat{x}_i = 1\}$ and   $\mathcal{I}_f:=\{i\in N\,:\,\hat{x}_i \in (0,1)\}$.
        \begin{enumerate}
            \item If $i \in \mathcal{I}_0$,  $\hat{y}_i \leq \hat{x}_i$ holds by construction, because  $\hat{y}_i = 0 = \hat{x}_i$\,. 
            \item   If $i \in \mathcal{I}_1$,  $\hat{y}_i \leq \hat{x}_i$ follows because $\hat{y}_i \leq 1 = \hat{x}_i$\,, as we proved above. 
        
           \item If $i \in \mathcal{I}_f$,
            let $\hat{d}:= \diag(F\hat\Theta F^\top)_j$\,, for every $j \in \mathcal{I}_f$ (this is well defined, due to Lemma \ref{lem:optimal_xhat_diagf}, part~(ii)). 
           We consider two cases.
           \begin{enumerate}
           \item Suppose that $|\mathcal{I}_1|=0$.  Then $\hat{y} = \hat{d} \hat{x}$ and so, as we already proved that $\mathbf{e}^\top \hat{y}=t$, we have that  $t = \mathbf{e}^\top \hat{y} = \hat{d} (\mathbf{e}^\top \hat{x}) = \hat{d}s$. Therefore, $\hat{d}=t/s\leq 1$, and so $\hat{y}_i =  \hat{d} \hat{x}_i \leq \hat{x}_i$\,.
            
            \item Suppose that $|\mathcal I_1| > 0$. Let $\hat\jmath\in \mathcal{I}_1$. Because $\hat{x}_{\hat\jmath} = 1$, we have $\diag(F\hat\Theta F^\top)_{\hat\jmath} =\diag(F\hat\Theta F^\top)_{\hat\jmath} \hat{x}_{\hat\jmath} =  \hat{y}_{\hat\jmath} \leq 1$, as we already proved that $\hat{y} \leq \mathbf{e}$. Moreover, $\hat{x}_{\hat\jmath} > \hat{x}_i$\,, so by Lemma \ref{lem:optimal_xhat_diagf}, part (a), we have $\diag(F\hat\Theta F^\top)_i \leq \diag(F\hat\Theta F^\top)_{\hat\jmath} \leq 1$. Then, 
            $
            \hat{y}_i =\hat{x}_i\diag(F\hat\Theta F^\top)_i \leq \hat{x}_i$\,.
           \end{enumerate}
        \end{enumerate}
    \end{itemize}
    For $\hat\Upsilon=\mathbf{e}$, we have $
f_{\mbox{\normalfont\protect\tiny $\Gamma$}}(\mathbf e;x,y)
=
\Gamma_t(F^\top \Diag(x)F),$
so the objective-function value is independent of $y$. Hence, because $\hat x$ is optimal
for \ref{ddgfact} and $\hat y=\hat{x}\circ \diag(F\hat\Theta F^\top) $ is feasible for
\ref{ddgfactscale}, the pair $(\hat x,\hat y)$ is optimal for
\ref{ddgfactscale} at $\hat\Upsilon=\mathbf{e}$. Finally, from Lemma \ref{lem:gradUps_gscaleddgfact}, we have  
$
    g_{\mbox{\normalfont\protect\tiny $\Upsilon$}}(\hat{x},\hat{y};\mathbf{e}) = \hat{x}\circ \diag(F\hat\Theta F^\top) - \hat{y} = \mathbf{0}.
$
Using arguments analogous to those  in \cite[Proposition 10]{Fatma2026}, it is straightforward to show that $\fgscalegfactthing(\Upsilon;\hat{x},\hat{y})$ is convex in $\log(\Upsilon)$. Moreover, as the gradient with respect to $\Upsilon$ vanishes at $\hat\Upsilon=\mathbf{e}$, which implies that the gradient with respect to $\log(\Upsilon)$ also vanishes at this point, we conclude that $\hat\Upsilon=\mathbf{e}$ is optimal. 
\end{proof}


\subsection{Optimizing generalized scaling}

Adapting the BFGS methodology from
\cite[Sec. 5]{gscale} for the generalized scaling for \ref{CMESP},  we describe an optimization procedure for computing a g-scaling
parameter $\hat\Upsilon$ that  minimizes the 
\ref{ddgfactscale} bound and  locally minimizes the \ref{gscaleglinx} bound. We present
the procedure for \ref{gscaleglinx}\,; the procedure for \ref{ddgfactscale}
is analogous, with $(\hat{x},\hat{y})$ replacing $\hat{X}$ and
$f_{\mbox{\normalfont\protect\tiny $\Gamma$}}(\hat{x},\hat{y};\hat\Upsilon)$ replacing
$\fgscaleglinx(\hat{X};\hat\Upsilon)$.

The optimization procedure proceeds iteratively. At each iteration, for
the current value of $\hat\Upsilon$, the matrix
$\hat{X}:=\hat{X}(\hat\Upsilon)$ is first obtained by solving
\ref{gscaleglinx} for the given $\hat\Upsilon$. The function
$\fgscaleglinx(\hat{X};\hat\Upsilon)$ is then locally minimized
with respect to $\log(\hat\Upsilon)$, treating $\hat{X}$ as fixed when
computing a gradient and performing the BFGS update. Although $\hat{X}$
is recomputed at every iteration as a function of $\hat\Upsilon$, it is
treated as constant during the update of $\hat\Upsilon$. This separation
is essential to ensure that, at every iteration, the
objective value of $\fgscaleglinx(\hat{X};\hat\Upsilon)$ remains
an upper bound for \ref{CGMESP}, and it follows the methodology of
\cite[Sec. 5]{gscale}.

We have the following remark:
We could also investigate a g-scaling analogue of the \ref{GNLP-Id} and \ref{GNLP-Id-comp} bounds. This approach was not considered in \cite{gscale}. For the NLP bound, the o-scaling parameter plays a more delicate role, because it influences the convexity of the relaxation. As a result, extending g-scaling to this setting is not straightforward. However, for a related problem, \cite{AugNLP_MERSP} proposed a diagonal-scaling approach for the NLP-Id bound, which suggests a possible direction for future work.

Finally, we present a simple example demonstrating that, in general, the \ref{gscaleglinx} bound is not convex in $\log(\Upsilon)$.

\begin{example}
    Let $n:=3$, $s:=2$, $t:=1$. Let $C:=vv^\top$ with $v:=(1,-1,1)^\top$, and $\hat{x}:=(1,1,0)^\top$, $\hat{X}:=ww^\top$ with $w:=\tfrac{1}{\sqrt 2}(1,1,0)^\top$. Note that $(\hat{x},\hat{X}) \in \Pnst$. Let $\Psi:=\log(\Upsilon)$ and $h(\Psi):=\fgscaleglinx(\exp(\Psi);\hat{X})$.  Take $\Psi^1:=(1,0,0)^\top$, $\Psi^2:=(3,0,0)^\top$, and $\theta:=1/2$, so that $\theta\Psi^1+(1-\theta)\Psi^2=(2,0,0)^\top$. Then,
\begin{itemize}
    \item $h\left(\theta\Psi^1+(1-\theta)\Psi^2\right) =  2\log(\exp(2)-1)-2-\log(2)\approx 1.0160$;
    \item $\theta h(\Psi^1)+(1-\theta)h(\Psi^2)=\log(\exp(1)-1)+\log(\exp(3)-1)-2-\log(2)\approx 0.7971$.
\end{itemize}
As $h\left(\theta\Psi^1+(1-\theta)\Psi^2\right)>\theta h(\Psi^1)+(1-\theta)h(\Psi^2)$, we conclude that $h$ is not convex in $\Psi$; equivalently, $\fgscaleglinx$ is not convex in $\log(\Upsilon)$.\hfill $\clubsuit$
\end{example}


\section{Down branching}\label{sec:BB}

In this section, we consider \ref{CGMESP} B\&B subproblems, where we have fixed variables $x_i = 0$ for $i \in F_0\subset N$, with $F_0\neq \emptyset$, i.e., ``down branching'' (see \cite[Section 3]{Johnson}). There are two equivalent ways of doing this. We can simply eliminate each such variable $x_i$\,, the associated rows and columns of $X$, the associated columns of
$A$, and the associated rows and columns of $C$. Alternatively, we can extend the $Ax\leq b$ constraints with $f:=|F_0|$ further constraints $x_i\leq 0$, for $i\in F_0$\,. While equivalent for \ref{CGMESP}, we will see that these two methods do not always produce equal bounds for a given convex-relaxation based bounding method.
We will consider applying the \ref{glinx} bound, the \ref{GNLP-Id} bound, and the \ref{GNLP-Id-comp} bound to such B\&B subproblems.

Formally, we consider the two possibilities:
\begin{itemize}
    \item[\rm($i$)] the matrix $C$ is replaced by its principal submatrix $C_{N\setminus F_0,N\setminus F_0}$\,, arriving at
    \begin{equation}\label{C-F0-out}\tag{$\mathcal{R}$}
       \mbox{CGMESP}(C_{N\setminus F_0,N\setminus F_0}\,,s,t,A_{\cdot\, N\setminus F_0}\,,b)  
    \end{equation}

    \item[\rm($ii$)] we employ additional side constraints $\mathcal{A}x\leq \mathbf{0}$ with $\mathcal{A} \in \mathbb{R}^{f \times n}$ such that  $\mathcal{A}_{\cdot \,F_0} := I_f$ and $\mathcal{A}_{\cdot \,N\setminus F_0} := 0$, arriving at 
    \begin{equation}\label{A-F0-out}\tag{$\mathcal{C}$}
        \mbox{CGMESP}(C,s,t,\left(\begin{smallmatrix}
            A\\
            \mathcal{A}
        \end{smallmatrix}\right),\left(\begin{smallmatrix}
            b\\
            \mathbf{0}
        \end{smallmatrix}\right)).  
    \end{equation}
\end{itemize}

\ref{C-F0-out} indicates that the subproblem is ``reduced''
and \ref{A-F0-out} indicates that the subproblem is 
``constrained''. 
With the following example, we demonstrate that, 
the bounds for  \ref{C-F0-out} can be different from the bounds for \ref{A-F0-out}.

\begin{example}\label{ex:RvsC}
    Consider the GMESP instance with $C := \left(\begin{smallmatrix}
        4 & 2 & 1 & 1\\
        2 & 2 & 1 & 0\\
        1 & 1 & 1 & 0\\
        1 & 0 & 0 & 2
    \end{smallmatrix}\right)$, $s := 2$ and $t := 1$. Let  $F_0 := \{1\}$. Then, we have 

\begin{center}
\begin{tabular}{c|ccc}
             procedure   & $\zglinx$   & $\zop$  &   $\zopcomp$  \\[2pt] \hline
\ref{C-F0-out}             & 1.148            & 0.962  & 1.545            \\ \hline
\ref{A-F0-out}          & 1.322            & 1.058  & 1.636              
\end{tabular}
\end{center}
 \hfill $\clubsuit$
\end{example}

With the following theorem, we establish that the glinx bound for \ref{C-F0-out} always dominates the glinx bound for \ref{A-F0-out}\,, as was illustrated (strictly) in Example \ref{ex:RvsC}. 

\begin{theorem}\label{thm:Rbetter1}
    $
        \zglinxthing(C_{N\setminus F_0,N\setminus F_0},s,t,A_{\cdot \, N\setminus F_0},b) \leq  \zglinxthing(C,s,t,\left(\begin{smallmatrix}
            A\\
            \mathcal{A}
        \end{smallmatrix}\right),\left(\begin{smallmatrix}
            b\\
            \mathbf{0}
        \end{smallmatrix}\right)).
    $
\end{theorem}

\begin{proof}
Assume without loss of generality that
$C = \left(\begin{smallmatrix}
   C_{N\setminus F_0,N\setminus F_0} & C_{N\setminus F_0,F_0}\\
    C_{F_0,N\setminus F_0}& C_{F_0,F_0}
\end{smallmatrix}\right)$. Let 
    \[
    \hat{x} = \left(\tilde{x}:=\hat{x}_{N\setminus F_0}, \hat{x}_{F_0}:=0\right), \quad \hat{X} = \begin{pmatrix}
        \tilde{X}:= \hat{X}_{N\setminus F_0,N\setminus F_0} & \hat{X}_{N\setminus F_0, F_0}:= 0\\
        \hat{X}_{F_0,N\setminus F_0}:= 0 & \hat{X}_{F_0,F_0}:= 0
    \end{pmatrix}
    \]   
   describe a bijection between the set of feasible solutions $(\hat{x},\hat{X})$ of $\mathcal{C}$ and the set of feasible solutions $(\tilde{x},\tilde{X})$ of $\mathcal{R}$.
Let $L_{\mathcal{R}}:= C_{N\setminus F_0,N\setminus F_0}\tilde{X}C_{N\setminus F_0,N\setminus F_0} + I_{n-f} - \tilde{X}$, and let $L_{\mathcal{C}} := C\hat{X}C + I_n -\hat{X}$.  Assume that $L_{\mathcal{R}}\succ 0$ and $L_{\mathcal{C}}\succ 0$. Then, we have that
\begin{itemize}
    \item the objective value of \ref{glinx} for \ref{C-F0-out} at $(\tilde{x},\tilde{X})$ is equal to $\textstyle\frac{1}{2}\ldet(L_{\mathcal{R}})$.
    \item the objective value of \ref{glinx} for \ref{A-F0-out} at $(\hat{x},\hat{X})$ is equal to $\textstyle\frac{1}{2}\ldet(L_{\mathcal{C}})$.
\end{itemize}
Note that 
\[
L_{\mathcal{C}} = \begin{pmatrix}
    L_{\mathcal{R}} & C_{N\setminus F_0,N\setminus F_0} \tilde{X}C_{N\setminus F_0,F_0}\\[5pt]
    C_{F_0,N\setminus F_0} \tilde{X} C_{N\setminus F_0,N\setminus F_0} & I_{f} + C_{F_0,N\setminus F_0}\tilde{X}C_{N\setminus F_0, F_0}
\end{pmatrix},
\]
so from the Schur-complement determinant formula, we have
\begin{align*}
   \ldet(L_{\mathcal{C}}) &=  \ldet(L_{\mathcal{R}}) + \ldet(I_{f} + C_{F_0,N\setminus F_0}\tilde{X}C_{N\setminus F_0, F_0} - C_{F_0,N\setminus F_0} \tilde{X} C_{N\setminus F_0,N\setminus F_0}L_{\mathcal{R}}^{-1} C_{N\setminus F_0,N\setminus F_0} \tilde{X}C_{N\setminus F_0,F_0})\\
    &= \ldet(L_{\mathcal{R}}) +\ldet(I_{f} + C_{F_0,N\setminus F_0}\tilde{X}^{\scriptscriptstyle 1/2}(I_{n-f} - \tilde{X}^{\scriptscriptstyle 1/2}C_{N\setminus F_0,N\setminus F_0} L_{\mathcal{R}}^{-1}C_{N\setminus F_0,N\setminus F_0} \tilde{X}^{\scriptscriptstyle 1/2})\tilde{X}^{\scriptscriptstyle 1/2}C_{N\setminus F_0,F_0})\\
    &= \ldet(L_{\mathcal{R}}) +\ldet(I_{f} + C_{F_0,N\setminus F_0}\tilde{X}^{\scriptscriptstyle 1/2}\Psi\tilde{X}^{\scriptscriptstyle 1/2}C_{N\setminus F_0,F_0}),
\end{align*}
where $\Psi := I_{n-f} - \tilde{X}^{\scriptscriptstyle 1/2}C_{N\setminus F_0,N\setminus F_0} L_{\mathcal{R}}^{-1}C_{N\setminus F_0,N\setminus F_0} \tilde{X}^{\scriptscriptstyle 1/2}$.

Then, in order to conclude that $\ldet(L_{\mathcal{C}}) \geq \ldet(L_{\mathcal{R}})$, it suffices to demonstrate that $\Psi\succeq0$. Consider a matrix 
\[
\Upsilon := \begin{pmatrix}
    L_\mathcal{R} & C_{N\setminus F_0,N\setminus F_0}\tilde{X}^{\scriptscriptstyle 1/2} \\
   \tilde{X}^{\scriptscriptstyle 1/2}C_{N\setminus F_0,N\setminus F_0}   & I_{n-f}
\end{pmatrix}.
\]
Note that $\Psi$ is the Schur complement of $L_{\mathcal{R}}$ in $\Upsilon$  and  $L_{\mathcal{R}} - C_{N\setminus F_0,N\setminus F_0}\tilde{X}C_{N\setminus F_0,N\setminus F_0}$ is the Schur complement of $I_{n-f}$ in $\Upsilon$.
Then, from \cite[Theorem 1.12]{SchurBook}, we have that 
\[
\Psi \succeq 0 \Leftrightarrow \Upsilon \succeq 0 \Leftrightarrow L_{\mathcal{R}} - C_{N\setminus F_0,N\setminus F_0}\tilde{X}C_{N\setminus F_0,N\setminus F_0} \succeq 0.
\]
Because $L_{\mathcal{R}} - C_{N\setminus F_0, N\setminus F_0}\tilde{X}C_{N\setminus F_0,N\setminus F_0} = I_{n-f} - \tilde{X}$ and $\tilde{X} \preceq \Diag(\tilde{x}) \preceq I_{ n-f}$\,, the result follows. 
\end{proof}

Next, we demonstrate that the GNLP-Id (resp., GNLP-Id$_{\rm{c}}$) bound for \ref{C-F0-out} always dominates the  GNLP-Id (resp., GNLP-Id$_{\rm{c}}$) for \ref{A-F0-out}\,, as was illustrated (strictly) in the example. 
\begin{lemma}\label{lem:natural_to_nlp_modified}
Let $C\in\mathbb{S}^n_{+}$\,, $0 < t < n$, $0 \preceq \hat{X} \preceq I_n$\,, with $\tr(\hat{X})=t$, and  $\gamma>0$\,. Define $\Theta(\gamma):= I_n + \hat{X}^{\scriptscriptstyle 1/2}(\gamma C-I_n)\hat{X}^{\scriptscriptstyle 1/2}$ and assume that $\Theta(\gamma) \succ 0$. Let 
     \begin{equation*} 
         h(\gamma) := \ldet( \Theta(\gamma)) - t\log(\gamma).
     \end{equation*}
Then $h(\gamma)$ is {nonincreasing} in $\gamma$ on $\left(0,1/\lambda_{\max}\right]$, and $h(\gamma)$ is {nondecreasing} in $\gamma$ on $\left[1/\lambda_{\min},\infty\right)$.
\end{lemma}

\begin{proof}
We have  
\begin{equation*} 
\begin{array}{ll}
    h'(\gamma) &= \gamma^{-1}   \tr(\Theta(\gamma)^{-1}\gamma \hat{X}^{\scriptscriptstyle 1/2} C\hat{X}^{\scriptscriptstyle 1/2} ) - \gamma^{-1}t\\
    &=\gamma^{-1}\tr(\Theta(\gamma)^{-1}(\Theta(\gamma) - (I_n-\hat{X}))) -  \gamma^{-1}\tr(\hat{X}) \\
     &=\gamma^{-1}\tr(I_n - \Theta(\gamma)^{-1} (I_n-\hat{X})) -  \gamma^{-1}\tr(\hat{X}) \\
     &=\gamma^{-1}\tr((I_n - \hat{X}) - \Theta(\gamma)^{-1} (I_n-\hat{X})) \\
    &= \gamma^{-1}\tr((I_n - \Theta(\gamma)^{-1})(I_n - \hat{X}))\\
     &= \gamma^{-1}\tr((I_n - \hat{X})^{\scriptscriptstyle 1/2}(I_n-\Theta(\gamma)^{-1})(I_n - \hat{X})^{\scriptscriptstyle 1/2}).
\end{array}
\end{equation*}
For $\gamma \in (0, 1/\lambda_{\max}]$\,, we have $\gamma C \preceq I_n \Rightarrow \Theta(\gamma) \preceq I_n \Rightarrow \Theta(\gamma)^{-1} \succeq I_n\,,$ so we conclude that $h$ is nonincreasing on the interval $(0,1/\lambda_{\max}]$.  For $\gamma \in [1/\lambda_{\min},\infty)$, we have $\gamma C \succeq I_n \Rightarrow \Theta(\gamma) \succeq I_n \Rightarrow \Theta(\gamma)^{-1} \preceq I_n\,,$ so we conclude that $h$ is nondecreasing on the interval $[1/\lambda_{\min},\infty)$. 
\end{proof}

\begin{theorem}\label{thm:Rbetter2}
     \phantom{.}
     \begin{enumerate}
         \item[\rm($i$)] $\zopthing(C_{N\setminus F_0,N\setminus F_0}\,,s,t,A_{\cdot \,N\setminus F_0}\,,b) \leq  \zopthing(C,s,t,\left(\begin{smallmatrix}
            A\\
            \mathcal{A}
        \end{smallmatrix}\right),\left(\begin{smallmatrix}
            b\\
            \mathbf{0}
        \end{smallmatrix}\right))$.
         \item[\rm($ii$)] $\zopcompthing(C_{N\setminus F_0,N\setminus F_0}\,,s,t,A_{\cdot \,N\setminus F_0}\,,b) \leq  \zopcompthing(C,s,t,\left(\begin{smallmatrix}
            A\\
            \mathcal{A}
        \end{smallmatrix}\right),\left(\begin{smallmatrix}
            b\\
            \mathbf{0}
        \end{smallmatrix}\right))$.
     \end{enumerate}
\end{theorem}

\begin{proof}
        Assume without loss of generality that
$C = \left(\begin{smallmatrix}
   C_{N\setminus F_0,N\setminus F_0} & C_{N\setminus F_0,F_0}\\
    C_{F_0,N\setminus F_0}& C_{F_0,F_0}
\end{smallmatrix}\right)$. Let 
    \[
    \hat{x} = \left(\tilde{x}:=\hat{x}_{N\setminus F_0}, \hat{x}_{F_0}:=0\right), \quad \hat{X} = \begin{pmatrix}
        \tilde{X}:= \hat{X}_{N\setminus F_0,N\setminus F_0} & \hat{X}_{N\setminus F_0, F_0}:= 0\\
        \hat{X}_{F_0,N\setminus F_0}:= 0 & \hat{X}_{F_0,F_0}:= 0
    \end{pmatrix}
    \]   
   describe a bijection between the set of feasible solutions 
      $(\tilde{x},\tilde{X})$ of  \ref{C-F0-out}
   and the set of feasible solutions 
   $(\hat{x},\hat{X})$ of \ref{A-F0-out}.

\medskip

\noindent ($i$):~ Let $\tilde{\lambda}_{\max} := \lambda_{1}(C_{N\setminus  F_0 ,N\setminus  F_0 })$.  Consider any factorizations 
$HH^\top  = I_{n} - (1/{\lambda_{\max}})C$ and 
$H_{\scriptscriptstyle\mathcal{R}}H_{\scriptscriptstyle\mathcal{R}}^\top  = I_{n-f} - (1/{\tilde\lambda_{\max}})C_{N\setminus F_0, N\setminus F_0 }$ where $H \in \mathbb{R}^{n\times n}$ and $H_{\scriptscriptstyle\mathcal{R}} \in \mathbb{R}^{(n-f)\times (n-f)}$. Let 
$L_{\mathcal{C}} := I_n - H^\top \hat{X} H$, and let $L_{\mathcal{R}} := I_{n-f} - H_{\scriptscriptstyle\mathcal{R}}^\top \tilde{X} H_{\scriptscriptstyle\mathcal{R}}$.  Assume that $L_{\mathcal{R}}\succ 0$ and $L_{\mathcal{C}}\succ 0$. 
Then, we have that
\begin{itemize}
    \item the objective value of \ref{GNLP-Id} for \ref{C-F0-out} at $(\tilde{x},\tilde{X})$ is equal to $\ldet(L_{\mathcal{R}}) + t\log(\tilde\lambda_{\max})$.
    \item the objective value of \ref{GNLP-Id} for \ref{A-F0-out} at $(\hat{x},\hat{X})$ is equal to $\ldet(L_{\mathcal{C}})+ t\log(\lambda_{\max})$\,.
\end{itemize}
We have
        \begin{align*}
        \ldet(L_{\mathcal{C}}) + t\log(\lambda_{\max})
        &=
        \ldet(I_n - H^\top \hat{X} H) + t\log(\lambda_{\max})\\
        &=  \ldet(I_n - H_{(N\setminus F_0)\, \cdot}^\top \tilde{X} H_{(N\setminus F_0)\, \cdot}) + t\log(\lambda_{\max})\\
        &=  \ldet(I_{n-f} -\tilde{X}^{\scriptscriptstyle 1/2}  H_{(N\setminus F_0)\, \cdot} H_{(N\setminus F_0)\, \cdot}^\top \tilde{X}^{\scriptscriptstyle 1/2} ) + t\log(\lambda_{\max})\\
        &=  \ldet(I_{n-f} -\tilde{X}^{\scriptscriptstyle 1/2}  (HH^\top)_{N\setminus F_0, N\setminus F_0 }\tilde{X}^{\scriptscriptstyle 1/2} ) + t\log(\lambda_{\max})\\
        &=  \ldet(I_{n-f} -\tilde{X}^{\scriptscriptstyle 1/2}  (I_{n-f} - (1/{\lambda_{\max}})C_{N\setminus F_0, N\setminus F_0 })\tilde{X}^{\scriptscriptstyle 1/2} ) + t\log(\lambda_{\max})\\
        &\geq  \ldet(I_{n-f} -\tilde{X}^{\scriptscriptstyle 1/2}  (I_{n-f} - (1/{\tilde\lambda_{\max}})C_{N\setminus F_0, N\setminus F_0 })\tilde{X}^{\scriptscriptstyle 1/2} ) + t\log(\tilde\lambda_{\max})\\
        &=  \ldet(I_{n-f} -\tilde{X}^{\scriptscriptstyle 1/2} H_{\scriptscriptstyle\mathcal{R}}H_{\scriptscriptstyle\mathcal{R}}^\top\tilde{X}^{\scriptscriptstyle 1/2} ) + t\log(\tilde\lambda_{\max})\\
        &=  \ldet(I_{n-f} - H_{\scriptscriptstyle\mathcal{R}}^\top \tilde{X} H_{\scriptscriptstyle\mathcal{R}} ) + t\log(\tilde\lambda_{\max})=  \ldet(L_{\mathcal{R}} ) + t\log(\tilde\lambda_{\max}),
    \end{align*}
    where the inequality comes from Lemma \ref{lem:natural_to_nlp_modified} as  $\lambda_{\max} \geq \tilde\lambda_{\max}$\, (see, e.g.,  \cite[Theorem 4.3.8]{HJBook}).    
    
\medskip

\noindent ($ii$):~  Let $\tilde{\lambda}_{\min} := \lambda_{n-f}(C_{N\setminus  F_0 ,N\setminus  F_0 })$.  
        Consider any factorizations 
$GG^\top  = (1/{\lambda_{\min}})C - I_n$ and 
$G_{\scriptscriptstyle\mathcal{R}}G_{\scriptscriptstyle\mathcal{R}}^\top  = (1/{\tilde\lambda_{\min}})C_{N\setminus F_0, N\setminus F_0 } - I_{n-f}$ where $G \in \mathbb{R}^{n\times n}$ and $G_{\scriptscriptstyle\mathcal{R}} \in \mathbb{R}^{(n-f)\times (n-f)}$.  Let 
$L_{\mathcal{C}} := I_n + G^\top \hat{X} G$, and let $L_{\mathcal{R}} := I_{n-f} + G_{\scriptscriptstyle\mathcal{R}}^\top \tilde{X} G_{\scriptscriptstyle\mathcal{R}}$. 
Assume that $L_{\mathcal{R}}\succ 0$ and $L_{\mathcal{C}}\succ 0$. 
Then, we have that
        \begin{itemize}
            \item the objective value of \ref{GNLP-Id-comp} for \ref{C-F0-out} at $(\tilde{x},\tilde{X})$ is equal to $\ldet(L_{\mathcal{R}}) + t\log(\tilde\lambda_{\min})$.
         \item the objective value of \ref{GNLP-Id-comp} for \ref{A-F0-out} at $(\hat{x},\hat{X})$ is equal to $\ldet(L_{\mathcal{C}}) + t\log(\lambda_{\min})$.
\end{itemize}
We have
   \begin{align*}
        \ldet(L_{\mathcal{C}}) + t\log(\lambda_{\min})&=\ldet(I_n + G^\top \hat{X} G) + t\log(\lambda_{\min})\\
        &=  \ldet(I_{n} + G_{(N\setminus F_0)\, \cdot}^\top \tilde{X} G_{(N\setminus F_0)\, \cdot}) + t\log(\lambda_{\min})\\
        &=  \ldet(I_{n-f} +\tilde{X}^{\scriptscriptstyle 1/2}  G_{(N\setminus F_0)\, \cdot} G_{(N\setminus F_0)\, \cdot}^\top \tilde{X}^{\scriptscriptstyle 1/2} ) + t\log(\lambda_{\min})\\
        &=  \ldet(I_{n-f} +\tilde{X}^{\scriptscriptstyle 1/2}  (GG^\top)_{N\setminus F_0, N\setminus F_0 }\tilde{X}^{\scriptscriptstyle 1/2} ) + t\log(\lambda_{\min})\\
        &=  \ldet(I_{n-f} -\tilde{X}^{\scriptscriptstyle 1/2}  (I_{n-f} - (1/{\lambda_{\min}})C_{N\setminus F_0, N\setminus F_0 })\tilde{X}^{\scriptscriptstyle 1/2} ) + t\log(\lambda_{\min})\\
        &\geq \ldet(I_{n-f} -\tilde{X}^{\scriptscriptstyle 1/2}  (I_{n-f} - (1/{\tilde\lambda_{\min}})C_{N\setminus F_0, N\setminus F_0 })\tilde{X}^{\scriptscriptstyle 1/2} ) + t\log(\tilde\lambda_{\min})\\
        &= \ldet(I_{n-f} +\tilde{X}^{\scriptscriptstyle 1/2}  G_{\scriptscriptstyle\mathcal{R}} G_{\scriptscriptstyle\mathcal{R}}^\top\tilde{X}^{\scriptscriptstyle 1/2} ) + t\log(\tilde\lambda_{\min})\\
        &= \ldet(I_{n-f} +G_{\scriptscriptstyle\mathcal{R}}^\top\tilde{X} G_{\scriptscriptstyle\mathcal{R}}) + t\log(\tilde\lambda_{\min}) = \ldet(L_{\mathcal{R}}) + t\log(\tilde\lambda_{\min}),
    \end{align*}
     where the inequality comes from Lemma \ref{lem:natural_to_nlp_modified} as   $\lambda_{\min} \leq \tilde\lambda_{\min}$\, (see, e.g.,  \cite[Theorem 4.3.8]{HJBook}).    
\smallskip

    The result follows.
    \end{proof}

\begin{remark}
 With Example \ref{ex:RvsC} and 
 Theorems \ref{thm:Rbetter1} and \ref{thm:Rbetter2}, we saw the 
 superiority of the reduced formulation \ref{C-F0-out}
versus the constrained formulation \ref{A-F0-out}. In 
addition, the size of \ref{C-F0-out} is smaller than that of 
\ref{A-F0-out}, so typically we can expect that 
working with \ref{C-F0-out} versus
\ref{A-F0-out} should be better from the 
perspective of computational effort to solve relaxations. 
\end{remark}

\begin{remark}\label{rem:fixat1}
If we wish to fix variables $x_i = 1$ for $i \in F_1\subset N$, with $F_1 \neq \emptyset$, i.e., ``up branching'' (see \cite[Section 3]{Johnson}), we only know one way to do this, which is
analogous to \ref{A-F0-out}; we simply 
extend the $Ax\leq b$ constraints with $|F_1|$ further constraints $-x_i\leq -1$, for $i\in F_1$\,.
It would be of great interest to find a better method for fixing at one, just as we have established that \ref{C-F0-out}
is better than \ref{A-F0-out} for fixing at zero.
We do know how to do this for CMESP (see \cite[Chapter 2]{FLbook}, for example; also see \cite[Section 4.2]{MESP2DOPT}), using the Schur-complement determinant formula.
\end{remark}


\section{Numerical experiments}\label{sec:numexp}

We used the general-purpose solver KNITRO
14.0.0 (see \cite{KNITRO}), via the Julia wrapper KNITRO.jl v0.14.8,
and the conic solver Hypatia 11.0.1 (see \cite{coey2022solving}) to solve the semidefinite problems. 
We used the parameter settings for the solvers aiming at their best performance. For KNITRO, we used the following settings:
\texttt{CONVEX = true}, \texttt{FEASTOL} = $10^{-6}$ (feasibility tolerance), \texttt{OPTTOL} = $10^{-6}$ (absolute optimality tolerance),
\texttt{ALGORITHM} = 1 (Interior/Direct algorithm), 
\texttt{HESSOPT} = 2 (KNITRO computes a quasi-Newton BFGS Hessian). 

We ran our experiments on ``zebratoo'' (running Windows Server 2022 Standard):
two Intel Xeon Gold 6444Y 3.60GHz processors, with 16 cores each, and 128 GB of memory.

When comparing different upper bounds for our test instances, we present the gaps defined by the difference between the upper bounds for \ref{CGMESP}, and the lower bounds computed by a local-search heuristic.

We observed in our experiments that the SOC constraints are both weak and computationally expensive to incorporate into the interior-point solvers. From a theoretical standpoint, these constraints also appear to be ineffective; see Section \ref{sec:SOCP}.  Therefore, in this section, we present results for the  \ref{scaleglinx} bound, the \ref{gscaleglinx} bound, the  \ref{GNLP-Id} bound and the \ref{GNLP-Id-comp} bound obtained without the SOC constraints. 


\subsection{Constructing dual-feasible solutions for GMESP} 
To certify a valid upper bound for  GMESP and avoid drawing incorrect conclusions from \emph{near}-optimal solutions, we rigorously construct dual  solutions for the relaxations of GMESP. For the \ref{ddgfact} bound, \cite{GMESP_Alg} proposed  a closed-form construction of dual solutions yielding a small duality gap by adapting the approach in \cite{Weijun}. For the extended-variable relaxations proposed in Section \ref{sec:contrel}, we follow a similar procedure for all formulations. We explain the procedure for \ref{glinx}; the constructions for \ref{GNLP-Id} and \ref{GNLP-Id-comp} are analogous.

First, we extract the dual variable $\hat Z$ associated with the
constraint $\Diag(x) - X \succeq 0$ from the interior-point solver.
Because the solver may return a $\hat{Z} \not\succeq 0$, we project
$\hat{Z}$ onto the PSD cone by computing its eigendecomposition and
setting all negative eigenvalues to zero.  Next, note that when there
are no constraints $Ax\leq b$, we may set $\hat\pi:=\mathbf{0}$.  We
also set $\hat{W} := 0$, $\hat\eta:=\mathbf{0}$, as the SOC constraints
are not included.  Then, we define $\hat\tau,\hat\nu,\hat\upsilon$
following the construction given in
Appendix \ref{app:construct_dgfact}, after replacing
$F\hat\Theta F^\top$ with $\hat Z$.  This implies that
\[
\hat\upsilon - \hat\nu - A^\top \hat\pi - \hat\tau\mathbf{e}
       + \hat\eta + \diag(\hat Z) = 0.
\]
We then construct a feasible dual solution for the second constraint,
\[
\frac{1}{2}(W+W^\top) - \Omega + Z - C\Theta C + \Theta + \xi I_n = 0.
\]
Following a similar approach to that in \cite[Sec. 3.3.4]{FLbook}, we
define $\hat\Theta := \frac{1}{2}(C\hat{X}C + I_n - \hat{X})^{-1}$,
where $(\hat x,\hat X)$ is the primal solution returned by the
interior-point solver.  Set
$\hat M := \hat Z - C\hat\Theta C + \hat\Theta$ and define
\[
\hat\xi := -\lambda_n(\hat M), \qquad
\hat\Omega := \hat M + \hat\xi I_n
            = \hat Z - C\hat\Theta C + \hat\Theta + \hat\xi I_n\,.
\]
By construction, $\hat\Omega$ satisfies the second constraint with
equality.  Moreover, the choice
$\hat\xi := -\lambda_n(\hat M)$ ensures that $\hat\Omega \succeq 0$. Therefore, we conclude that $(\hat\Theta,\hat\upsilon,\hat\nu,\hat\pi,
\hat\tau,\hat\xi,\hat Z,\hat\Omega,\hat\eta)$ is a feasible solution
for \ref{eq:dual_glinx}.


\subsection{Bound comparison}

In our first experiment, we  use a benchmark covariance matrix of dimension $n=63$,  originally  
obtained from J. Zidek (University of British Columbia), coming from an application for re-designing an environmental monitoring network;
see \cite{Guttorp-Le-Sampson-Zidek1993} and \cite{HLW}. This  matrix has been used extensively in testing and developing algorithms for $\MESP$; see \cite{KLQ,LeeConstrained,AFLW_Using,LeeWilliamsILP,HLW,AnstreicherLee_Masked,BurerLee,Anstreicher_BQP_entropy,Kurt_linx,Mixing,FactPaper,PFLXadmm}.  
We note that \cite{GMESP_Alg} reported good results when $\kappa := s-t$ was small; the largest value considered was $\kappa = 3$. Their approach performs well when $\GMESP$ is similar to $\MESP$. In what follows, we will see that our approach for \ref{scaleglinx} (our results for \ref{scaleglinx} are computed with the optimal o-scaling parameter $\gamma$; see Subsection \ref{sec:linx}) is much more robust in terms of $\kappa$, showing that  \ref{scaleglinx} remains a \emph{very strong} relaxation even when $\kappa$ is very large, such as $\kappa = 24$. 

In Figure \ref{fig:varying_t_for_different_kappa63}, we present results analogous to those shown in \cite[Figure 1]{GMESP_Alg}, while also consider larger values of $\kappa$. For $\kappa = 0$, the experiments reduce to $\MESP$, as the spectral,  \ref{ddgfact} and  \ref{scaleglinx} bounds coincide with their $\MESP$ bounds when $s = t$. In this setting,  \ref{ddgfact} and \ref{scaleglinx} are competitive, with no clear dominance; moreover, the complementary factorization bound for $\MESP$ performs poorly
 on these instances (see \cite{FactPaper}). 
 This behavior changes in the $\GMESP$ setting. We observe that  \ref{scaleglinx} consistently outperforms \ref{ddgfact}, even for $\kappa = 1$. Furthermore,  \ref{scaleglinx} remains stable as $\kappa$ increases,
in sharp contrast with the deterioration observed for \ref{ddgfact}. 
 Even for a very large value of $\kappa$, all gaps associated with  \ref{scaleglinx} remain below one, highlighting the effectiveness of this new proposed bound. We do not show  results for \ref{GNLP-Id} and \ref{GNLP-Id-comp} in Figure \ref{fig:varying_t_for_different_kappa63} because they were not competitive for the instances considered there. Finally,  as noted in \cite[Figure 1]{GMESP_Alg},  the spectral bound becomes more competitive with the \ref{ddgfact} bound as $\kappa$ increases. Now, our results show that, for larger  values of $\kappa$, the spectral bound can significantly outperforms \ref{ddgfact} bound. Nevertheless, it remains consistently weaker than the \ref{scaleglinx} bound.

\begin{figure}[!ht] 
    \centering
    \subfloat[$\kappa = 0$]{{\includegraphics[width=0.4\textwidth]{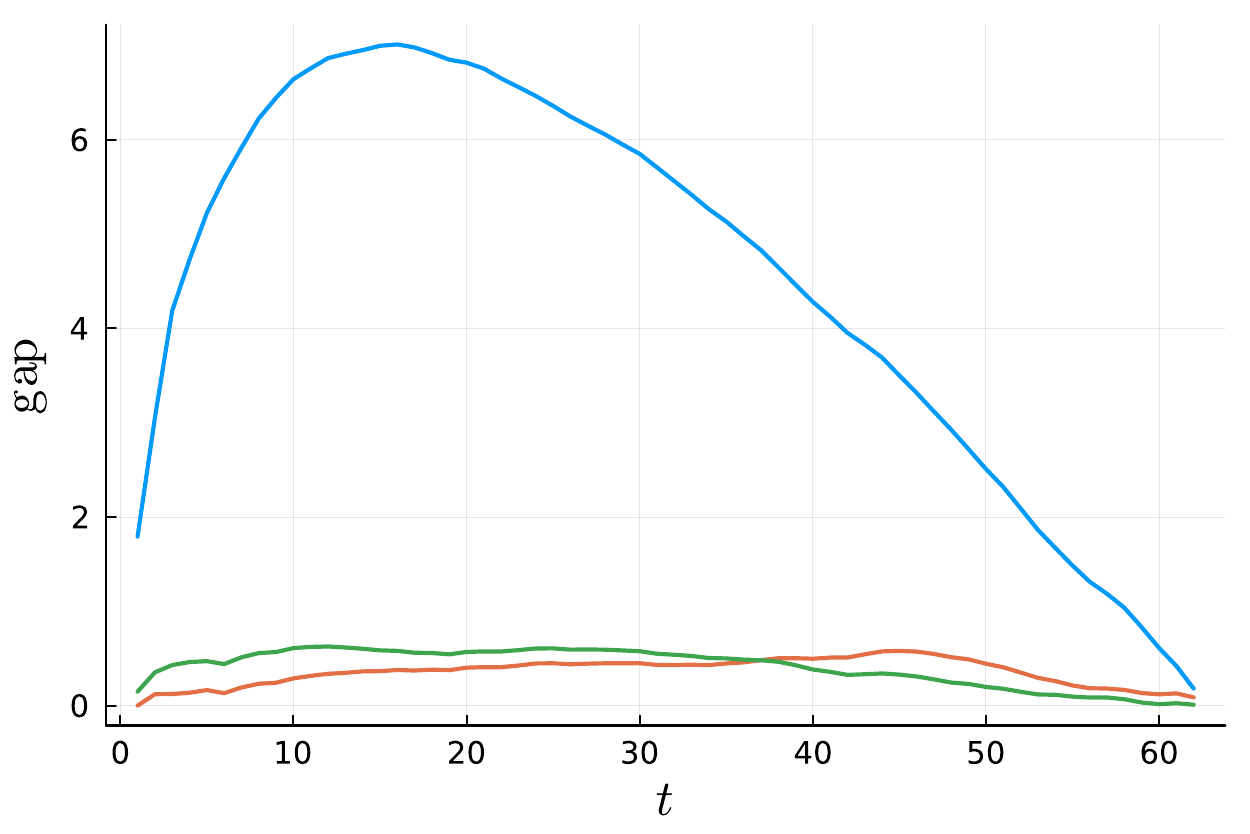} }}
    \subfloat[$\kappa = 1$]{{\includegraphics[width=0.4\textwidth]{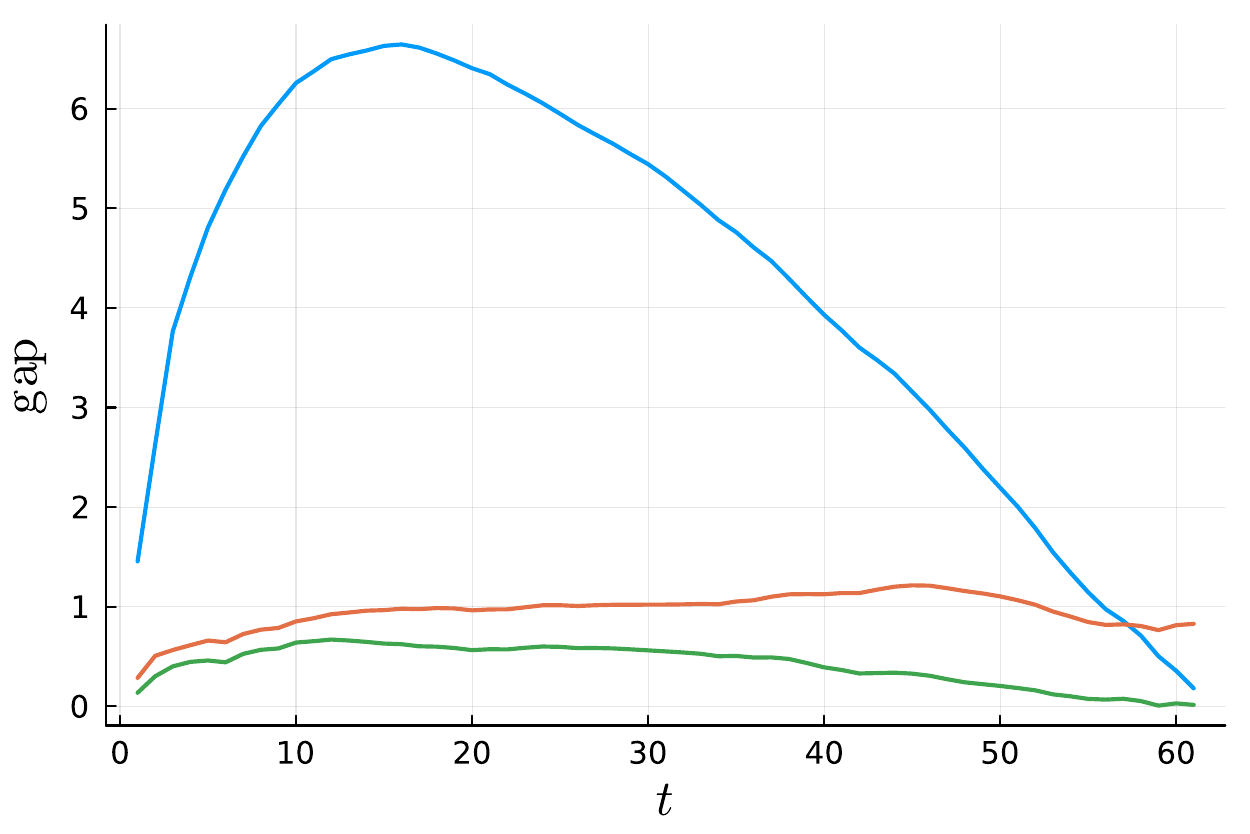} }}
    
    \subfloat[$\kappa = 12$]{{\includegraphics[width=0.4\textwidth]{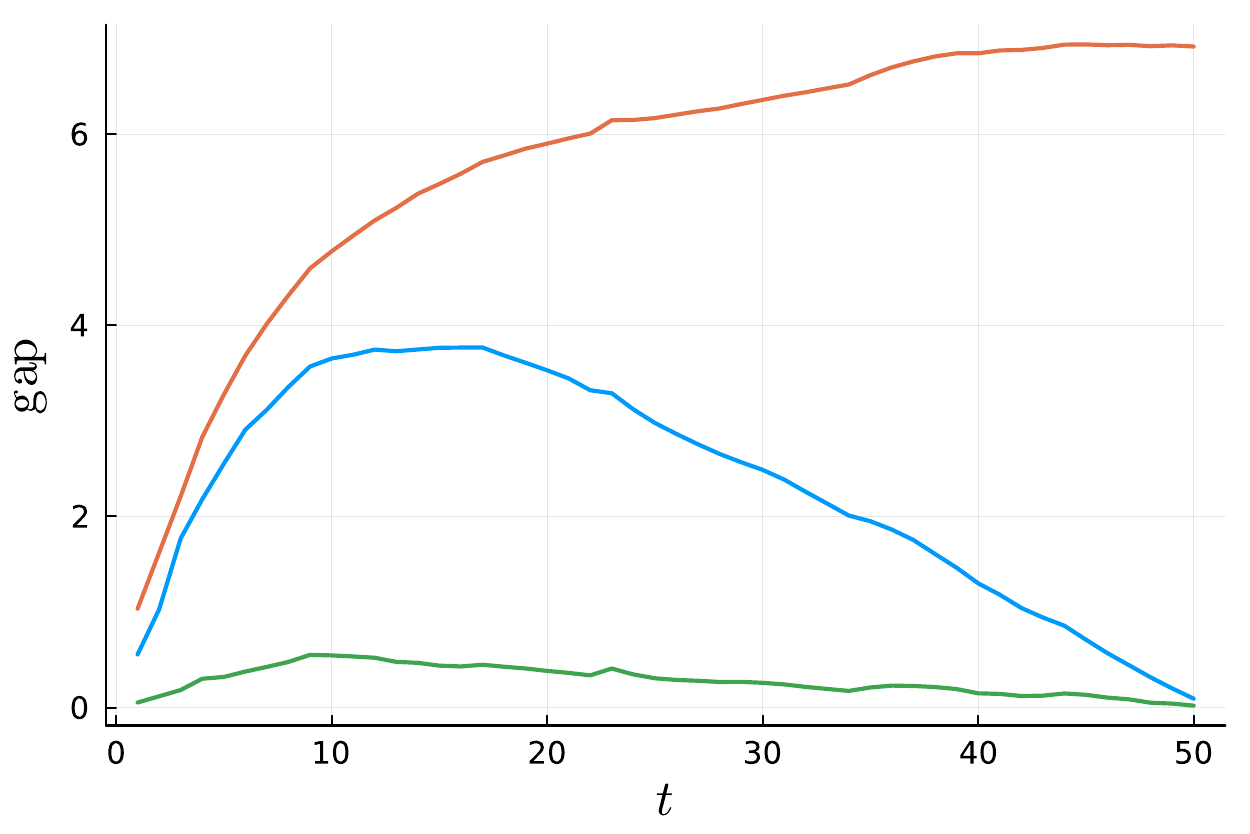} }}%
    ~
    \subfloat[$\kappa = 24$]{{\includegraphics[width=0.4\textwidth]{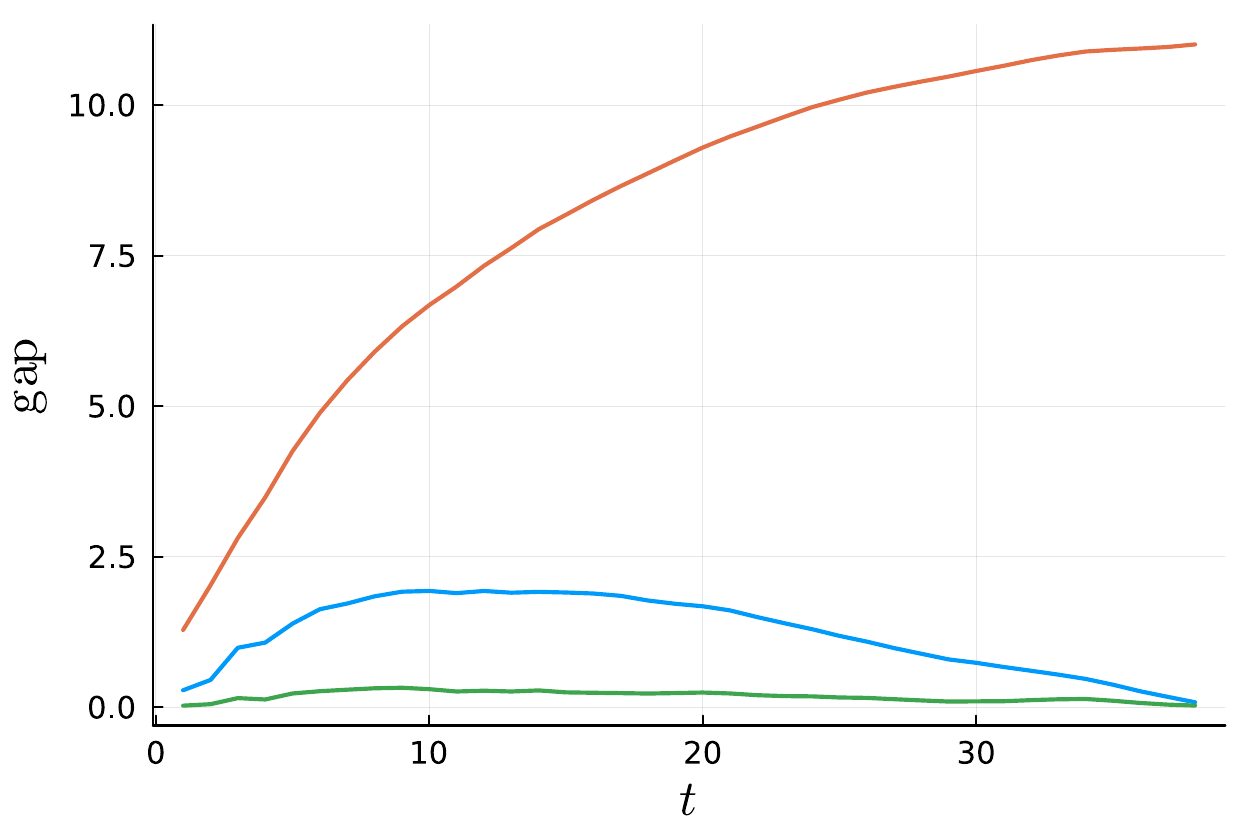} }}\\
     \subfloat{{\includegraphics[scale=0.65]{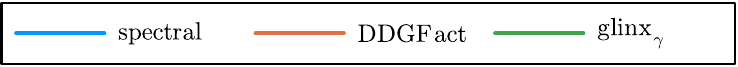} }}%
     \caption{Gaps for GMESP varying $t = s - \kappa$ ($n=63$) }%
\label{fig:varying_t_for_different_kappa63}%
\end{figure}

Next, we consider larger instance sizes that were not tested in \cite{GMESP_Alg}. In particular, we consider two additional  
benchmark instances with  $n=90$ and $n=124$, based on nonsingular covariance matrices $C$ obtained from J. Zidek (University of British Columbia). The former is based on temperature data from monitoring stations in the Pacific Northwest of the
United States, while the latter is derived from an application to the redesign of an environmental monitoring network. These matrices have been widely used in testing and developing algorithms for $\MESP$; see, for example, \cite{Kurt_linx,gscale,li2025augmented,MESP2DOPT}. 
Figures
\ref{fig:varying_t_for_different_kappa90} and \ref{fig:varying_t_for_different_kappa124} show that \ref{scaleglinx} consistently outperforms the other bounds in most cases, with the advantage being particularly pronounced as $\kappa$ increases. 

\begin{figure}[!ht]%
    \centering
    \subfloat[$\kappa = 0$]{{\includegraphics[width=0.4\textwidth]{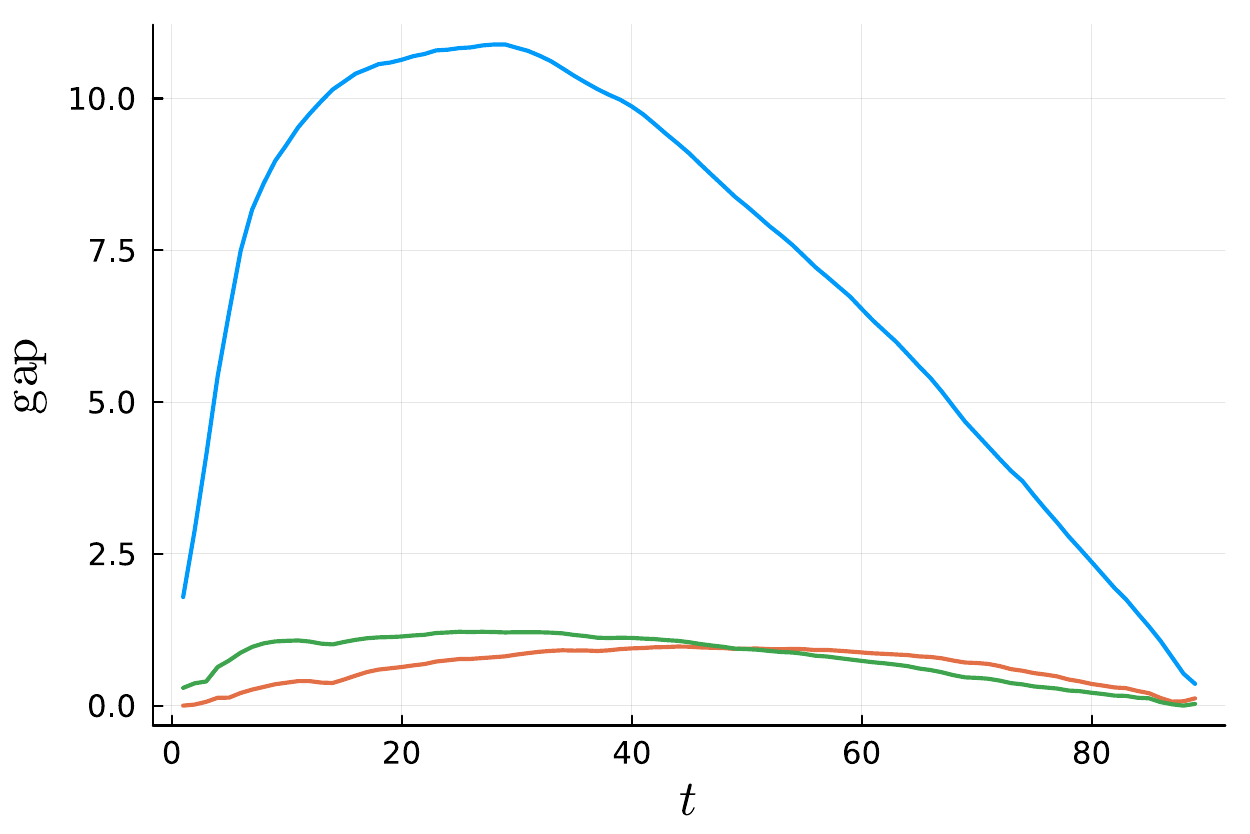} }}
    \subfloat[$\kappa = 1$]{{\includegraphics[width=0.4\textwidth]{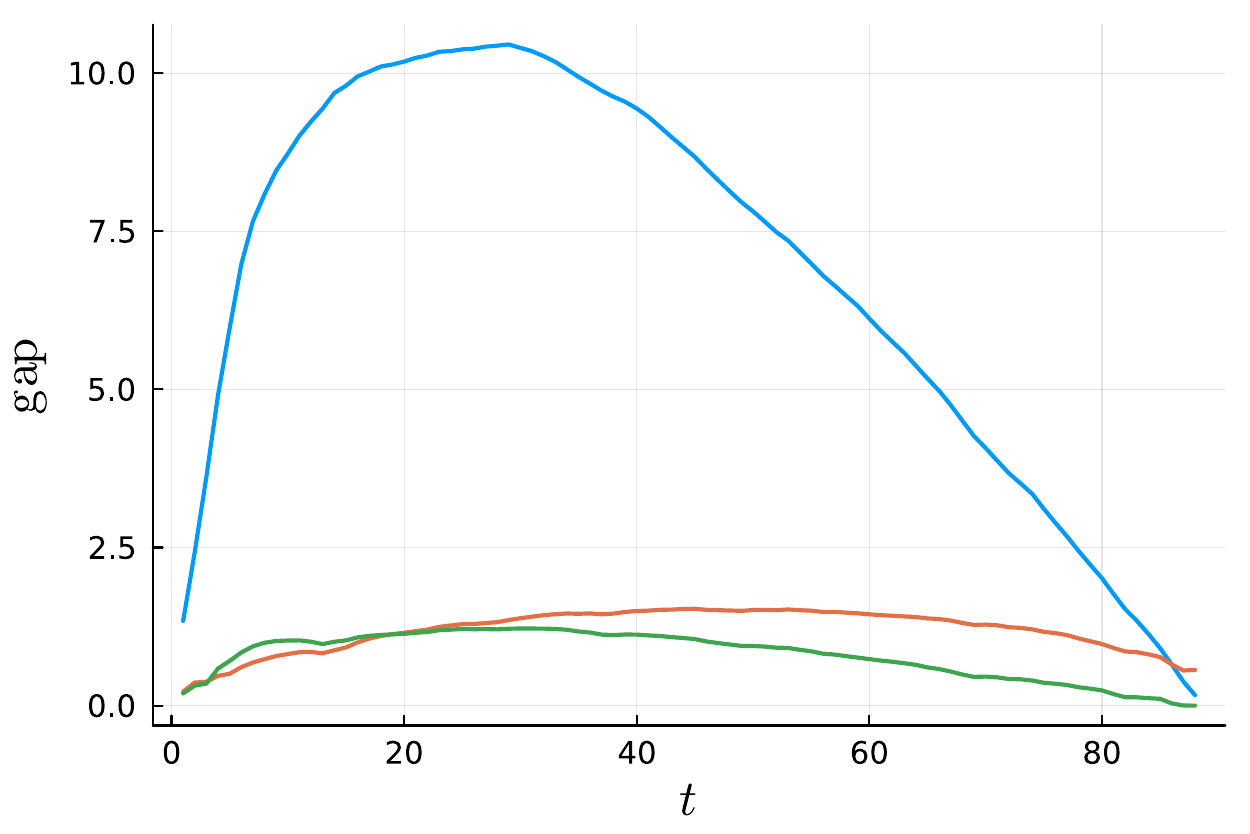} }}
    
    \subfloat[$\kappa = 12$]{{\includegraphics[width=0.4\textwidth]{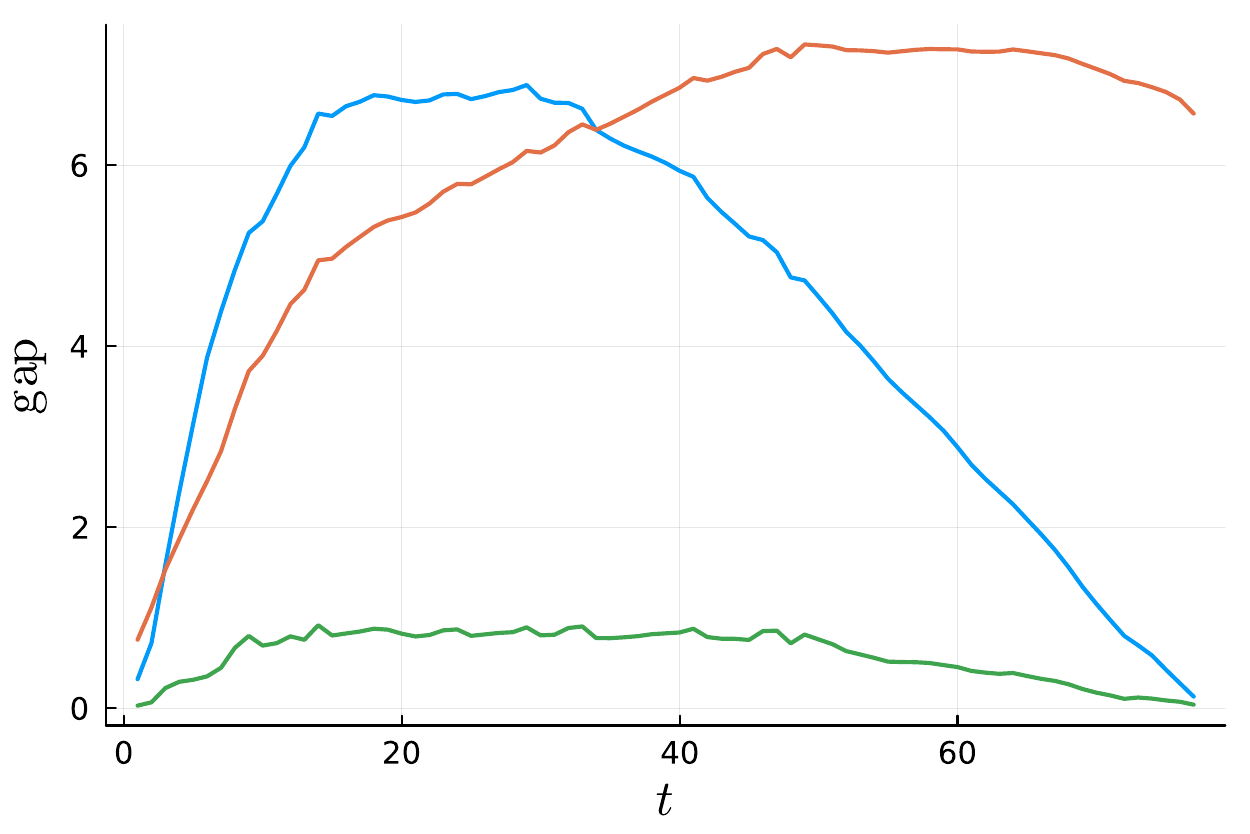} }}%
    ~
    \subfloat[$\kappa = 24$]{{\includegraphics[width=0.4\textwidth]{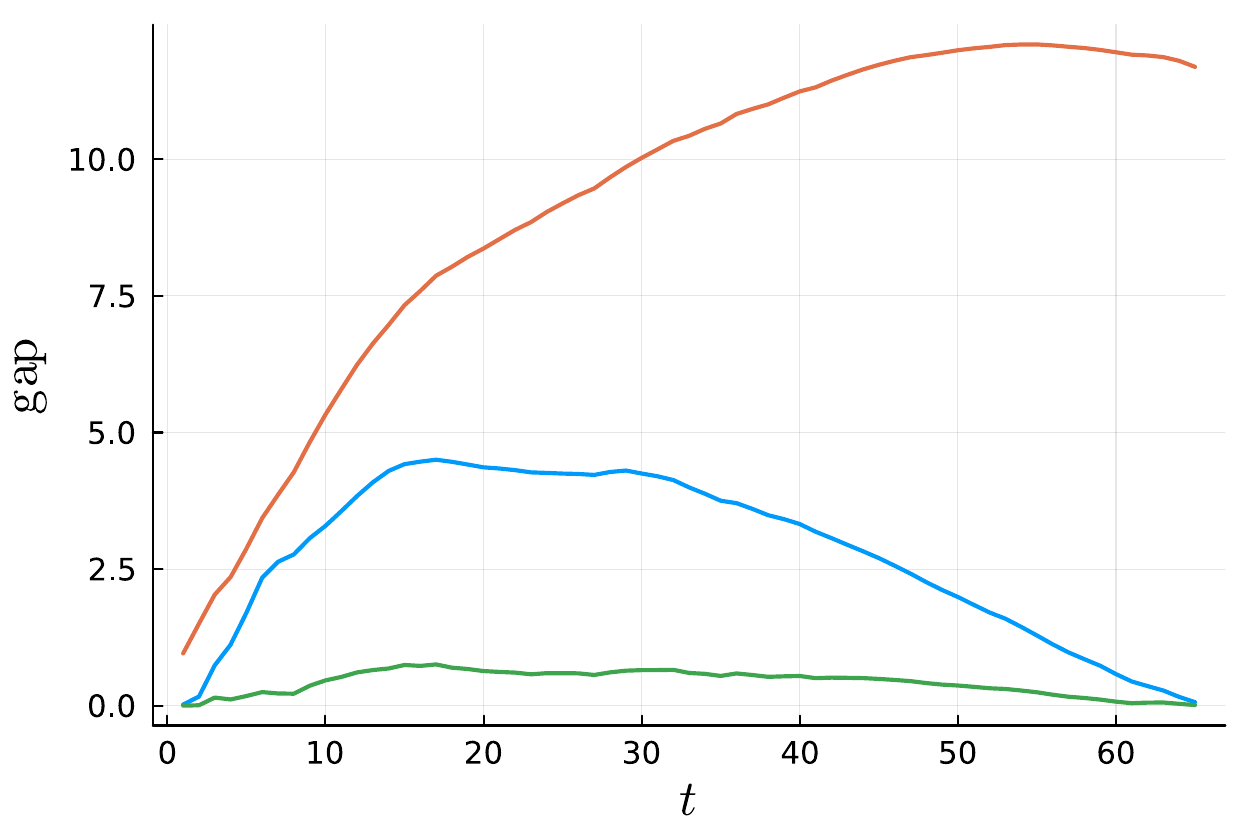} }}\\
    
     \subfloat{{\includegraphics[scale=0.65]{label63.pdf} }}%
     \caption{Gaps for GMESP varying $t = s - \kappa$ ($n=90$) }%
\label{fig:varying_t_for_different_kappa90}%
\end{figure}

\begin{figure}[!ht] 
    \centering
    \subfloat[$\kappa = 0$]{{\includegraphics[width=0.4\textwidth]{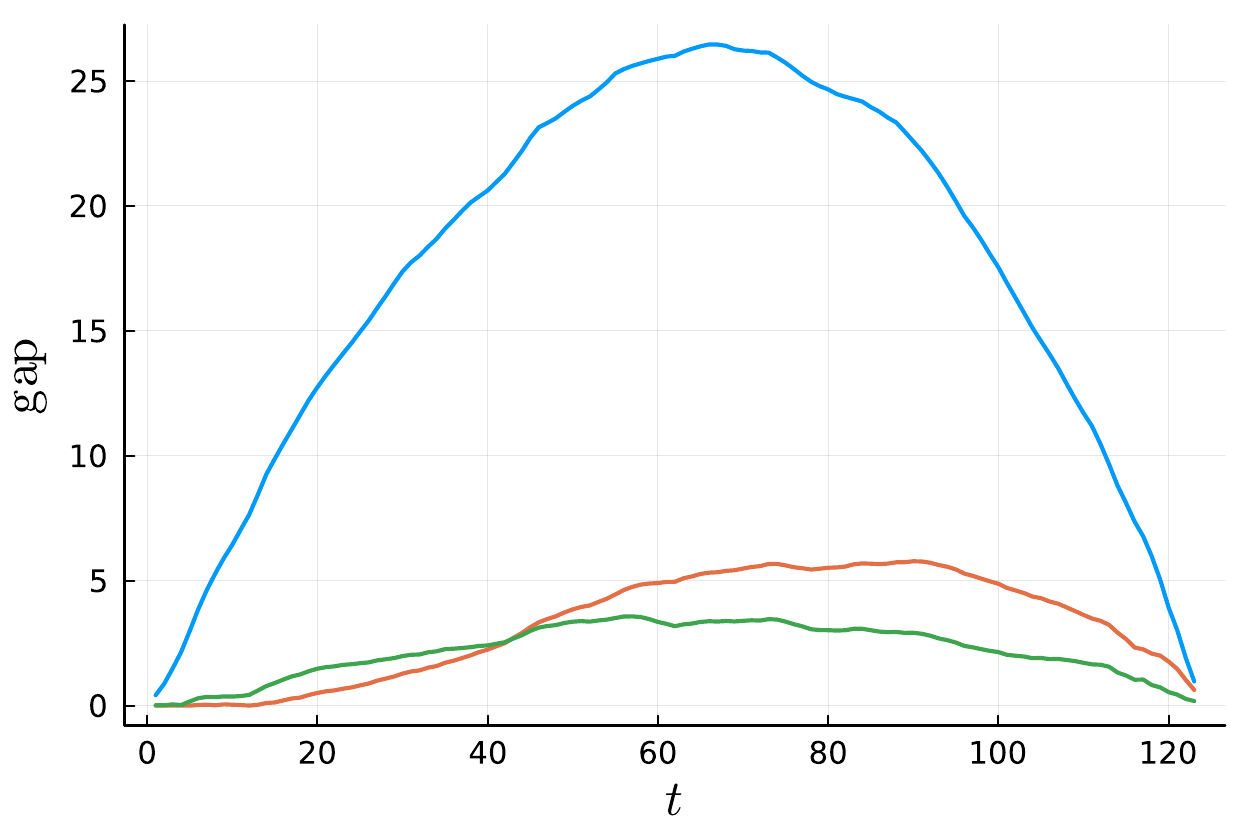} }}
    \subfloat[$\kappa = 1$]{{\includegraphics[width=0.4\textwidth]{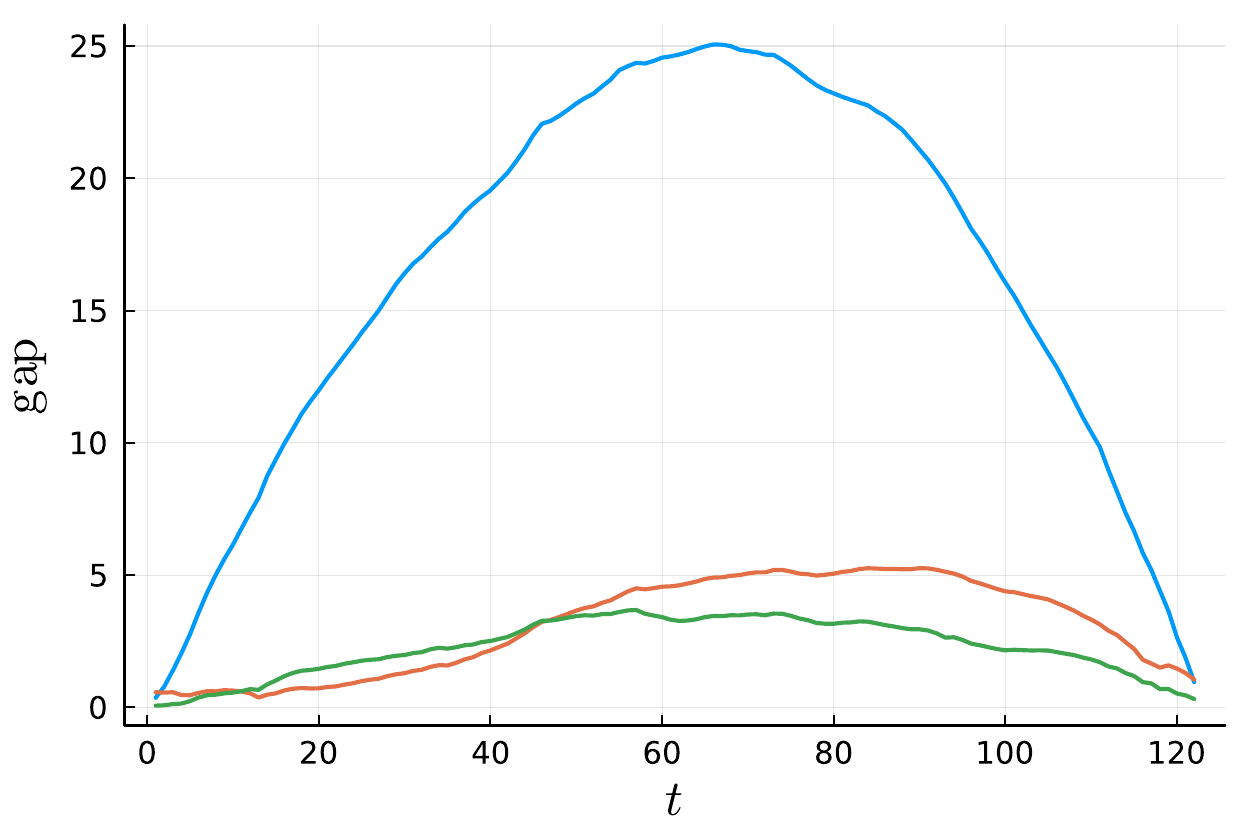} }}
    
    \subfloat[$\kappa = 12$]{{\includegraphics[width=0.4\textwidth]{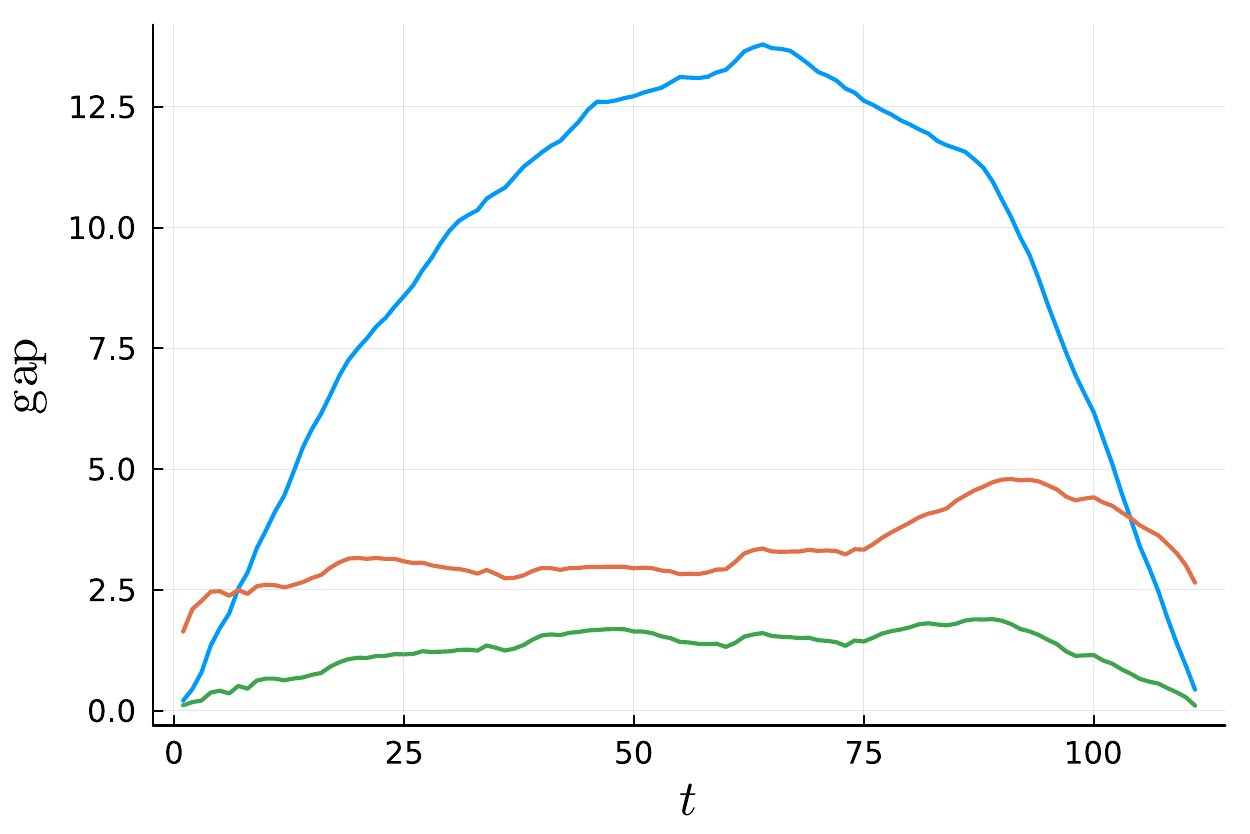} }}%
    ~
    \subfloat[$\kappa = 24$]{{\includegraphics[width=0.4\textwidth]{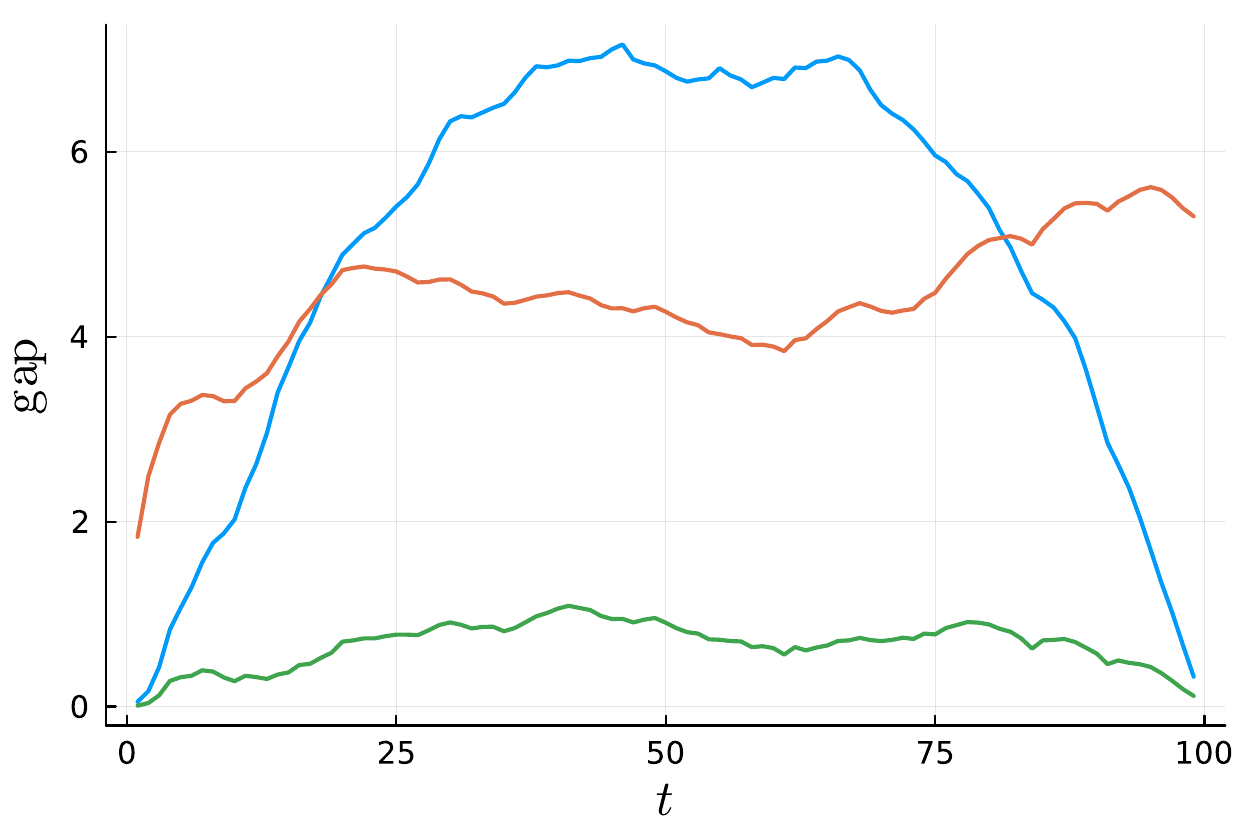} }}
    
     \subfloat{{\includegraphics[scale=0.65]{label63.pdf} }}%
     \caption{Gaps for GMESP varying $t = s - \kappa$ ($n=124$) }%
\label{fig:varying_t_for_different_kappa124}%
\end{figure}

Finally, we consider instances for which the \ref{GNLP-Id} and \ref{GNLP-Id-comp} bounds perform well. To this end, we extract the leading principal submatrix of order $50$ from our $124$-dimensional benchmark instance and investigate the impact of having similar eigenvalues.
Following \cite{MESP2DOPT}, we construct covariance matrices as described next. To create instances with  similar greatest eigenvalues, we keep the eigenvectors of $C$ fixed and  modify its greatest $19$ eigenvalues
so that they are similar the  20th. Specifically, for $i = 1, \dots, 19$, we set 
$\lambda_i(C) := 1.001^{20-i}\cdot \lambda_{20}(C)$. An analogous procedure is applied to generate covariance matrices with similar least eigenvalues.    
This adjustment slightly separates the top or bottom eigenvalues while keeping them similar. 
In Figure \ref{fig:similar-eigvals},  we observe that 
\ref{GNLP-Id} (resp., \ref{GNLP-Id-comp}) performs well when the greatest (resp., least) eigenvalues of $C$ are similar, $\kappa$ is small, and $t$ is small (resp., large). We omit  the spectral bound, as it is not competitive in this setting. 
\begin{figure}[!ht]
    \centering
    \includegraphics[width=0.4\textwidth]{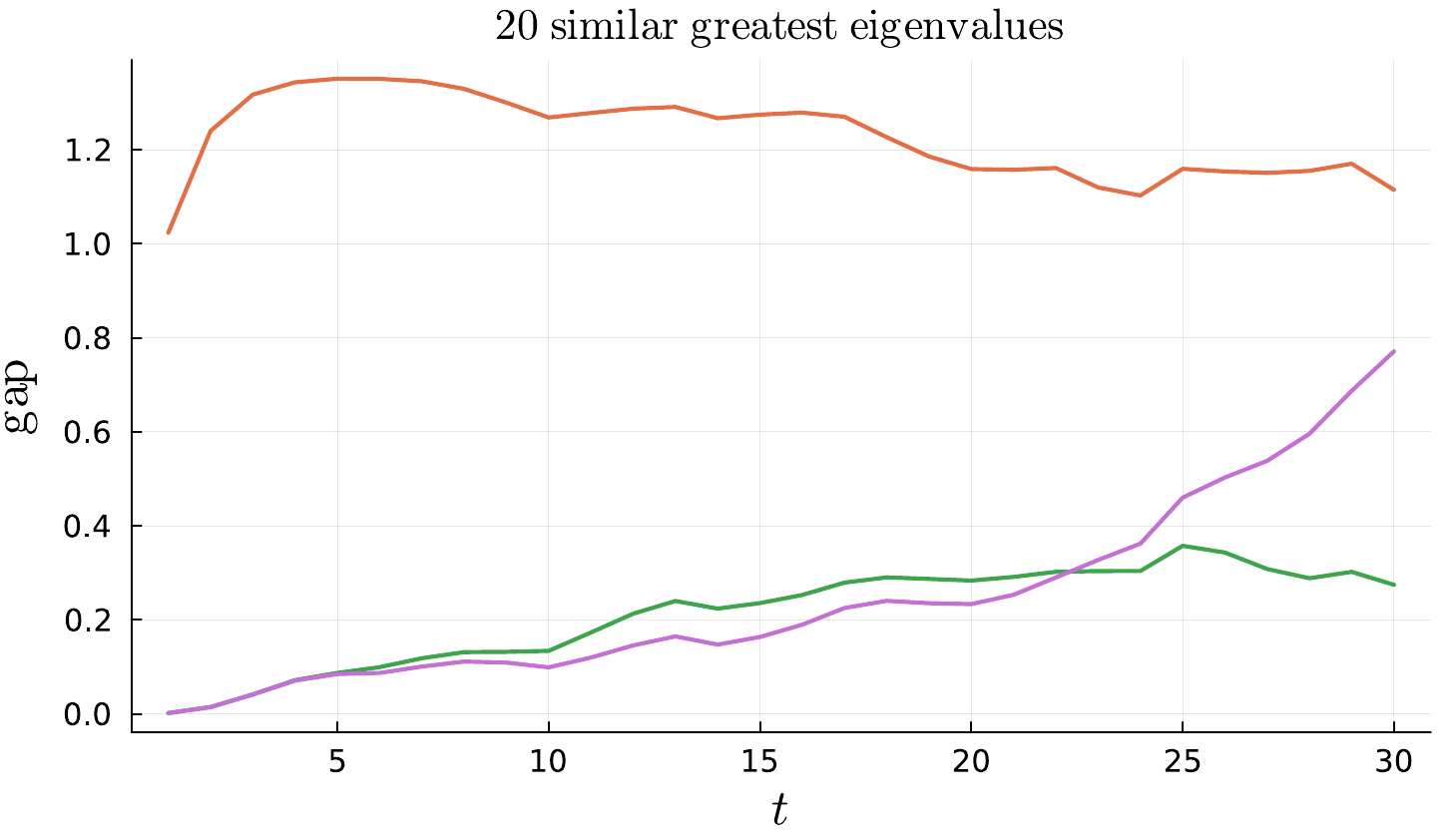}
    \includegraphics[width=0.4\textwidth]{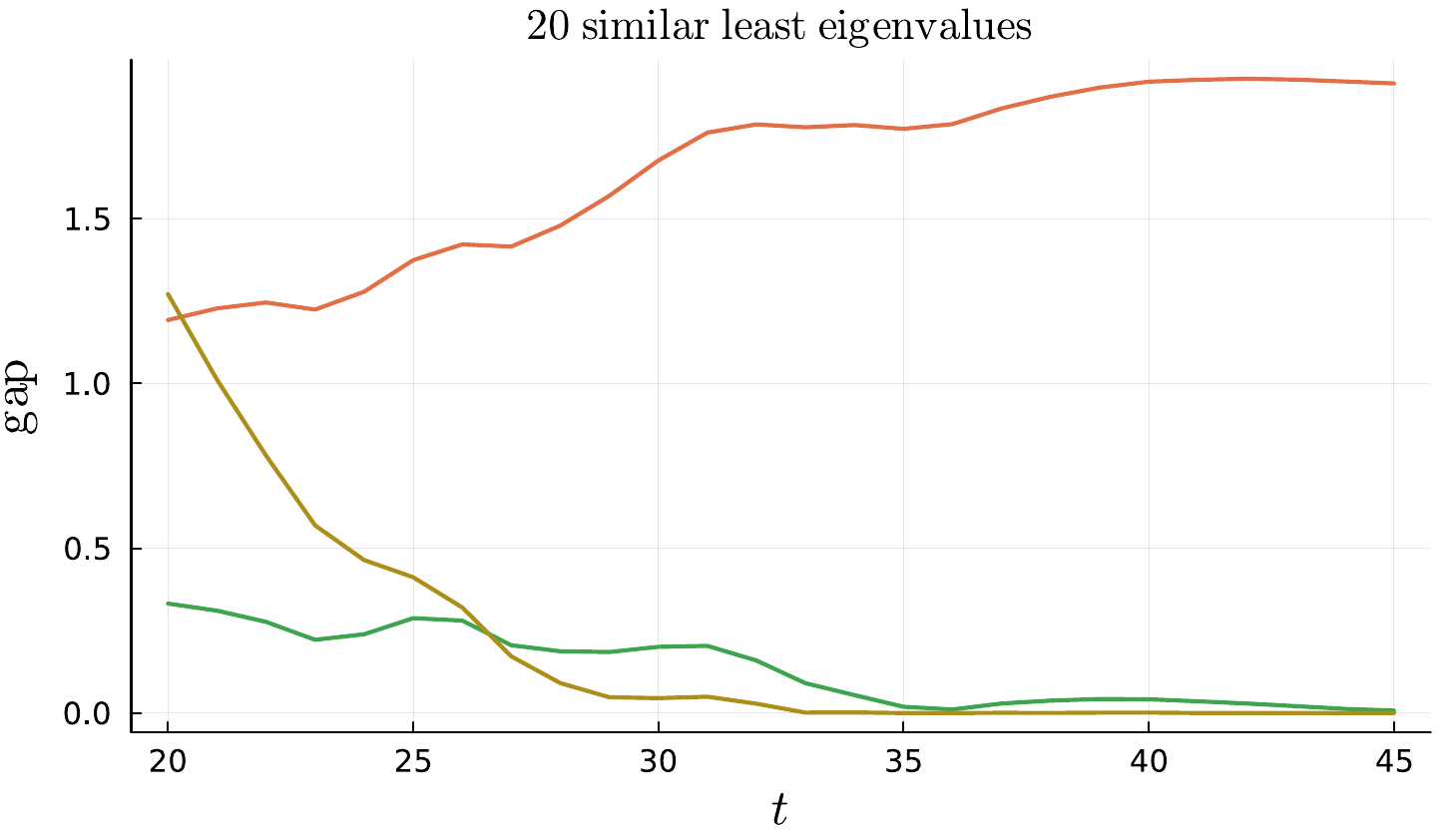}\\
    \includegraphics[width=0.5\textwidth]{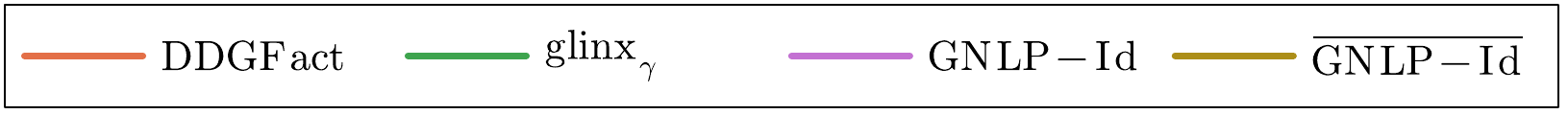}
    \caption{GMESP bounds for modified covariance matrices ($\kappa =2$, $n=50$)
    }
   \label{fig:similar-eigvals}
\end{figure}


\subsection{Effect of generalized scaling}

In the next experiments, we analyze the numerical performance of the g-scaling procedures proposed in Section \ref{sec:gscale}.

In the first experiment, we consider the same instances used in Figure \ref{fig:varying_t_for_different_kappa63}  and compare the performance of \ref{scaleglinx} and \ref{gscaleglinx}\,. We begin by  computing the optimal o-scaling parameter $\hat\gamma$ for \ref{scaleglinx}\,. Then, we initialize the BFGS algorithm used to locally optimize $\Upsilon$ at $\hat{\gamma}^{\scriptscriptstyle 1/4}\mathbf{e}$. We note that when $\gamma:=\hat\gamma$, and $\Upsilon=\hat{\gamma}^{\scriptscriptstyle 1/4}\mathbf{e}$, the objective values of \ref{scaleglinx} and \ref{gscaleglinx} coincide. Consequently, after locally optimizing $\Upsilon$, the bound produced by \ref{gscaleglinx}  is guaranteed to be at least as strong as the bound produced by \ref{scaleglinx}\,. Overall, as $\kappa$ increases, the advantage of the \ref{gscaleglinx} bound over the \ref{scaleglinx} bound becomes increasingly pronounced, even though the gaps associated with \ref{scaleglinx} are already very small for the instances considered.

\begin{figure}[!ht] 
    \centering
    \subfloat[$\kappa = 0$]{{\includegraphics[width=0.4\textwidth]{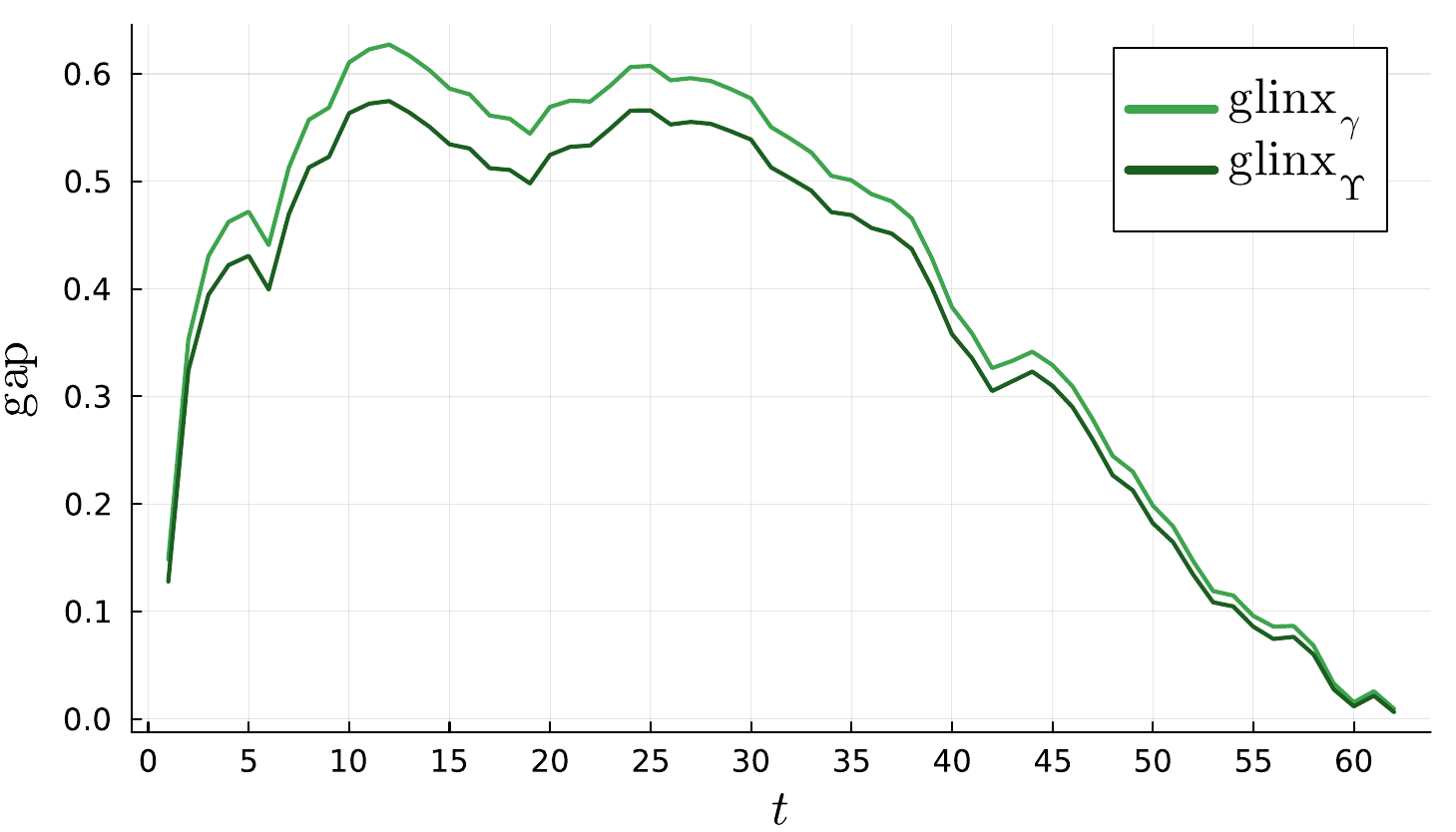} }}
    \subfloat[$\kappa = 1$]{{\includegraphics[width=0.4\textwidth]{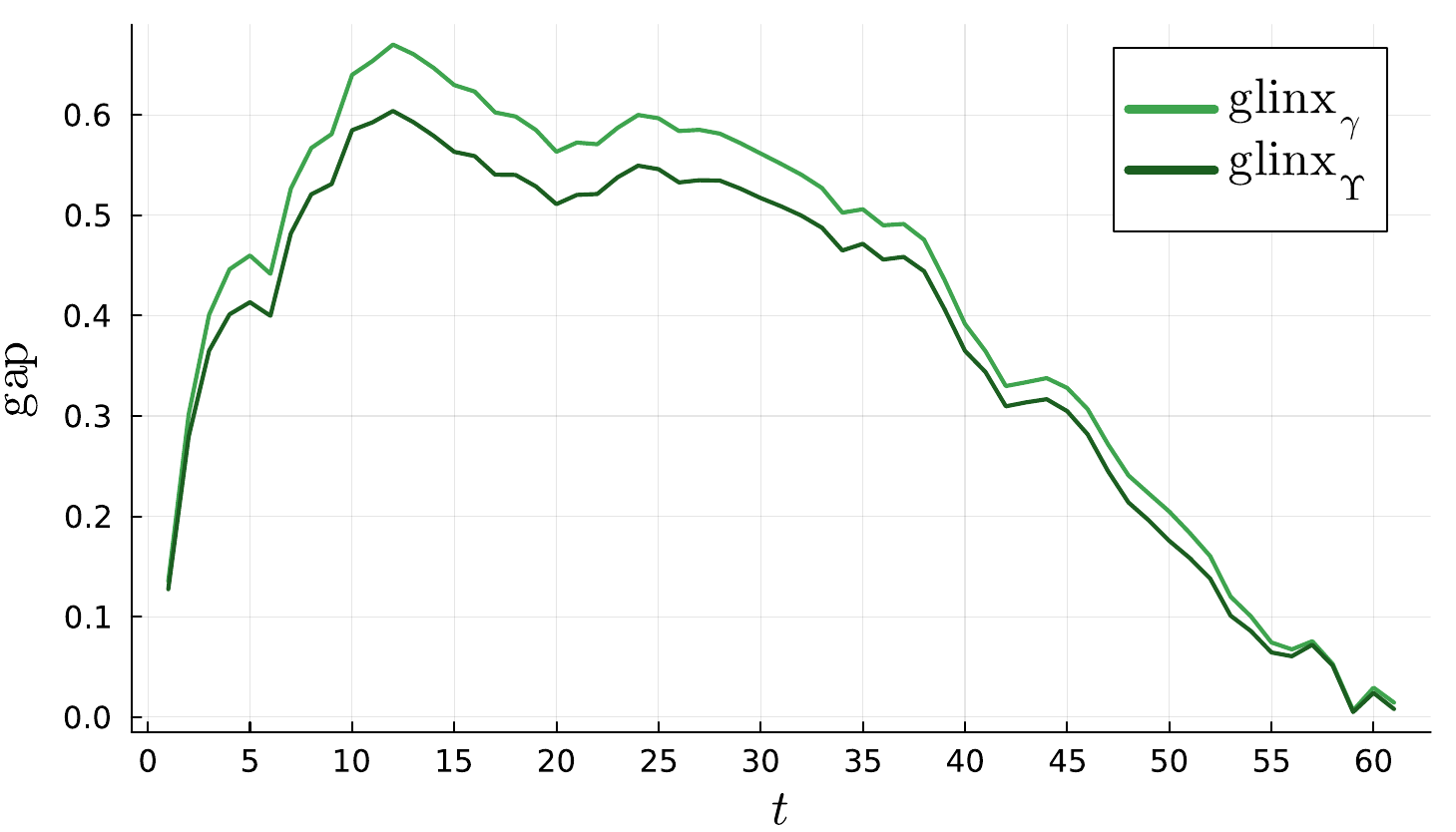} }}
    
    \subfloat[$\kappa = 12$]{{\includegraphics[width=0.4\textwidth]{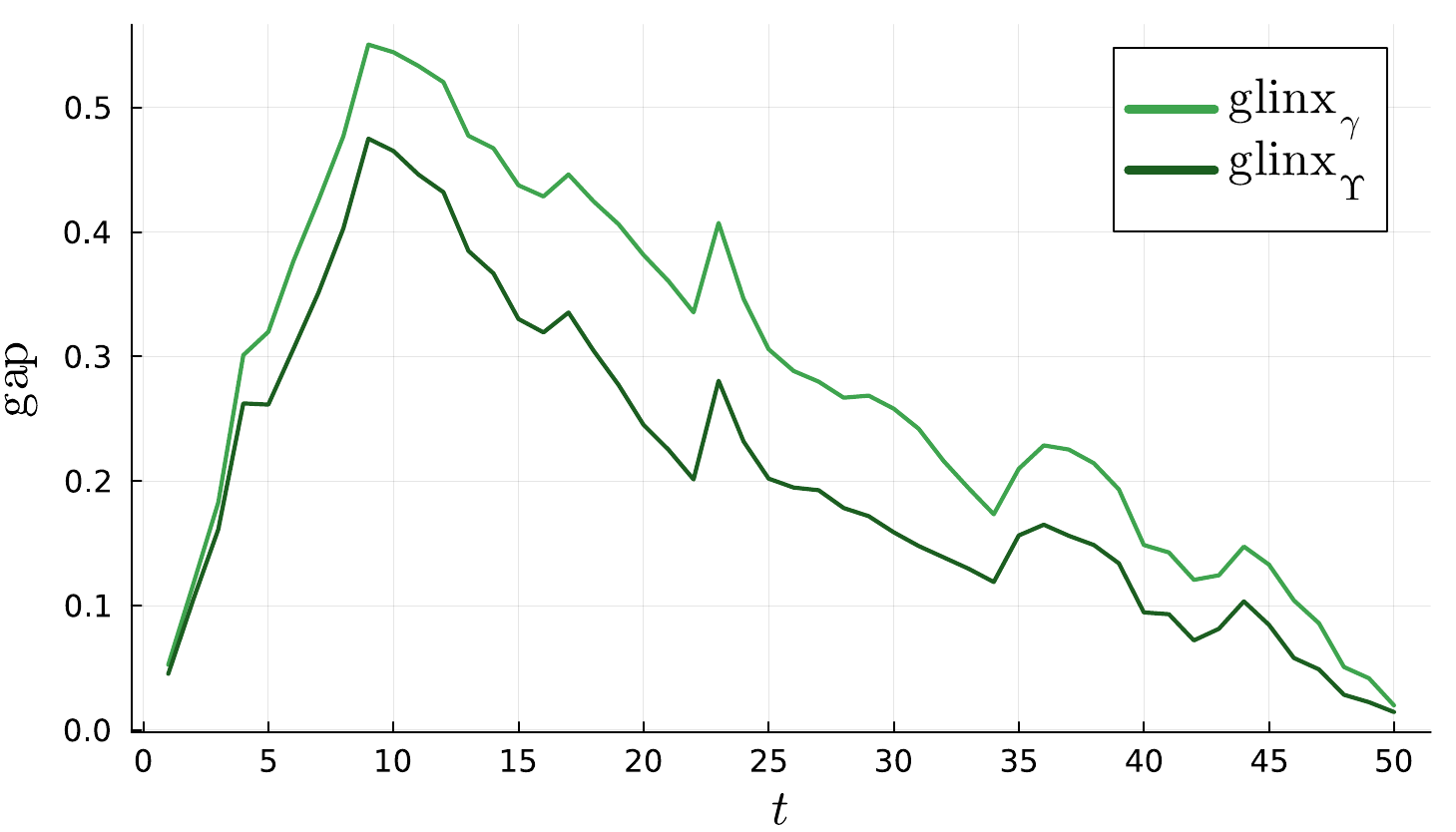} }} 
    ~
    \subfloat[$\kappa = 24$]{{\includegraphics[width=0.4\textwidth]{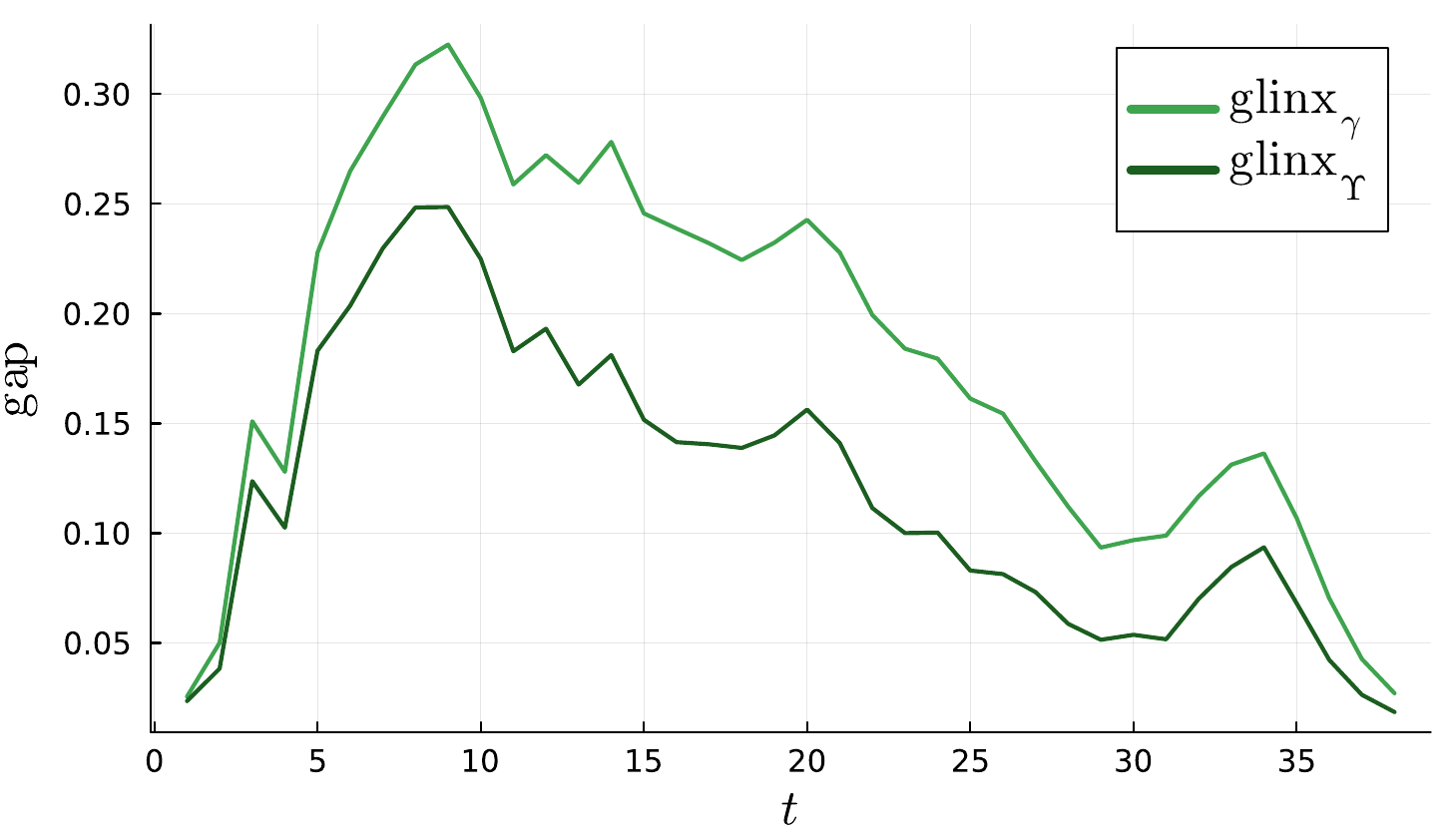} }}\\
     \caption{Gaps for \ref{glinx} for generalized scaling varying $t = s - \kappa$ ($n=63$)
     } 
\label{fig:varying_t_for_different_kappa_gscale_63} 
\end{figure}

For our second experiment, we use the benchmark covariance matrix with $n=124$ mentioned above to analyze scaling for \ref{ddgfact}. We set $s=62$ and vary $\kappa$ between $1$ and $3$. 
We begin with the following observations. By Remark \ref{rem:gamma_invariant},  the  \ref{ddgfact} bound is invariant under o-scaling, so we compare it directly with the \ref{ddgfactscale} bound. 
Moreover, in light of Theorem \ref{thm:UpsOptOnes},  $\Upsilon = \mathbf{e}$ is optimal in the absence of side constraints. 
 Nevertheless,  once variables are fixed to one (by adding constraints of the form $x_i=1$), we observe consistent improvements of \ref{ddgfactscale} over \ref{ddgfact}. This effect is particularly relevant in a B\&B framework for $\GMESP$, where such variable fixings naturally arise.
We initialize the experiment by computing a solution $\hat x$ to \ref{CGMESP} using a local-search heuristic, and let $S$ denote its support. We then perform a sequence of $s-1$ diving steps, fixing one free variable $x_i$ to $1$ in \ref{ddgfactscale} at each step. The selected index $i \in S$ corresponds to the variable whose value in the current relaxation is farthest from one. As fixings are restricted to indices in $S$, the resulting gaps (i.e., the difference between the optimal values of \ref{ddgfactscale} and the lower bound provided by the local-search heuristic) remain nonnegative.
At the root node, we set $\Upsilon = \mathbf{e}$, as this choice is optimal by Theorem \ref{thm:UpsOptOnes}. At each child node, we locally optimize the
 scaling vector using a BFGS algorithm, warm-starting $\Upsilon$ from the parent's scaling vector. 
The results are reported in Figure \ref{fig:diving_ddgfactgscale}. 
We observe a clear pattern: as the number of fixed variables increases (up to approximately 45), the improvement of \ref{ddgfactscale} over \ref{ddgfact} becomes more pronounced. 
However, as $\kappa$ increases,  the optimization of the scaling factor becomes less effective, and the observed gains  diminishes. 

\begin{figure}[!ht]
    \centering
    \subfloat[$\kappa = 1$]{{\includegraphics[width=0.33\textwidth]{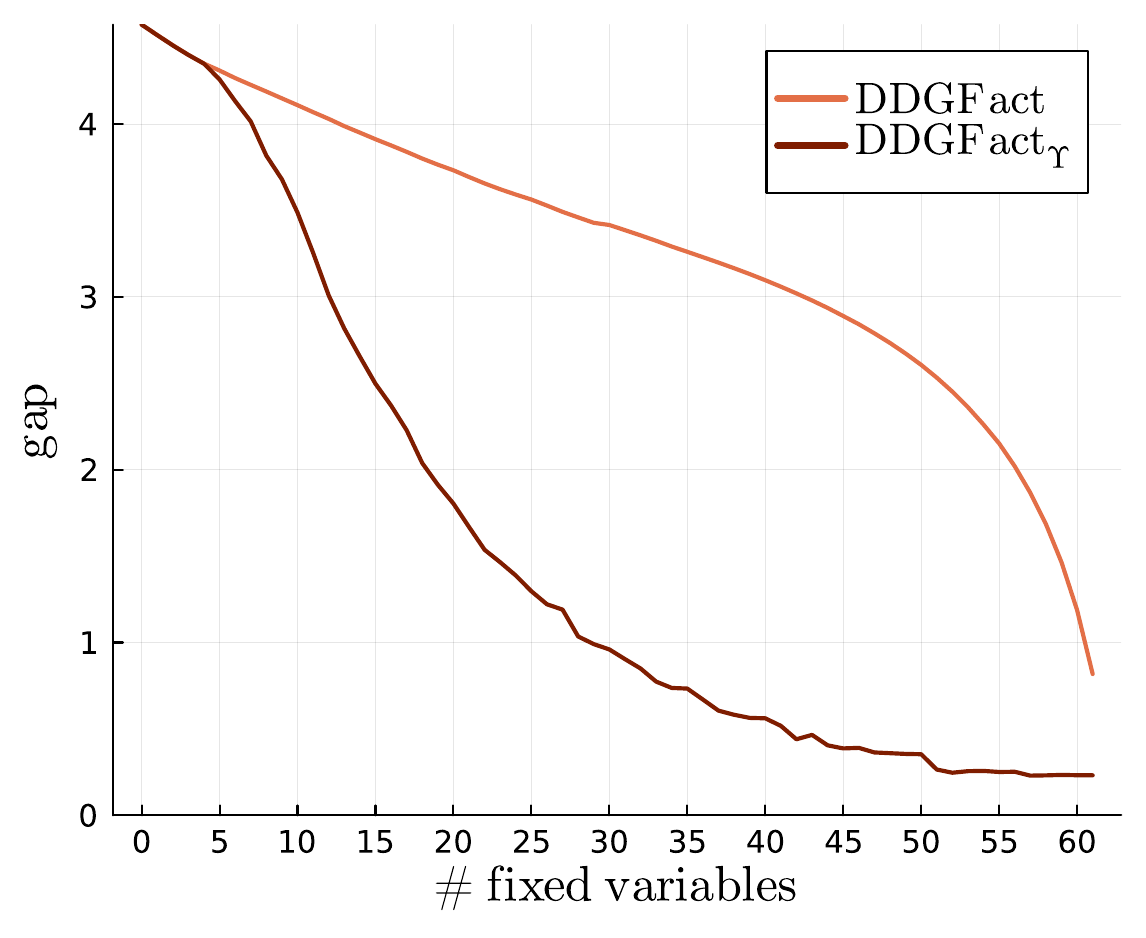} }}
    \subfloat[$\kappa = 2$]{{\includegraphics[width=0.33\textwidth]{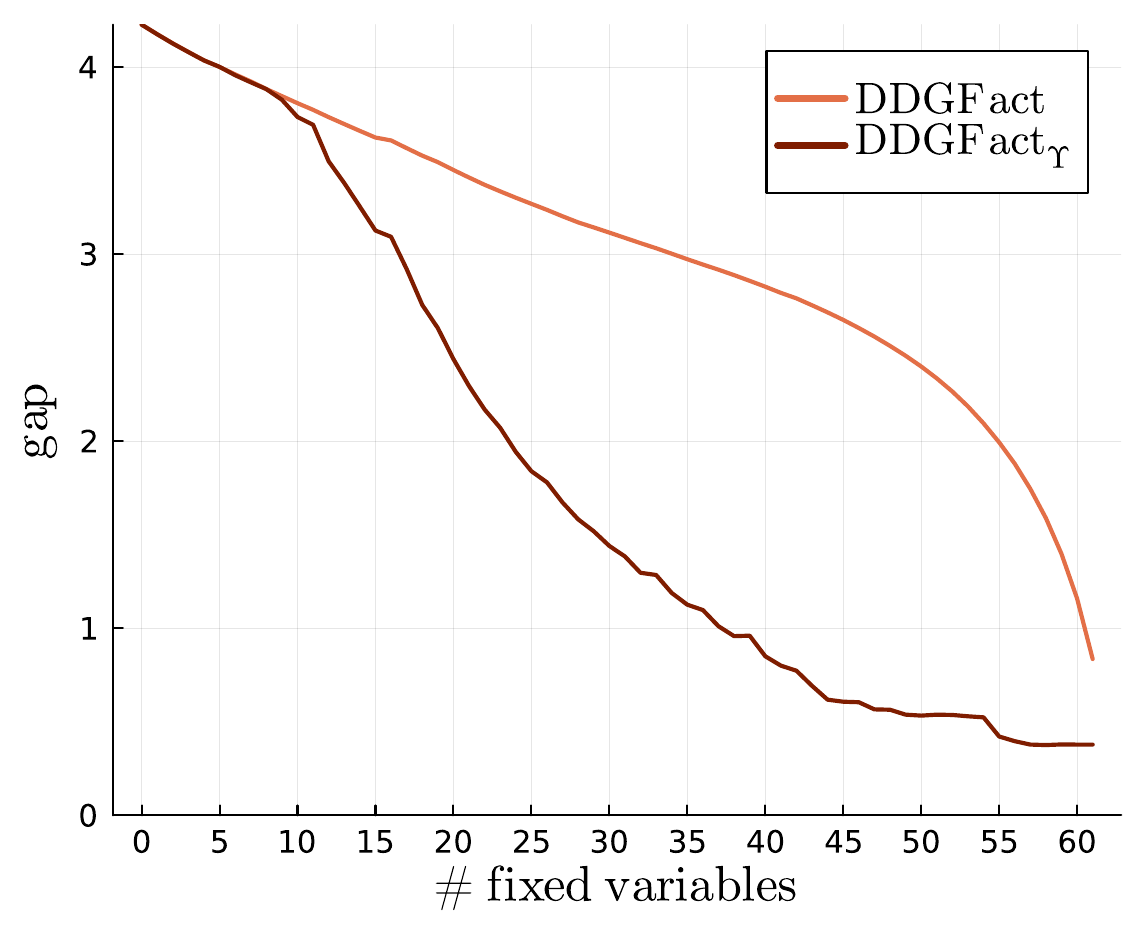} }}
    \subfloat[$\kappa = 3$]{{\includegraphics[width=0.33\textwidth]{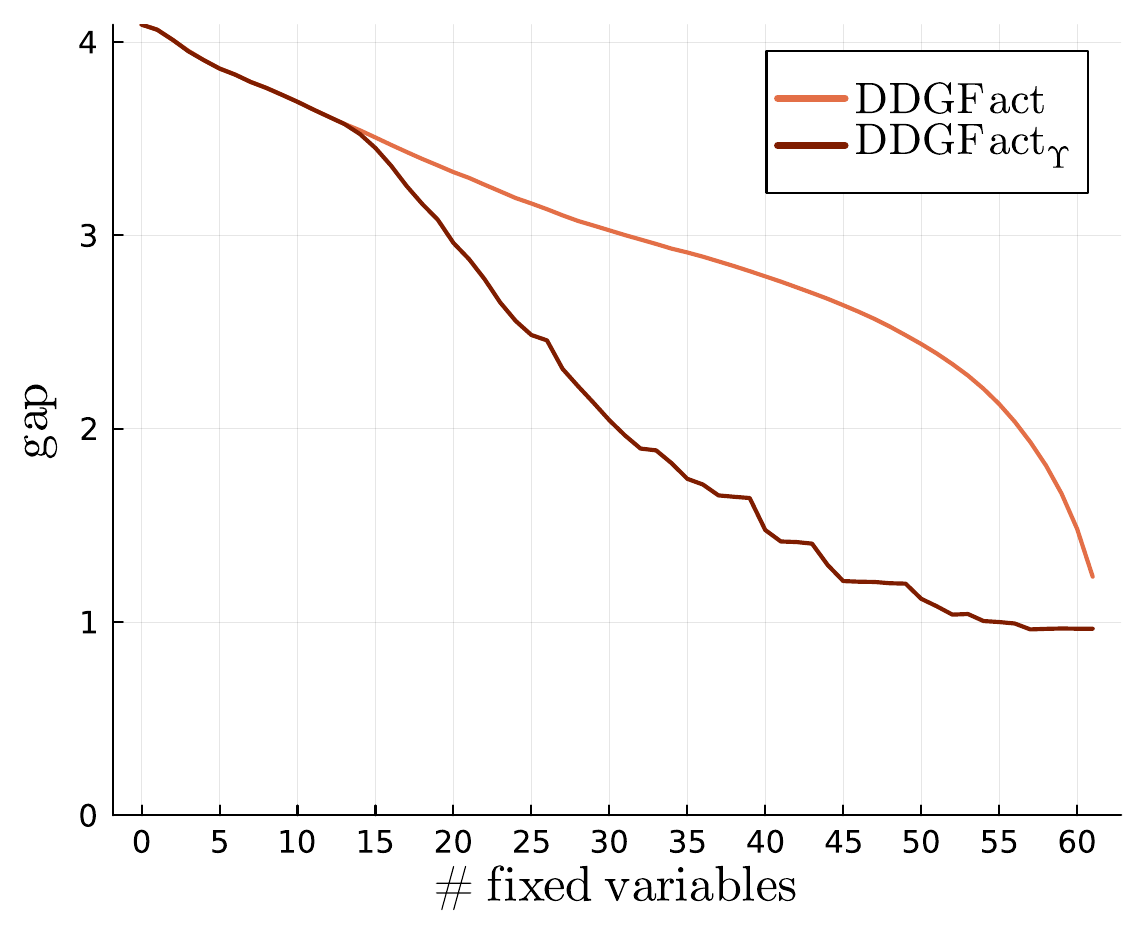} }}
    \caption{Effect of generalized scaling for \ref{ddgfactscale} in variable fixing ($n=124,s=62$)
    }
    \label{fig:diving_ddgfactgscale}
\end{figure}


\FloatBarrier


\section{Outlook}\label{sec:outlook}

First, we would like to highlight the issue described in Remark \ref{rem:fixat1}, namely,  the challenge of  finding a better way
to fix variables at 1, in the context of B\&B for \ref{CGMESP},
as an important topic for exploration. 

The so called ``BQP relaxation'' (developed by \cite{Anstreicher_BQP_entropy}; also see \cite[Section 3.6]{FLbook})
provides, in a more natural way, the first extended-variable convex relaxation for \ref{CMESP}.
So that would seem like another possible starting point for
an extended-variable based convex relaxation for \ref{CGMESP}
that can be explored; but this does not seem to be straightforward.
The relationship between the BQP bound and the ``Boolean quadric polytope'' (see \cite{padberg}) which lent its name, immediately
allows combinatorial cutting planes for that polytope to be
applied to improve the bound provided by the BQP relaxation (see \cite{AnstreicherLee_Masked}
and \cite{Mixing}). We do not know a similar scheme
for getting useful cutting planes for our 
convex relaxations of \ref{CGMESP} (without adding even more variables to the extended formulation).

Regarding the \ref{GNLP-Id} and \ref{GNLP-Id-comp}
relaxations, the ``NLP bound'' (and its complement)
for \ref{CMESP} has other useful parameter choices
(see \cite{AFLW_Using}; also see \cite[Section 3.5]{FLbook}),
so it would be interesting to see those extended to
\ref{CGMESP}.

\cite{Mixing} introduced the idea of ``mixing'' two or more convex-relaxation based bounds, to achieve improvements over the individual bounds. They did this in the context of \ref{CMESP}, but it is
a rather general concept that could be developed 
with any of the convex-optimization bounds for \ref{CGMESP}.


\section*{Acknowldegments} M. Fampa was supported in part by CNPq grant 307167/2022-4.
J. Lee was supported in part by AFOSR grant FA9550-22-1-0172.

This work was initiated at the Reunion Workshop (August 26--30, 2024) of the Institute for Computational and Experimental Research in Mathematics (ICERM) semester program Discrete Optimization: Mathematics, Algorithms, and Computation, supported by the 
National Science Foundation under Grant No. DMS-1929284. 

\bibliographystyle{alpha}
\bibliography{Bib}

@misc{Fatma2026,
      title={From Majorization to Scaling: Advancing Convex Relaxations of Maximum Entropy Sampling Problem}, 
      author={Lingqing Shen and Fatma Kılınç-Karzan},
      year={2026},
      note={\url{https://arxiv.org/abs/2604.10363}}, 
}

@misc{yamagishi2026dualpathfixingstrategyapplication,
      title={The dual-path fixing strategy and its application to the set-covering problem}, 
      author={Paulo Michel F. Yamagishi and Marcia Fampa and Jon Lee},
      year={2026},
      note={To appear in: SEA 2026 (24th Symposium on Experimental Algorithms, Copenhagen). \url{https://arxiv.org/abs/2601.20977}}, 
}

@incollection{Guttorp-Le-Sampson-Zidek1993,
    AUTHOR = {Guttorp, Peter and Le, Nhu D. and Sampson, Paul D. and Zidek,
              James V.},
     TITLE = {Using entropy in the redesign of an environmental monitoring
              network},
 BOOKTITLE = {Multivariate Environmental Statistics},
 publisher = {North-Holland},
    year = {1993},
    VOLUME = {6},
    editor = {Patil, G.P. and Rao, C.R. and Ross, N.P.},
     PAGES = {175--202},
note={\url{https://marciafampa.com/pdf/GLSZ.pdf}},
}

@book {HJBook,
    AUTHOR = {Horn, Roger A. and Johnson, Charles R.},
     TITLE = {Matrix Analysis},
   EDITION = {{F}irst},
 PUBLISHER = {Cambridge University Press, Cambridge},
      YEAR = {1985},
note={\url{https://doi.org/10.1017/CBO9780511810817}},
}

@incollection{SchurBook,
BOOKTITLE = {The {S}chur {C}omplement and its {A}pplications},
title={Basic properties of the {S}chur complement},
    SERIES = {Numerical Methods and Algorithms},
    VOLUME = {4},
author={Roger A. Horn and Fuzhen Zhang},
    EDITOR = {F. Zhang},
 PUBLISHER = {Springer-Verlag, New York},
      YEAR = {2005},
     PAGES = {17--46},
note={\url{https://doi.org/10.1007/b105056}},
}

@article{Kurt_linx,
  author =       {Kurt M. Anstreicher},
  title =        {Efficient Solution of Maximum-Entropy Sampling Problems},
  journal =  {Operations Research},
  year =         {2020},
  volume = {68},
pages = {1826--1835},
note={\url{https://doi.org/10.1287/opre.2019.1962}},
}

@Article{Anstreicher_BQP_entropy,
  author={Kurt M. Anstreicher},
  title={{Maximum-entropy sampling and the Boolean quadric polytope}},
  journal={Journal of Global Optimization},
  year=2018,
  volume={72},
  pages={603--618},
note={\url{https://doi.org/10.1007/s10898-018-0662-x}},
}

@Book{FLbook,
author="Fampa, Marcia
and Lee, Jon",
Title="Maximum-Entropy Sampling: Algorithms and Application",
year="2022",
publisher="Springer",
note={\url{https://doi.org/10.1007/978-3-031-13078-6}},
}

@article{LeeLind2020,
author = {Lee, Jon and Lind, Joy},
title = {Generalized maximum-entropy sampling},
journal = {INFOR: Information Systems and Operations Research},
volume = {58},
number = {2},
pages = {168--181},
year  = {2020},
note={\url{https://doi.org/10.1080/03155986.2018.1533774}},
}

@PHDTHESIS{WilliamsPhD,
  author =       {Joy Denise Williams},
  title =        {Spectral Bounds for Entropy Models},
  school =       {University of Kentucky},
  year =         {1998},
  type =         {{Ph.D.}},
   month =        {April},
note={\url{https://saalck-uky.primo.exlibrisgroup.com/permalink/01SAA_UKY/15remem/alma9914832986802636}}
}

@article{GMESP_Alg,
      title={Convex relaxation for the generalized maximum-entropy sampling problem}, 
      author={Gabriel Ponte and Marcia Fampa and Jon Lee},
      journal={Algorithmica},
      volume={88, Article 22}, 
      year={2026},
      note={\url{https://doi.org/10.1007/s00453-026-01373-9}}, 
}

@InProceedings{SEA_proceedings,
  author =	{Ponte, Gabriel and Fampa, Marcia and Lee, Jon},
  title =	{Convex Relaxation for the Generalized Maximum-Entropy Sampling Problem},
  booktitle =	{Proceedings of SEA 2024 (22nd International Symposium on Experimental Algorithms)},
  pages =	{25:1--25:14},
  series =	{Leibniz International Proceedings in Informatics (LIPIcs)},
  year =	{2024},
  volume =	{301},
  editor =	{Liberti, L.},
  note={\url{https://doi.org/10.4230/LIPIcs.SEA.2024.25}},
}

@article{gscale,
author = {Chen, Zhongzhu and Fampa, Marcia and Lee, Jon},
year = {2024},
title = {Generalized scaling for the constrained maximum-entropy sampling problem},
journal = {Mathematical Programming},
volume={212},
pages={177--216},
note={\url{https://doi.org/10.1007/s10107-024-02101-3}},
}

@article {BurerLee,
    AUTHOR = {Burer, Samuel and Lee, Jon},
     TITLE = {Solving maximum-entropy sampling problems using factored masks},
  JOURNAL = {Mathematical Programming},
    VOLUME = {109},
      YEAR = {2007},
     PAGES = {263--281},
note={\url{https://doi.org/10.1007/s10107-006-0024-1}},
}

@inproceedings{AnstreicherLee_Masked,
    AUTHOR = {Anstreicher, Kurt M. and Lee, Jon},
     TITLE = {A masked spectral bound for maximum-entropy sampling},
 BOOKTITLE = {Proceedings of: m{OD}a 7},
    SERIES = {Contributions to Statistics},
    editor={A. Di Bucchianico and H. Läuter and H.P. Wynn},
     PAGES = {1--12},
 PUBLISHER = {Physica, Heidelberg},
      YEAR = {2004},
note={\url{https://doi.org/10.1007/978-3-7908-2693-7_1}},
}

@ARTICLE{LeeWilliamsILP,
  author =       {Lee, Jon and Williams, Joy},
  title =        {A linear integer programming bound for maximum-entropy sampling},
  journal =      {Mathematical Programming},
  year =         {2003},
  volume =       {94},
  pages =        {247--256},
note={\url{https://doi.org/10.1007/s10107-002-0318-x}},
}

@inproceedings {HLW,
    AUTHOR = {Hoffman, Alan and Lee, Jon and Williams, Joy},
     TITLE = {New upper bounds for maximum-entropy sampling},
 BOOKTITLE = {Proceedings of m{OD}a 6},
     PAGES = {143--153},
 PUBLISHER = {Physica, Heidelberg},
      YEAR = {2001},
      editor={A.C. Atkinson and P. Hackl and W.G. Müller},
note={\url{https://doi.org/10.1007/978-3-642-57576-1_16}},
}

@article {AFLW_Using,
    AUTHOR = {Anstreicher, Kurt M. and Fampa, Marcia and Lee, Jon and
              Williams, Joy},
     TITLE = {Using continuous nonlinear relaxations to solve constrained
              maximum-entropy sampling problems},
JOURNAL = {Mathematical Programming},
    VOLUME = {85},
      YEAR = {1999},
     PAGES = {221--240},
note={\url{https://doi.org/10.1007/s101070050055}},
}

@article{ChenFampaLeeLinxGaps,
  author =       {Zhongzhu Chen  and Marcia Fampa and  Jon Lee},
  title =        {Masking {A}nstreicher's linx bound for improved entropy bounds},
  year =         {2022},
  journal =      {Operations Research},
  note= {\url{https://doi.org/10.1287/opre.2022.2324}},
  volume={72},
  number={2},
  pages={591--603},
}

@inproceedings{Nikolov,
author = {Nikolov, Aleksandar},
title = {Randomized Rounding for the Largest Simplex Problem},
year = {2015},
editor={R. Rubinfeld},
publisher = {Association for Computing Machinery},
booktitle = {Proceedings of STOC 2015 (Forty-Seventh Annual ACM Symposium on Theory of Computing)},
pages = {861--870},
note={\url{https://doi.org/10.1145/2746539.2746628}},
}

@article{Weijun,
author =       {Yongchun Li and Weijun Xie},
title =        {Best principal submatrix selection for the maximum entropy sampling problem: scalable algorithms and performance guarantees},
journal={Operations Research},
year =         {2023},
volume={72},
number={2},
pages={493--513},
note   =    {\url{https://doi.org/10.1287/opre.2023.2488}},
}

@article{LeeConstrained,
 author = {Lee, Jon},
 title = {Constrained Maximum-Entropy Sampling},
 journal = {Operations Research},
 volume = {46},
 year = {1998},
 pages = {655--664},
note={\url{https://doi.org/10.1287/opre.46.5.655}},
}

@article{Mixing,
  author =       {Zhongzhu Chen and Marcia Fampa and Am\'elie Lambert and Jon Lee},
  title =        {Mixing convex-optimization bounds for maximum-entropy sampling},
  year =         {2021},
  volume= {188},
  pages = {539--568},
  journal = {Mathematical Programming},
note={\url{https://doi.org//10.1007/s10107-020-01588-w}},
}

@article{FactPaper,
  author =       {Zhongzhu Chen and Marcia Fampa and Jon Lee},
  title =        {On computing with some convex relaxations for the maximum-entropy sampling problem},
  year =         {2023},
volume ={35},
number ={2},
pages={368--385},
  journal =        {INFORMS Journal on Computing},
note={\url{https://doi.org/10.1287/ijoc.2022.1264}},
}

@article{FL_update,
author =       {Marcia Fampa and Jon Lee},
title = {Recent advances in maximum-entropy sampling},
journal = {Kuwait Journal of Science},
volume = {53},
number = {1},
pages = {100527},
year = {2026},
note={\url{https://doi.org/10.1016/j.kjs.2025.100527}},
}

@article {KLQ,
    AUTHOR = {Ko, Chun-Wa and Lee, Jon and Queyranne, Maurice},
     TITLE = {An exact algorithm for maximum entropy sampling},
JOURNAL = {Operations Research},
    VOLUME = {43},
      YEAR = {1995},
     PAGES = {684--691},
note={\url{https://doi.org/10.1287/opre.43.4.684}},
}

@article{PonteFampaLeeMPB,
  title={Branch-and-bound for integer {D}-optimality with fast local search and variable-bound tightening},
  author={Ponte, Gabriel and Fampa, Marcia and Lee, Jon},
  journal={Mathematical Programming},
  year={2025},
  publisher={Springer},
note = {\url{https://doi.org/10.1007/s10107-025-02196-2}}
}

@article{padberg,
author = {Padberg, Manfred},
title = {The {B}oolean quadric polytope: Some characteristics, facets and relatives},
year = {1989},
volume = {45},
number = {1--3},
journal = {Mathematical Programming},
pages = {139--172},
note={\url{https://doi.org/10.1007/BF01589101}},
}

@inproceedings{li2025augmented,
  title={The augmented factorization bound for maximum-entropy sampling},
  author={Li, Yongchun},
  booktitle={International Conference on Integer Programming and Combinatorial Optimization},
  pages={412--426},
  year={2025},
  organization={Springer},
  editor={N. Megow and A.  Basu},
note={\url{https://doi.org/10.1007/978-3-031-93112-3_30}},
}

@misc{PFLXadmm,
  title={{ADMM} for 0/1 {D-Opt} and {MESP} relaxations},
  author={Ponte, Gabriel and Fampa, Marcia and Lee, Jon and Xu, Luze},
  year={2024},
note={\url{https://arxiv.org/abs/2411.03461}}
}

@inproceedings{hugedoptmohit,
year={2025},
author = {Aditya Pillai and Gabriel Ponte and Marcia Fampa and Jon Lee and Mohit Singh and Weijun Xie},
title = {Computing Experiment-Constrained {D}-Optimal Designs},
booktitle={Proceedings of the 2025 Conference on Applied and Computational Discrete Algorithms (ACDA)},
pages = {182--195},
publisher={SIAM},
editor={A. Conway and A. Pothen and M. Farach-Colton and B. Ucar},
note={\url{https://doi.org/10.1137/1.9781611979084.14}},
}

@article{papandreou2007bit,
  title={Bit and power allocation in constrained multicarrier systems: The single-user case},
  author={Papandreou, Nikolaos and Antonakopoulos, Theodore},
  journal={EURASIP Journal on Advances in Signal Processing},
  volume={2008},
  pages={643081},
  year={2007},
  publisher={Springer},
note={\url{https://doi.org/10.1155/2008/643081}},
}

@Incollection{KNITRO,
author="Byrd, Richard H.
and Nocedal, Jorge
and Waltz, Richard A.",
editor="Di Pillo, G.
and Roma, M.",
title="{KNITRO}: An Integrated Package for Nonlinear Optimization",
booktitle="Large-Scale Nonlinear Optimization",
year="2006",
publisher="Springer",
pages="35--59",
note={\url{https://doi.org/10.1007/0-387-30065-1_4}},
}

@article{PCA,
author = {Jolliffe, Ian and Cadima, Jorge},
year = {2016},
month = {04},
pages = {20150202},
title = {Principal component analysis: A review and recent developments},
volume = {374},
journal = {Philosophical Transactions of the Royal Society A: Mathematical, Physical and Engineering Sciences},
note={\url{https://doi.org/10.1098/rsta.2015.0202}},
}

@article{Johnson,
title = {Progress in Linear Programming-Based Algorithms for Integer Programming: An Exposition},
author={Ellis L. Johnson and George L. Nemhauser and Martin W.P. Savelsbergh},
year = {2000},
volume = {12},
number = {1},
journal = {INFORMS Journal on Computing},
pages = {2--23},
note={\url{https://doi.org/10.1287/ijoc.12.1.2.11900}},
}

@misc{MESP2DOPT,
      title={On the relationship between {MESP} and 0/1 {D-Opt} and their upper bounds}, 
      author={Gabriel Ponte and Marcia Fampa and Jon Lee},
      year={2025},
note={\url{https://arxiv.org/abs/2511.04350}}
}

@book{boyd2004convex,
  title={Convex Optimization},
  author={Boyd, Stephen and Vandenberghe, Lieven},
  year={2004},
  publisher={Cambridge University Press},
note={\url{https://doi.org/10.1017/CBO9780511804441}},
}

@book{bertsekas2009convex,
  title={Convex Optimization Theory},
  author={Bertsekas, Dimitri},
  volume={1},
  year={2009},
  publisher={Athena Scientific},
  note={\url{https://web.mit.edu/dimitrib/www/Convex_Theory_Entire_Book.pdf}},
}

@misc{AugNLP_MERSP,
      title={The augmented {NLP} bound for maximum-entropy remote sampling}, 
      author={Gabriel Ponte and Marcia Fampa and Jon Lee},
      year={2026},
      note={\url{https://arxiv.org/abs/2601.20970}}, 
}

@book {HJBook2,
    AUTHOR = {Horn, Roger A. and Johnson, Charles R.},
     TITLE = {Topics in Matrix Analysis},
      NOTE = {\url{https://doi.org/10.1017/CBO9780511840371}},
 PUBLISHER = {Cambridge University Press, Cambridge},
      YEAR = {1994},
     PAGES = {viii+607},
      ISBN = {0-521-46713-6},
}

@article{coey2022solving,
    title={Solving natural conic formulations with {H}ypatia.jl},
    author={Chris Coey and Lea Kapelevich and Juan Pablo Vielma},
    year={2022},
    journal={INFORMS Journal on Computing},
    publisher={INFORMS},
    volume={34},
    number={5},
    pages={2686--2699},
    note={\url{https://doi.org/10.1287/ijoc.2022.1202}},
}

@book{hardy1952inequalities,
  title={Inequalities},
  author={Hardy, Godfrey Harold and Littlewood, John Edensor and P{\'o}lya, George},
  year={1952},
  publisher={Cambridge University Press},
  note={\url{https://mathematicalolympiads.wordpress.com/wp-content/uploads/2012/08/inequalities-hardy-littlewood-polya.pdf}},
}


\appendix\normalsize

\section{Lagrangian duals}\label{app:a}

\subsection{Proof of Theorem \ref{thm:dual-glinx}}\label{app:proof_dual_glinx}
Before proving Theorem \ref{thm:dual-glinx}, we have the following lemma.

\begin{lemma}\label{lem:appendix_sup_Xsym}
    For a symmetric matrix $B \in \mathbb{S}^n$ and a vector $\eta\in \mathbb{R}^n_{+}$\,,  we have
     \begin{align*}
        \sup_{X\in\mathbb{S}^n}(B\bullet X - \textstyle\sum_{i\in N} \eta_i \|X_{i\cdot}\|_2)\qquad\qquad\qquad\qquad\qquad\qquad\qquad\qquad\qquad\qquad& \\
        = \begin{cases}
        0,\quad &\exists W\in \mathbb{R}^{n\times n}\text{ such that } \frac{1}{2}(W+W^\top)=B,~\|W_{i\cdot}\|_2 \leq \eta_i\,,~i \in N\,,\\
        +\infty,&\text{otherwise.}
    \end{cases}
    \end{align*}
\end{lemma}

\begin{proof}
    We would like to solve
    \begin{equation}\label{eq:appendix_supXsym_1}
        \sup_{X\in\mathbb{S}^n}(B\bullet X - \textstyle\sum_{i\in N} \eta_i \|X_{i\cdot}\|_2).
    \end{equation}

    Let $W \in \mathbb{R}^{n\times n}$. From  \cite[Example 3.26]{boyd2004convex}, for each $i \in N$, 
    \begin{equation}\label{eq:appendix_eta_i_row_normX_sup}
        \eta_i\|X_{i\cdot}\|_2 = \sup_{\|W_{i\cdot}\|_2 \leq \eta_i} W_{i\cdot}X_{i\cdot}^\top \,.
    \end{equation}
    Define $\mathcal{W}_\eta:= \{W\in \mathbb{R}^{n\times n}\,:\,\|W_{i\cdot}\|_2 \leq \eta_i\,,~i \in N\}$. Because the constraints on $W$ are row-wise, the supremum separates by rows, then
    \[
    \sum_{i \in N}\eta_i\|X_{i\cdot}\|_2 = \sup_{W \in \mathcal{W}_{\eta}} \sum_{i \in N}W_{i\cdot}X_{i\cdot}^\top =\sup_{W \in \mathcal{W}_\eta}W\bullet X.
    \]
    Therefore, for each fixed $X$, we have
    \[
    B\bullet X - \sup_{W \in \mathcal{W}_\eta}W\bullet X = \inf_{W \in \mathcal{W}_\eta} (B-W)\bullet X,
    \]
    and \eqref{eq:appendix_supXsym_1} can be rewritten as 
    \[
    \sup_{X \in \mathbb{S}^n}\inf_{W \in \mathcal{W}_\eta} (B-W)\bullet X.
    \]
    
    Note that for any $X \in \mathbb{S}^n$, we have
    \[
     \textstyle W\bullet X = \tr(W^\top X) = \frac{1}{2}(\tr(W^\top X)+ \tr(WX)) = \frac{1}{2}\tr((W+W^\top)X),
    \]
    so
    \[
   \textstyle (B-W)\bullet X = \left( B - \frac{W+W^\top}{2}\right)\bullet X ,\qquad \forall X \in \mathbb{S}^n.
    \]
    
    \begin{itemize}
        \item If $W \in \mathcal{W}_\eta$ such that $B = \frac{1}{2}(W+W^\top)$, then \eqref{eq:appendix_supXsym_1} is equal to zero.\\ 
        
    Take any $X \in \mathbb{S}^n$, then
    \[
    B\bullet X = \frac{1}{2}(W+W^\top)\bullet X = W \bullet X. 
    \]
    Then for $i \in N$, we have 
    \[
    W_{i\cdot}X_{i\cdot}^\top \leq \|W_{i\cdot}\|_2\|X_{i\cdot}\|_2 \leq \eta_i\|X_{i\cdot}\|_2\,,
    \]
    where the first inequality comes from Cauchy-Schwarz and the last inequality comes from the fact that $W \in \mathcal{W}_\eta$\,.
    then
    \[
    W\bullet X \leq \sum_{i\in N} \eta_i\|X_{i\cdot}\|_2 
    \]
    so we conclude that 
    \[
    B\bullet X - \sum_{i\in N} \eta_i\|X_{i\cdot}\|_2 \leq 0,
    \]
    where equality holds at $X = 0$, so the supremum equals zero.\\
    
    \item If $B \neq \frac{1}{2}(W+W^\top)$, then \eqref{eq:appendix_supXsym_1} is $+\infty$.

    We note that for $i \in N$, we have that $\|W_{i\cdot}\|_2 \leq \eta_i$ is a closed ball and the intersection of closed sets is closed. Also, $\|W\|_F^2 = \sum_{i\in N} \|W_{i\cdot}\|_2^2 \leq \sum_{i \in N}\eta_i^2$\,. Then, we have that $\mathcal{W}_\eta$ is closed and bounded and therefore $\mathcal{W}_\eta$ is a compact set.
    
    Let $\mathcal{S}_\eta := \{\frac{1}{2}(W+W^\top)\,:\, W\in \mathcal{W}_\eta\}$.
    From \cite[Proposition A.2.6(d)]{bertsekas2009convex}, we have that $\mathcal{S}_\eta$ is compact. Also, note that $\mathcal{S}_\eta$ is a convex set as well.

    We have $B \not\in \mathcal{S}_\eta$, then, from \cite[Example 2.20]{boyd2004convex}, we have that there exists a non-zero matrix $\hat{X} \in \mathbb{S}^n$ where 
    \[
    M\bullet \hat{X}\leq \beta,~\forall M \in \mathcal{S}_\eta\,,\quad B\bullet \hat{X} > \beta,
    \]
    for some $\beta \in \mathbb{R}$. Then, we have
    \begin{align*}
        B \bullet \hat{X} > \sup_{ M \in \mathcal{S}_\eta} M\bullet \hat{X} = \sup_{ W \in \mathcal{W}_\eta} \frac{1}{2}(W+W^\top)\bullet \hat{X} = \sup_{ W \in \mathcal{W}_\eta} W\bullet \hat{X} = \sum_{i\in N}\eta_i\|\hat{X}_{i\cdot}\|_2\,,
    \end{align*}
    where the last equation comes from \eqref{eq:appendix_eta_i_row_normX_sup}. Then
    for $X = t\hat{X}$, we have that for $t \to +\infty$, \eqref{eq:appendix_supXsym_1} goes to $+\infty$ as well.
 \qed
\end{itemize}
\renewcommand{\qedsymbol}{}
\end{proof}

\begin{proof}[Proof of Theorem \ref{thm:dual-glinx}]
    Introducing a variable $E \in \mathbb{S}^n_{++}$\, we consider
\begin{equation*}
\begin{array}{rrll}
&\max&\frac{1}{2}\ldet(E)\\
&\text{\rm s.t.} &CXC + I - X - E = 0;\\
&&\mathbf{e}^\top x = s;\\
&&Ax \leq b;\\
&&\tr(X) = t;\\
&&x_i - \|X_{i\cdot}\|_2 \geq 0,~ \forall~ i \in N;\\
&&x\geq 0,~\mathbf{e} - x\geq 0;\\
&&X\succeq 0,~\Diag(x) - X\succeq 0,
\end{array}
\end{equation*}
and consider the Lagrangian function
\begin{align*}
    \!\!\!\!\!\!\!\!\mathcal{L}(E,X,x,\Theta,\upsilon,\nu,\pi,\tau,\xi,Z,\Omega,\eta)\!=\! 
    &\textstyle\frac{1}{2}\ldet(E) +  
    \tr(\Theta) + (C\Theta C)\bullet X 
    - \Theta \bullet (E+X) +  \nu^\top(\mathbf{e}-x) ~+ \\
    &~~ \upsilon^\top x +\pi^\top (b-Ax) + \tau(s-\mathbf{e}^\top x)  + \xi(t-\tr(X)) ~+\\ &~~ \Omega\bullet X + Z\bullet (\Diag(x)-X) +\textstyle\sum_{i = 1}^n\eta_i(x_i - \|X_{i\cdot}\|_2)
\end{align*}
with $\mathrm{dom}\,\mathcal{L} = \mathbb{S}^n_{++}\times \mathbb{S}^n\times \mathbb{R}^n \times \mathbb{S}^n \times \mathbb{R}^n \times \mathbb{R}^n \times \mathbb{R}^m\times \mathbb{R} \times \mathbb{R}\times \mathbb{S}^{n}_{+}\times \mathbb{S}^{n}_{+}\times \mathbb{R}^n$. The corresponding dual function is 
\begin{equation}\label{eq:dual_func_glinx_sup}
    \mathcal{L}^*(\Theta,\upsilon,\nu,\pi,\tau,\xi,Z,\Omega,\eta) := \sup_{E\succ 0,\, x,\, X}\!\! \mathcal{L}(E,X,x,\Theta,\upsilon,\nu,\pi,\tau,\xi,Z,\Omega,\eta).
\end{equation}
Assuming that $\zglinx > -\infty$, we are justified to use minimum in the formulation of the Lagrangian dual problem, rather than infimum, because, in this case, Slater's condition holds for the primal (e.g., for $\hat{X}:= (t/n)I$ and $\hat{x}:=(s/n)\mathbf{e}$), and so the optimal value of the Lagrangian is attained.

We note that Lagrangian function is concave in $x$, $X$ and $E \succ 0$. Then the supremum in \eqref{eq:dual_func_glinx_sup} for $E$ and $x$ follows from  first-order optimality conditions when the supremum is finite, while the supremum over $X$ is given by using Lemma \ref{lem:appendix_sup_Xsym}.

\begin{itemize}
    \item Supremum over $E$. 
    \begin{align*}
        \sup_{E \succ 0}(\textstyle\frac{1}{2}\ldet(E) - \Theta \bullet E) = \begin{cases}
        -\frac{n}{2} - \frac{1}{2}\ldet(2\Theta),\quad &\Theta \succ 0,\\
        +\infty,&\text{otherwise.}
    \end{cases}
    \end{align*}
    \item Supremum over $x \in \mathbb{R}^n$.
    \begin{align*}
        \sup_{x\in\mathbb{R}^n}((\upsilon - \nu - A^\top \pi - \tau\mathbf{e} + \eta + \diag(Z))^\top x) = \begin{cases}
        0,\quad &\upsilon - \nu - A^\top \pi - \tau\mathbf{e} + \eta + \diag(Z) = 0,\\
        +\infty,&\text{otherwise.}
    \end{cases}
    \end{align*}
    \item Supremum over $X \in \mathbb{S}^n$. We apply Lemma \ref{lem:appendix_sup_Xsym}, for the matrix $B:= \Omega-Z + C\Theta C - \Theta - \xi I_n$\,, where $B \in \mathbb{S}^n$. 
\end{itemize}
The result follows.
\end{proof}


\subsection{Proof of Theorem \ref{thm:dual-GNLP-Id}}\label{app:proof_dual_GNLP_Id}
\begin{proof}[Proof of Theorem  \ref{thm:dual-GNLP-Id}]
    Introducing the variable $E \in \mathbb{S}^p_{++}$, we consider  the problem
\begin{equation*}
\begin{array}{rrll}
&\max&\ldet(E)\\[2pt]
&\text{\rm s.t.} &I_p - H^\top X H - E = 0;\\
&&\mathbf{e}^\top x = s;\\
&&Ax \leq b;\\
&&\tr(X) = t;\\
&&x_i - \|X_{i\cdot}\|_2 \geq 0,~ \forall~ i \in N;\\
&&x\geq 0,~\mathbf{e} - x\geq 0;\\
&&X\succeq 0,~\Diag(x) - X\succeq 0,
\end{array}
\end{equation*}
and consider the Lagrangian function
\begin{align*}
    \!\!\!\!\!\!\!\!\mathcal{L}(E,X,x,\Theta,\upsilon,\nu,\pi,\tau,\xi,Z,\Omega,\eta)\!=\! 
    &\ldet(E) +  
    \tr(\Theta) - (H\Theta H^\top)\bullet X 
    - \Theta \bullet E +  \nu^\top(\mathbf{e}-x) ~+ \\
    &~~ \upsilon^\top x +\pi^\top (b-Ax) + \tau(s-\mathbf{e}^\top x)  + \xi(t-\tr(X)) ~+\\ &~~ \Omega\bullet X + Z\bullet (\Diag(x)-X) +\textstyle\sum_{i = 1}^n\eta_i(x_i - \|X_{i\cdot}\|_2)
\end{align*}
with $\mathrm{dom}\,\mathcal{L} = \mathbb{S}^p_{++}\times \mathbb{S}^n\times \mathbb{R}^n \times \mathbb{S}^p \times \mathbb{R}^n \times \mathbb{R}^n \times \mathbb{R}^m\times \mathbb{R} \times \mathbb{R}\times \mathbb{S}^{n}_{+}\times \mathbb{S}^{n}_{+}\times \mathbb{R}^n$. The corresponding dual function is 
\begin{equation}\label{eq:dual_func_op_sup}
    \mathcal{L}^*(\Theta,\upsilon,\nu,\pi,\tau,\xi,Z,\Omega,\eta) := \sup_{E\succ 0,\, x,\, X}\!\! \mathcal{L}(E,X,x,\Theta,\upsilon,\nu,\pi,\tau,\xi,Z,\Omega,\eta).
\end{equation}
Assuming that $\zop > -\infty$, we are justified to use minimum in the formulation of the Lagrangian dual problem, rather than infimum, because, in this case, Slater's condition holds for the primal (e.g., for $\hat{X}:= (t/n)I$ and $\hat{x}:=(s/n)\mathbf{e}$), and so the optimal value of the Lagrangian is attained.

We note that Lagrangian function is concave in $x$, $X$ and $E \succ 0$. Then the supremum in \eqref{eq:dual_func_op_sup} for $E$ and $x$ follows from  first-order optimality conditions when the supremum is finite, while the supremum over $X$ is given by using Lemma \ref{lem:appendix_sup_Xsym}.

\begin{itemize}
    \item Supremum over $E$. 
    \begin{align*}
        \sup_{E \succ 0}(\ldet(E) - \Theta \bullet E) = \begin{cases}
        -p - \ldet(\Theta),\quad &\Theta \succ 0,\\
        +\infty,&\text{otherwise.}
    \end{cases}
    \end{align*}
    \item Supremum over $x \in \mathbb{R}^n$
    \begin{align*}
        \sup_{x\in\mathbb{R}^n}((\upsilon - \nu - A^\top \pi - \tau\mathbf{e} + \eta)^\top x) = \begin{cases}
        0,\quad &\upsilon - \nu - A^\top \pi - \tau\mathbf{e} + \eta = 0,\\
        +\infty,&\text{otherwise.}
    \end{cases}
    \end{align*}
    \item Supremum over $X \in \mathbb{S}^n$. We apply Lemma \ref{lem:appendix_sup_Xsym}, for the matrix $B:= \Omega-Z - H\Theta H^\top - \xi I_n$\,, where $B \in \mathbb{S}^n$. 
\end{itemize}
The result follows.
\end{proof}

\subsection{Proof of Theorem \ref{thm:dual-GNLP-Id-comp}}\label{app:proof_dual_GNLP_Id_comp}
\begin{proof}[Proof of Theorem  \ref{thm:dual-GNLP-Id-comp}]
The result follows by simply adapting the proof of Theorem \ref{thm:dual-GNLP-Id} by replacing $I_p - H^\top X H$ by $I_q + G^\top X G$. 
\end{proof}


\section{Details on the specific instance used in Theorem \ref{thm:family_glinx}}\label{app:sols_family_glinx}
\begin{proof}[Instance used in Theorem  \ref{thm:family_glinx}]
Consider the following example:
    \[
    C:=\begin{pmatrix}
    217 & 220 & 156 & 110 & 106 & 230\\
    220 & 249 & 191 & 96 & 99 & 256\\
    156 & 191 & 154 & 58 & 64 & 194\\
    110 & 96 & 58 & 72 & 66 & 104\\
    106 & 99 & 64 & 66 & 62 & 106\\
    230 & 256 & 194 & 104 & 106 & 264
    \end{pmatrix},~ s:=4,~ t:=3.
    \]
    Consider the rational point $(\hat{x},\hat{X})\in   \Pwedgenst \setminus \Pnst$:
\[
\hat{x}:=\frac{1}{100000}\begin{pmatrix}
    99993\\
    21423\\
    99933\\
    66510\\
    99929\\
    12212
\end{pmatrix},\quad \hat{X}:=\frac{1}{100000}\begin{pmatrix}
    94959 & 5313 & -6586 & 6860 & -8610 & 3873\\
    5313 & 12411 & 13475 & 1108 & 2594 & -5850\\
    -6586 & 13475 & 78035 & -8172 & 2024 & 8974\\
    6860 & 1108 & -8172 & 34884 & 28990 & -252\\
    -8610 & 2594 & 2024 & 28990 & 71825 & 2748\\
    3873 & -5850 & 8974 & -252 & 2748 & 7886
\end{pmatrix},
\]
    and note that 
\begin{itemize}
    \item $\hat{x} \in [0,1]^n$,
    \item $\sum_{i=1}^n \hat{x}_i = 4$,
    \item $\sum_{i=1}^n \hat{X}_{ii} = 3$,
    \item $\lambda_{n}(\hat{X}) \approx 0.000393$,
    \item $\lambda_{n}(\Diag(\hat{x})-\hat{X}) \approx 2.5346\cdot 10^{-5}$,
    \item $~\|\hat{X}_{i\cdot}\|_2 \leq \hat{x}_i\,,~i \in \{1,\dots,n-1\}$, but $~\|\hat{X}_{n\cdot}\|_2 \geq 0.141265 > 0.1222 \geq \hat{x}_n$\,.
\end{itemize}
We have 
\begin{align*}
    C\hat{X}C + I - \hat{X} &= \begin{pmatrix}
\frac{1341527117}{12500} & \frac{5687052423}{50000} & \frac{8317460963}{100000} & \frac{2626422719}{50000} & \frac{5192973889}{100000} & \frac{11825322993}{100000} \\[5pt]
\frac{5687052423}{50000} & \frac{6091697029}{50000} & \frac{8982700827}{100000} & \frac{1366820209}{25000} & \frac{679755549}{12500} & \frac{6321248633}{50000} \\[5pt]
\frac{8317460963}{100000} & \frac{8982700827}{100000} & \frac{83307451}{1250} & \frac{985098193}{25000} & \frac{3938504459}{100000} & \frac{4653764917}{50000} \\[5pt]
\frac{2626422719}{50000} & \frac{1366820209}{25000} & \frac{985098193}{25000} & \frac{331475569}{12500} & \frac{1298017539}{50000} & \frac{712923211}{12500} \\[5pt]
\frac{5192973889}{100000} & \frac{679755549}{12500} & \frac{3938504459}{100000} & \frac{1298017539}{50000} & \frac{1274923559}{50000} & \frac{708324747}{12500} \\[5pt]
\frac{11825322993}{100000} & \frac{6321248633}{50000} & \frac{4653764917}{50000} & \frac{712923211}{12500} & \frac{708324747}{12500} & \frac{6561813519}{50000}
\end{pmatrix}.
\end{align*}
Let  
\[
\hat{L}:= \left(
\begin{smallmatrix}
1 & 0 & 0 & 0 & 0 & 0 \\[5pt]
\frac{5687052423}{5366108468} & 1 & 0 & 0 & 0 & 0 \\[5pt]
\frac{1188208709}{1533173848} & \frac{900310450438239687}{692283499714341286} & 1 & 0 & 0 & 0 \\[5pt]
\frac{2626422719}{5366108468} & -\frac{267592692414338513}{346141749857170643} & -\frac{591912634178085623519968}{260446043245878019400653} & 1 & 0 & 0 \\[5pt]
\frac{5192973889}{10732216936} & -\frac{351578676742071591}{692283499714341286} & -\frac{581934415229849258748264}{260446043245878019400653} & \frac{466611097417659357896625128915}{461755657789801734359843316874} & 1 & 0 \\[5pt]
\frac{11825322993}{10732216936} & \frac{589779855651186449}{692283499714341286} & -\frac{15689651694517513631476}{260446043245878019400653} & \frac{114325227305199597422110268483}{461755657789801734359843316874} & \frac{6558844605110015549951209798555933}{39718589495834992577088455301412726} & 1
\end{smallmatrix}\right)
\]
and 
\[
\hat{d} :=
\left(\begin{smallmatrix}
\frac{1341527117}{12500}\\[5pt]
\frac{346141749857170643}{268305423400000}\\[5pt]
\frac{260446043245878019400653}{69228349971434128600000}\\[5pt]
\frac{230877828894900867179921658437}{13022302162293900970032650000}\\[5pt]
\frac{19859294747917496288544227650706363}{11543891444745043358996082921850000}\\[5pt]
\frac{4478246380267821269292745691159420377611}{3971858949583499257708845530141272600000}
\end{smallmatrix}\right),
\]
then, we have $C\hat{X}C + I - \hat{X} = \hat{L}\Diag(\hat{d})\hat{L}^\top$, and therefore
\[
\frac{1}{2}\ldet(C\hat{X}C + I - \hat{X}) = \frac{1}{2}\sum_{i=1}^6 \log(\hat{d}_i)\approx 11.80439587.
\]

Next, consider the feasible solution for \ref{eq:dual_glinx}: $\hat{\nu}:= (131673/312500,\; 0,\; 562609/10000000,\;\allowbreak 19/312500,\;\allowbreak 58013/1000000,\; 0)^\top$,
$\hat{\upsilon}:= (0,\; 27/10000000,\; 0,\; 0,\; 0,\; 607/10000000)^\top$,
$\hat\eta:= (0,\; 231/25000,\; 0,\; 0,\; 0,\;\allowbreak 7039/100000)^\top$,
$\hat\tau:=5963/10000$, $\hat{\xi}:= -3170969/5000000$,
\begin{equation*}
    \begin{array}{ll}
       \hat\Theta := 10^{-7}
{\begin{pmatrix}
907052 ~&~ -891279 ~&~ 1276501 ~&~ -677235 ~&~ 926199 ~&~ -969631 \\[2pt]
-891279 ~&~ 4302195 ~&~ -2344586 ~&~ 70790 ~&~ -298866 ~&~ -1580259 \\[2pt]
1276501 ~&~ -2344586 ~&~ 3874474 ~&~ 1002743 ~&~ 172454 ~&~ -2149685 \\[2pt]
-677235 ~&~ 70790 ~&~ 1002743 ~&~ 3288708 ~&~ -2888223 ~&~ -351266 \\[2pt]
926199 ~&~ -298866 ~&~ 172454 ~&~ -2888223 ~&~ 3034885 ~&~ -724204 \\[2pt]
-969631 ~&~ -1580259 ~&~ -2149685 ~&~ -351266 ~&~ -724204 ~&~ 4386025
\end{pmatrix}}, \\[20pt]
       \hat{Z}:= 10^{-7}
{\begin{pmatrix}
10176536 ~&~ 2078316 ~&~ -1511788 ~&~ 1343043 ~&~ -1434442 ~&~ 2136729 \\[2pt]
2078316 ~&~ 5870573 ~&~ 5419707 ~&~ -51650 ~&~ 797717 ~&~ 5485309 \\[2pt]
-1511788 ~&~ 5419707 ~&~ 6525609 ~&~ -1797065 ~&~ 68161 ~&~ 4816399 \\[2pt]
1343043 ~&~ -51650 ~&~ -1797065 ~&~ 5963608 ~&~ 5687307 ~&~ 813706 \\[2pt]
-1434442 ~&~ 797717 ~&~ 68161 ~&~ 5687307 ~&~ 6543130 ~&~ 1570388 \\[2pt]
2136729 ~&~ 5485309 ~&~ 4816399 ~&~ 813706 ~&~ 1570388 ~&~ 5258493
\end{pmatrix}}, \\[20pt]
\hat{\Omega}:= 10^{-6}
\begin{pmatrix}
957 ~&~ -7721 ~&~ 2255 ~&~ 46 ~&~ 625 ~&~ -7477 \\[2pt]
-7721 ~&~ 62480 ~&~ -18244 ~&~ -376 ~&~ -5053 ~&~ 60493 \\[2pt]
2255 ~&~ -18244 ~&~ 5331 ~&~ 110 ~&~ 1475 ~&~ -17666 \\[2pt]
46 ~&~ -376 ~&~ 110 ~&~ 11 ~&~ 27 ~&~ -364 \\[2pt]
625 ~&~ -5053 ~&~ 1475 ~&~ 27 ~&~ 413 ~&~ -4893 \\[2pt]
-7477 ~&~ 60493 ~&~ -17666 ~&~ -364 ~&~ -4893 ~&~ 58580
\end{pmatrix},  \\[20pt]
\hat{W}:= 10^{-4}
\begin{pmatrix}
0 ~&~ 0 ~&~ 0 ~&~ 0 ~&~ 0 ~&~ 0 \\[2pt]
25 ~&~ 53 ~&~ 63 ~&~ 5 ~&~ 12 ~&~ -31 \\[2pt]
0 ~&~ 0 ~&~ 0 ~&~ 0 ~&~ 0 ~&~ 0 \\[2pt]
0 ~&~ 0 ~&~ 0 ~&~ 0 ~&~ 0 ~&~ 0 \\[2pt]
0 ~&~ 0 ~&~ 0 ~&~ 0 ~&~ 0 ~&~ 0 \\[2pt]
194 ~&~ -241 ~&~ 451 ~&~ -1 ~&~ 135 ~&~ 422
\end{pmatrix}.
    \end{array}
\end{equation*}
We have $\lambda_n(\hat{\Theta}) \approx 1.04\cdot 10^{-6}$, $\lambda_n(\hat{Z}) \approx 4.03\cdot 10^{-6}$ and $\lambda_n(\hat{\Omega}) \approx 2.12\cdot 10^{-6}$. Also, clearly $\hat\upsilon \geq 0$, $\hat\nu \geq 0$ and $\hat\eta \geq 0$. We can check that 
\begin{itemize}
    \item $\hat\upsilon - \hat\nu - \hat\tau\mathbf{e} + \hat\eta + \diag(\hat Z)= 0$
    \item $\frac{1}{2}(\hat{W}+\hat{W}^\top) - \hat\Omega +\hat Z - C\hat\Theta C + \hat\Theta + \hat\xi I_n = 0$
    \item $\|\hat{W}_{i\cdot}\|_2\leq \hat\eta_i\,,~ \forall~ i \in N$.
\end{itemize}
Let
\[
\hat{L}_\theta:= \left(
\begin{smallmatrix}
1 & 0 & 0 & 0 & 0 & 0 \\[5pt]
-\frac{891279}{907052} & 1 & 0 & 0 & 0 & 0 \\[5pt]
\frac{1276501}{907052} & -\frac{988942885693}{3107936323299} & 1 & 0 & 0 & 0 \\[5pt]
-\frac{677235}{907052} & -\frac{539395122485}{3107936323299} & \frac{5490473568355330862}{5380203476068630927} & 1 & 0 & 0 \\[5pt]
\frac{926199}{907052} & \frac{184804905163}{1035978774433} & -\frac{2910584700482330391}{5380203476068630927} & -\frac{6106081299522798751033888}{4718705572918752518333685} & 1 & 0 \\[5pt]
-\frac{969631}{907052} & -\frac{765862944839}{1035978774433} & -\frac{4945111681205475471}{5380203476068630927} & \frac{585864798084689586866776}{4718705572918752518333685} & -\frac{9928885245318851640114519161}{22985951912735606662851035783} & 1
\end{smallmatrix}\right)
\]
and 
\[
\hat{d}^\theta :=
\left(\begin{smallmatrix}
\frac{226763}{1250000}\\[5pt]
\frac{3107936323299}{4535260000000}\\[5pt]
\frac{5380203476068630927}{15539681616495000000}\\[5pt]
\frac{943741114583750503666737}{5380203476068630927000000}\\[5pt]
\frac{22985951912735606662851035783}{23593527864593762591668425000000}\\[5pt]
\frac{434120922176907067771605323383}{57464879781839016657127589457500000}
\end{smallmatrix}\right),
\]
then, we have $2\hat\Theta = \hat{L}_{\theta}\Diag(\hat{d}^{\theta})\hat{L}_{\theta}^\top$, and therefore $\frac{1}{2}\ldet(2\hat\Theta) = \frac{1}{2}\sum_{i=1}^6 \log(\hat{d}^{\theta}_i)$\,.

Finally, the objective value of \ref{eq:dual_glinx} at  $(\hat\Theta,\hat\upsilon,\hat\nu,\hat\pi,\hat\tau,\hat\xi,\hat Z,\hat \Omega,\hat\eta)$ is
\[
-\frac{1}{2}\ldet(2\hat\Theta) + \tr(\hat\Theta) + {\hat\nu}^\top \mathbf{e} + \hat\pi^\top b + \hat\tau s + \hat\xi t - \frac{n}{2} = -\frac{1}{2}\sum_{i=1}^6 \log(\hat{d}^{\theta}_i) -\frac{2949}{1250000} \approx 11.80435231.
\]
\end{proof}


\section{Constructing dual-feasible solutions for DDGFact}\label{app:construct_dgfact}
The Lagrangian dual of \ref{ddgfact} is
\[
\begin{array}{ll}
\multicolumn{2}{l}{z_{{\mbox{\tiny DGFact}}}(C,s,t,A,b;F)=} \\	
\qquad \min& -
 \sum_{\ell=k-t+1}^k 
\log\left(\lambda_{\ell} \left(\Theta\right)\right)
+ \nu^\top \mathbf{e}  + \pi^\top b +\tau s - t\\[4pt]
\qquad \mbox{s.t.}
     & \diag(F \Theta F^\top) + \upsilon - \nu  - A^\top \pi - \tau\mathbf{e}=0,\\[4pt]
& \Theta\succ 0, ~\upsilon\geq 0, ~\nu\geq 0, ~\pi\geq 0.
\end{array}
\tag{DGFact}\label{DGFact}
\]

To certify a valid upper bound for  GMESP and avoid drawing incorrect conclusions from \emph{near}-optimal solutions, following  \cite{Weijun}, we will show  how to rigorously construct a  feasible solution for \ref{DGFact} from  a feasible solution $\hat x$ of \ref{ddgfact}  with finite objective value, with the goal of producing a small gap.

Although in \ref{CGMESP},  the lower bounds are all zero  and the upper bounds are all one, we will consider the more general problem with lower and upper bounds on the variables given respectively by $l,c\in\{0,1\}^n$,
with $l\leq c$. 
So, we consider the constraints $l\leq\! x\!\leq c$ in \ref{ddgfact}, instead of $0\leq\! x\!\leq \mathbf{e}$. The motivation for this, is to derive the technique to fix variables at any subproblem considered during the execution of the B\&B algorithm, when some of the variables may already be fixed (i.e., $l_i=c_i$ fixes $x_i$). Instead of redefining the problem with fewer variables in our numerical experiments, we found that it was more efficient to change the upper bound $c_i$ from one to zero, when variable $i$ is fixed at zero in a subproblem, and similarly,  change the lower bound $l_i$ from zero to one, when variable $i$ is fixed at one. 

Then, to construct the dual solution, we  consider a feasible solution $\hat x$ of \ref{ddgfact}, with the constraints $0\leq x\leq \mathbf{e}$  replaced by $l\leq x\leq c$, and the spectral
decomposition $F(\hat{x})=\sum_{\ell=1}^{k} \hat \lambda_\ell \hat u_\ell \hat u_\ell^\top\,,$
with $\hat \lambda_1\geq\hat \lambda_2\geq\cdots\geq \hat \lambda_{\hat r}>\hat \lambda_{\hat{r}+1}=\cdots=\hat \lambda_k=0$.  Notice that $\rank(F(\hat x))= \hat{r}\geq t$.
Following \cite{Nikolov}, we  set 
$\hat{\Theta}:=\sum_{\ell=1}^{k} {\hat \beta}_\ell \hat{u}_\ell \hat{u}_\ell^\top$\,,
where
\begin{equation*}
\hat{\beta}_\ell:=\left\{
\begin{array}{ll}
        \textstyle 1/\hat{\lambda}_\ell\,,
       &~1\leq \ell\leq \hat{\iota};\\
     1/\hat{\delta},&~\hat{\iota}<\ell\leq \hat{r};\\
     (1+\epsilon)/\hat{\delta},&~\hat{r}<\ell\leq k,
\end{array}\right.
\end{equation*}
 for any $\epsilon>0$, where $\hat{\iota}$ is the unique integer defined  in Lemma \ref{Ni13} for $\lambda_\ell=\hat{\lambda}_\ell$\,, and
$
\hat \delta:=\frac{1}{t-\hat \iota}\sum_{\ell=\hat \iota+1}^{k}\hat \lambda_\ell
$\thinspace.
From Lemma \ref{Ni13}, we have that $\hat\iota<t$.   Then, 
\begin{equation*} 
\textstyle
- \sum_{\ell=1}^t \log(\hat{\beta}_{\ell})=  \sum_{\ell=1}^{\hat{\iota}} \log(\hat{\lambda}_{\ell}) + (t-\hat{\iota})\log(\hat{\delta})= \Gamma_t(F(\hat{x})).
\end{equation*}
Therefore, the minimum duality gap between $\hat x$ in the modified \ref{ddgfact} and feasible solutions
of its dual problem  of the form $(\hat\Theta,\upsilon,\nu,\pi,\tau)$,
is the optimal value of
\begin{equation}\label{mingapproba}\tag{$G(\hat\Theta)$}
\begin{array}{ll}
\min&~
    -\upsilon^\top l+ \nu^\top c   + \pi^\top b +\tau s - t\\
      \mbox{s.t.} 
&     ~ \upsilon - \nu  - A^\top \pi - \tau\mathbf{e}= - \diag(F \hat \Theta F^\top) ,\\
&~\upsilon\geq 0, ~\nu\geq 0, ~\pi\geq 0.
\end{array}
\end{equation}
We note that \ref{mingapproba} is always feasible (e.g., 
$\upsilon:=0$, $\nu:=\diag(F \hat \Theta F^\top)$, $\pi:=0$, $\tau:=0$ is a feasible solution). 

Also,   \ref{mingapproba} has a simple closed-form solution for $\GMESP$, that is when there are no $Ax \leq b$ constraints.
 To construct this  optimal solution to \ref{mingapproba}, 
we consider the permutation $\sigma$ of the indices in $N$, such that $\diag(F\hat\Theta F^\top)_{\sigma(1)} \geq \dots \geq \diag(F\hat\Theta F ^\top)_{\sigma(n)}$\,.
If $c_{\sigma(1)}+\sum_{i=2}^n l_{\sigma(i)}>s$, we let $\varphi:=0$; otherwise we let $\varphi:=\max\{j\in N: \sum_{i=1}^j c_{\sigma(i)} +  \sum_{i=j+1}^n  l_{\sigma(i)}\leq s\}$. 
We define $P := \{\sigma(1),\dots,\sigma(\varphi)\}$ and $Q :=\{\sigma(\varphi + 2),\dots,\sigma(n)\}$.
Then, to obtain an optimal solution of  \ref{mingapproba},   we consider its  optimality conditions 
\begin{equation}\label{kktgfact1f}
\begin{array}{l}
    \diag(F\hat \Theta F^\top)  + \upsilon - \nu - \tau\mathbf{e} = 0,~\upsilon\geq 0,~\nu\geq 0,\\
    \mathbf{e}^\top x = s,~l\leq x \leq c,\\
     - \upsilon^\top l + \nu^\top c + \tau s = \diag(F\hat \Theta F^\top)^\top x.
\end{array}
\end{equation}
We  can verify that the following solution satisfies \eqref{kktgfact1f}.
\begin{align*}
    &x^*_\ell:= \begin{cases}
        c_\ell\,,\quad&\text{for }\ell \in P;\\
        l_\ell\,,&\text{for }\ell \in Q;\\
        s-\sum_{\ell\in P}c_\ell -\sum_{\ell\in Q}l_\ell\,,&\text{for }\ell = \sigma(\varphi +1), ~ \text{if  }\varphi<n,
    \end{cases}\\ 
    &\tau^* :=  \begin{cases}
    \diag(F\hat\Theta F^\top)_{\sigma(\varphi+1)}\,,\quad &\text{if  }\varphi<n;\\
    0,&\text{otherwise},
    \end{cases}\\
    &\nu^*_\ell := \begin{cases}
         \diag(F\hat\Theta F^\top)_{\ell} - \tau^*,\quad &\text{for }\ell \in P;\\
         0,&\text{otherwise},
    \end{cases}\\
    &\upsilon^*_\ell := \begin{cases}
         \tau^*-\diag(F\hat\Theta F^\top)_{\ell}\,,\quad &\text{for } \ell \in Q;\\
         0,&\text{otherwise}.
    \end{cases}
\end{align*}

\end{document}